\documentclass[12pt]{amsart}
\usepackage{amscd,amssymb,graphics,color,a4wide,hyperref,verbatim}

\usepackage{multicol, fullpage} 
\usepackage[usenames,dvipsnames]{xcolor} 
\usepackage{tikz, tikz-3dplot, pgfplots}
\usepackage{tikz-cd}
\usetikzlibrary[positioning,patterns] 
\usetikzlibrary{matrix,arrows,decorations.pathmorphing}

\usepackage{graphicx}
\usepackage{psfrag}
\usepackage{marvosym}
\usepackage{amsmath}
\usepackage{amsfonts}
\usepackage{amssymb}
\usepackage{amsthm}
\usepackage{mathrsfs}
\usepackage{enumitem}

\usepackage{ytableau}

\input xy
\xyoption{all}

\usepackage{calligra}
\usepackage{mathrsfs}
\DeclareMathAlphabet{\mathcalligra}{T1}{calligra}{m}{n}
\DeclareFontShape{T1}{calligra}{m}{n}{<->s*[1.5]callig15}{}

\newtheorem{theorem}{Theorem}[section]
\newtheorem*{theoremstar}{Theorem}
\newtheorem{lemma}[theorem]{Lemma}
\newtheorem{proposition}[theorem]{Proposition}
\newtheorem{corollary}[theorem]{Corollary}

\theoremstyle{definition}
\newtheorem{definition}[theorem]{Definition}

\newtheorem{example}[theorem]{Example}

\newtheorem{remark}[theorem]{Remark}
\newtheorem{theorem-definition}[theorem]{Theorem-Definition}

\numberwithin{equation}{section}

\newcommand{\CC} {\mathbb{C}}
\newcommand{\DD} {\mathbb{D}}

\newcommand{\LL} {\mathbb{L}}

\newcommand{\NN} {\mathbb{N}}

\newcommand{\PP} {\mathbb{P}}
\newcommand{\QQ} {\mathbb{Q}}

\newcommand{\ZZ} {\mathbb{Z}}

\newcommand {\shA} {\mathcal{A}}

\newcommand {\shE} {\mathcal{E}}
\newcommand {\shF} {\mathcal{F}}

\newcommand {\shH} {\mathcal{H}}

\newcommand {\shK} {\mathcal{K}}

\newcommand {\shM} {\mathcal{M}}

\newcommand {\shS} {\mathcal{S}}
\newcommand {\shT} {\mathcal{T}}
\newcommand {\shP} {\mathcal{P}}

\newcommand {\sC} {\mathscr{C}}

\newcommand {\sE} {\mathscr{E}}
\newcommand {\sF} {\mathscr{F}}
\newcommand {\sG} {\mathscr{G}}

\newcommand {\sI} {\mathscr{I}}

\newcommand {\sK} {\mathscr{K}}
\newcommand {\sL} {\mathscr{L}}
\newcommand {\sM} {\mathscr{M}}
\newcommand {\sN} {\mathscr{N}}
\newcommand {\sO} {\mathscr{O}}

\newcommand {\sV} {\mathscr{V}}
\newcommand {\sW} {\mathscr{W}}

\newcommand {\foh}  {\mathfrak{h}}

\newcommand{\blank}{\underline{\hphantom{A}}}

\newcommand {\codim} {\operatorname{codim}}

\newcommand {\Coker} {\operatorname{Coker}}

\newcommand {\Ext} {\operatorname{Ext}}
\newcommand{\sExt}{\mathscr{E} \kern -1pt xt}

\newcommand{\Hilb}{\mathrm{Hilb}}

\newcommand {\Hom} {\operatorname{Hom}}
\newcommand {\sHom}{\mathscr{H}\kern-5pt\mathcalligra{om}}

\newcommand {\Id} {\operatorname{Id}}
\newcommand {\im} {\operatorname{im}}
\renewcommand {\Im} {\operatorname{Im}}

\newcommand {\Jac} {\operatorname{Jac}}

\newcommand {\kk} {\Bbbk}
\renewcommand {\ker } {\operatorname{Ker}}
\newcommand {\Ker} {\operatorname{Ker}}

\newcommand {\Pic} {\operatorname{Pic}}

\newcommand {\Proj} {\operatorname{Proj}}

\newcommand {\pr} {\operatorname{pr}}

\newcommand {\rank} {\operatorname{rank}}

\newcommand {\Span} {\operatorname{Span}}
\newcommand {\Spec} {\operatorname{Spec}}

\newcommand {\Sym} {\operatorname{Sym}}

\newcommand{\sTor}{\mathscr{T} \kern -3pt or}
\newcommand {\Tot} {\operatorname{Tot}}

\newcommand {\Bl} {\operatorname{Bl}}

\newcommand {\CH} {\operatorname{CH}}
\newcommand {\IH} {\operatorname{IH}}

\title[]{On the Chow theory of projectivizations}
\author[Q.Y.\ JIANG]{Qingyuan Jiang}


\address{School of Mathematics, University of Edinburgh, JCMB, Peter Guthrie Tait Road, Edinburgh EH9 3FD, UK.}
\email{Qingyuan.Jiang@ed.ac.uk}

\begin{document}

\begin{abstract} In this paper, we prove a decomposition result for the Chow groups of projectivizations of coherent sheaves of homological dimension $\le 1$. In this process, we establish the decomposition of Chow groups for the cases of Cayley's trick and standard flips. Moreover, we apply these results to study the Chow groups of symmetric powers of curves, nested Hilbert schemes of surfaces, and the varieties resolving Voisin maps for cubic fourfolds.
\vspace{-5mm} 
\end{abstract}

\maketitle

\section{Introduction}
Let $X$ be a Cohen--Macaulay scheme of pure dimension, and $\sG$ a coherent sheaf on $X$ of rank $r$ and homological dimension $\le 1$, that is, locally over $X$, there is a two-step resolution $0 \to \sF \to \sE \to \sG \to 0$, where $\sF$ and $\sE$ are finite locally free sheaves. (If $X$ is regular, this condition on $\sG$ is equivalent to $\sExt^i_X(\sG,\sO_X) =0$ for all $i \ge 2$.) The projectivization $\pi \colon \PP(\sG) : = \Proj_X \Sym_{\sO_X}^\bullet \sG \to X$ of $\sG$ is generically a projective bundle with fiber $\PP^{r-1}$; However, the dimension of the fiber of $\pi$ jumps along the degeneracy loci (see \S \ref{sec:deg}) of $\sG$.

The derived category of $\PP(\sG)$ was studied in \cite{JL18}, where we prove (under certain regularity and dimension conditions) that there is a semiorthogonal decomposition:
	$${\rm D^b_{coh}}(\PP(\sG)) = \big \langle {\rm D^b_{coh}}(\PP(\sExt^1(\sG,\sO_X))), ~{\rm D^b_{coh}}(X)\otimes \sO(1), \ldots, {\rm D^b_{coh}}(X)\otimes \sO(r)\big \rangle.$$ 
(For a space $Y$, ${\rm D^b_{coh}}(Y)$ stands for its bounded derived category of coherent sheaves). The theorem states that the (right) orthogonal of the ``projective bundle part" of ${\rm D^b_{coh}}(\PP(\sG))$ is given by the derived category of another projectivization $\PP(\sExt^1(\sG,\sO_X))$, which is a Springer type partial desingularization of the singular locus of $\sG$. See \cite{JL18} for more details

In this paper, we establish the Chow-theoretic version of the above formula:

\begin{theoremstar}(See Theorem \ref{thm:main}) Let $X$ and $\sG$ be as above.  Assume either 
	\begin{enumerate}
	\item[(A)] $\PP(\sG)$ and $\PP(\sExt^1(\sG,\sO_X))$ are non-singular and quasi-projective, and the degeneracy loci of $\sG$ satisfy a weak dimension condition $(\ref{eqn:weak.dim})$; Or
	\item[(B)] All degeneracy loci of $\sG$ (are either empty or) have expected dimensions. 
	\end{enumerate}
Then for each $k \ge 0$, there is an isomorphism of integral Chow groups:
	\begin{align*} 
	\CH_k(\PP(\sG)) \simeq \CH_{k-r}(\PP(\sExt^1(\sG,\sO_X)))  \oplus \bigoplus_{i=0}^{r-1} \CH_{k-(r-1)+i}(X).
	\end{align*}
\end{theoremstar}
Since the isomorphism of the theorem commutes with product with another space, by Manin's identity principle, if $\PP(\sG)$, $\PP(\sExt^1(\sG,\sO_X))$ and $X$ are smooth and projective over the ground field $\kk$, then there is an isomorphism of integral (pure effective) Chow motives:
	\begin{align*} 
		\foh(\PP(\sG)) \simeq \foh(\PP(\sExt^1(\sG,\sO_X)))(r) \oplus \bigoplus_{i=0}^{r-1} \foh(X)(i). 
	\end{align*}
See Corollary \ref{cor:main.motive}. Note that this result compares nicely with Vial's work \cite{Vial} on $\PP^{r-1}$-fibrations; In our case, $\PP(\sG)$ is a {\em generic} $\PP^{r-1}$-fibration. By taking a cohomological realization, for example by taking the Betti cohomology if $\kk \subset \CC$, the isomorphism of Motives induces an isomorphism of rational Hodge structures:
	\begin{align*} 
		H^n(\PP(\sG), \QQ) \simeq H^{n-2r}(\PP(\sExt^1(\sG,\sO_X)), \QQ)   \oplus\bigoplus_{i=0}^{r-1} H^{n-2i}(X,\QQ), \qquad \forall n \ge 0.
	\end{align*}	

This paper provides two approaches to proving the above theorem, each under one of the conditions (A) and (B). The idea behind both approaches is that one could view the projectivization phenomenon as a combination of Cayley's trick and flips. 

We study the Chow theory for Cayley's trick in \S \ref{sec:Cayley} (see Theorem \ref{thm:Cayley} and Corollary \ref{cor:Cayley}), and the Chow theory of standard flips  in \S \ref{sec:flips} (see Theorem \ref{thm:flip} and Corollary \ref{cor:flip}). These results are of independent interest on their own. For example, it follows from Theorem \ref{thm:Cayley} and Corollary \ref{cor:Cayley} that the Chow group (resp. motive) of every complete intersection variety can be split embedded into that of a Fano variety, 
see Example \ref{example:fano} (cf. \cite{KKLL}).

First examples of the theorem are: (i) universal $\Hom$ spaces, see \S \ref{sec:Hom}, (ii) flops from Springer type resolutions of determinantal hypersurfaces, see \S \ref{sec:springer}, and (iii) a blowup formula for blowing up along Cohen--Macaulay codimension $2$ subschemes, see \S \ref{sec:CM2}.
 
\subsection*{Applications}
The following applications parallel the applications of the projectivization formula in the study of derived categories \cite{JL18}. 
\begin{enumerate}[leftmargin=*]
	\item {\em Symmetric powers of curves \S\ref{sec:SymC}.} Let $C$ be a smooth projective curve of genus $g \ge 1$ and denote by $C^{(k)}$ the $k$-th symmetric power. For any $0 \le n \le g-1$, the relationships between the derived category of $C^{(g-1+n)}$ and $C^{(g-1-n)}$ (and also the Jacobian variety $\Jac(C)$) was established by Toda \cite{Tod} using wall--crossing of stable pair moduli, and later by \cite{JL18,BK19} using the projectivization formula. The main theorem of this paper implies the corresponding Chow-theoretic version of the formula: for any $k \ge 0$, there is an isomorphism of integral Chow groups
	\begin{align*}
	 \CH_{k}(C^{(g-1+n)}) \simeq \CH_{k-n}(C^{(g-1-n)}) \oplus \bigoplus_{i=0}^{n-1} \CH_{k-(n-1)+i}(\Jac(C)),
	\end{align*}
and a similar decomposition for integral Chow motives, see Corollary \ref{cor:SymC}.

	\item {\em Nested Hilbert schemes of surfaces \S \ref{sec:Hilb}}. Let $S$ be a smooth quasi-projective surface, and denote by $\Hilb_{n}(S)$ the Hilbert scheme of $n$ points on $S$, by convention $\Hilb_{1}(S) = S$, $\Hilb_{0} = {\rm point}$. Denote $\Hilb_{n,n+1}(S)$ the nested Hilbert scheme. Then projectivization formula of derived categories \cite{JL18} can be applied to obtain a semi-orthogonal decomposition of $D(\Hilb_{n,n+1}(S))$, see Belmans--Krug \cite{BK19}. In this paper, we show that for any $k \ge 0$, there is an isomorphism of integral Chow groups:
	\begin{align*}
	\CH_{k}(\Hilb_{n,n+1}(S)) &\simeq \CH_{k-1}(\Hilb_{n-1,n}(S)) \oplus \CH_{k}(\Hilb_n(S) \times S) \\
	& \simeq \bigoplus_{i=0}^n \CH_{k-i}(\Hilb_{n-i}(S) \times S),
	\end{align*}
and a similar decomposition for Chow motives, see Corollary \ref{cor:Hilb}.	
	
	\item {\em Voisin maps \S \ref{sec:Voisin}.} Let $Y$ be a cubic fourfold not containing any plane, let $F(Y)$ be the Fano variety of lines on $Y$, and let $Z(Y)$ be the corresponding LLSvS eightfold \cite{LLSVS17}. Voisin \cite{Voi16} constructed a rational map $v \colon F(Y) \times F(Y) \dashrightarrow Z(Y)$ of degree six, Chen \cite{Chen} showed that the Voisin map $v$ can be resolved by blowing up the indeterminacy locus $Z = \{ (L_1,L_2) \in F(Y) \times F(Y) \mid L_1 \cap L_2 \ne \emptyset\}$, and the blowup variety is a natural relative $Quot$-scheme over $Z(Y)$ if $Y$ is very general. The main theorem can be applied to this case, and implies that for any $k \ge 0$, there is an isomorphism of Chow groups:
	\begin{align*}
	\CH_k(\Bl_Z ( F(Y) \times F(Y))) \simeq \CH_{k-1}(\widetilde{Z}) \oplus \CH_k(F(Y) \times F(Y)),
	\end{align*}
where $\widetilde{Z} = \PP(\omega_Z)$ is a Springer type (partial) resolution of the indeterminacy locus $Z$, which is an isomorphism over $Z \backslash \Delta_2$, and a $\PP^1$-bundle over the type II locus $\Delta_2 = \{L \in \Delta \simeq F(Y) \mid \sN_{L/Y} \simeq \sO(1)^{\oplus 2} \oplus \sO(-1)\}$ which is an algebraic surface. See Corollary \ref{cor:Voisin}. 
\end{enumerate}

The results of this paper could also be applied to many other situations of moduli spaces, for example, moduli of sheaves on surfaces \cite{Neg1, Neg}, and the moduli  spaces of extensions of stable objects in K3 categories, which are generalizations of the varieties resolving Voisin's maps \cite{Voi16, Chen}.  Another such example is provided by the pair of Thaddeus moduli spaces \cite{Tha} $M_C(2,\sL) \to N_C(2,\sL)$ and $M_C(2,\sL^\vee \otimes \omega_C) \to N_C(2,\sL)$ studied by Koseki and Toda \cite{KT}. (Here, $\sL$ is a line bundle of odd degree $d>0$, $N_C(2,\sL)$ is the moduli space of rank $2$ semistable vector bundles over a curve $C$, with determinant $\sL$, and  $M_C(2,\sL)$ is the space $M_{\omega}$ of \cite{Tha}, where $\omega = [\frac{d-1}{2}]$). The results of this paper on flips \S \ref{sec:flips} and projectivizations Theorem \ref{thm:main} would shed light on the study of Chow theory of $N_C(2,\sL)$. \footnote{See \cite{FHL} for recent results in this direction about rational Chow motives; Our results here might also be helpful to obtain results for integral coefficients.}


\subsection*{Convention} Throughout this paper, $X$ is a Noetherian scheme of pure dimension, and $\sG$ is a coherent sheaf over $X$. We say that $\sG$ has rank $r$ if the rank of $\sG(\eta) : = \sG \otimes \kappa(\eta)$ is $r$ at the generic point $\eta$ of each irreducible component of $X$. Assume all schemes in consideration are defined over some fixed ground field $\kk$. The terms ``locally free sheaves" and ``vector bundles" will be used interchangeably. We use {\bf Grothendieck's notations}: for a coherent sheaf $\sF$ on a scheme $X$, denote by $\PP_X(\sF)  =\Proj_X \Sym_{\sO_X}^\bullet \sF$ its projectivization; we will write $\PP(\sF)$ if the base scheme is clear from context. For a vector bundle $V$, we also use $\PP_{\rm sub}(V) : = \PP(V^\vee)$ to denote the moduli space of $1$-dimensional linear subbundles of $V$. 

For motives, we use the {\bf covariant} convention of \cite{KMP, Mu1, Mu2, Vial, Vial15}. In particular, \cite{KMP} contains a dictionary for translating between covariant and contravariant conventions. For a smooth projective variety $X$ over a field $\kk$, denote by $\foh(X)$ its class $(X, \Id_X,0)$ in Grothendieck's category of integral Chow motives of smooth projective varieties over $\kk$. Notice that under the {\em covariant} convention, for a morphism $f \colon X \to Y$ of smooth projective varieties, $\Gamma_f$ induces the pushforward map $f_* \colon \foh(X) \to \foh(Y)$ and $[\Gamma_f^t]$ induces the pullback map $f^* \colon \foh(Y) \to \foh(X)(\dim Y - \dim X)$. Moreover, $\foh(\PP^1) = 1 \oplus \LL = 1 \oplus 1(1)$, where $1 = \foh(\Spec k)$, and  $\LL = 1(1)$ is the Lefschetz motive. In particular, the covariant Tate twist coincide with tensoring with $\LL$, i.e., $\foh(X)(i) = \foh(X) \otimes \LL^i$ for all $i \in \ZZ$. Furthermore, $\CH^\ell(\foh(X)(n)) = \CH^{\ell - n}(X)$ and $\CH_{k}(\foh(X)(n)) = \CH_{k-n}(X)$. We will use $h$ to denote the action $c_{1}(\sO(1)) \cap (\blank)$ on motives when the line bundle $\sO(1)$ is clear from the context.

\subsection*{Acknowledgement} 
The author would like to thank Arend Bayer for many helpful discussions, thank especially Huachen Chen for bringing his attention to this problem and many helpful discussions on Voisin maps and his work \cite{Chen}, and thank Dougal Davis for helpful conversations. 
This project started during a workshop at Liverpool, for which the author thanks the organisers Alice Rizzardo and Theo Raedschelders for hospitality. The author also thank the referee for the careful reading and many helpful suggestions which greatly improve the exposition of the paper. This work is supported by the Engineering and Physical Sciences Research Council (EPSRC) [EP/R034826/1].

\section{Preliminaries}

\subsection{Degeneracy loci} \label{sec:deg} Standard references are \cite{FP, Ful, GKZ, GG, Laz04}.
\begin{definition}
\begin{enumerate}[leftmargin=*]
	\item  Let $\sG$ be a coherent sheaf of (generic) rank $r$ over a scheme $X$. For an integer $k \in \ZZ$, the {\em degeneracy locus of $\sG$ of rank $\ge k$} is defined to be
	$$X^{ \ge k}(\sG): = \{x \in X \mid \rank \sG(x) \ge k\},$$
where $\sG(x) := \sG_x \otimes_{\sO_{X,x}} \kappa(x)$ is the fiber of $\sG$ at $x \in X$. Notice that $X^{ \ge k}(\sG) = X$ if $k \le r$; $X_{\mathrm{sg}}(\sG) : = X^{\ge r+1}(\sG)$ is called {\em first degeneracy locus} (or the {\em singular locus}) of $\sG$. 
	\item Let $\sigma: \sF \to \sE$ a morphism of $\sO_X$-modules between locally free sheaves $\sF$ and $\sE$ on $X$. For an integer $\ell$, the {\em degeneracy locus of $\sigma$ of rank $\ell$} is defined to be
	$$D_\ell(\sigma) := \{ x \in X \mid \rank \sigma (x) \le \ell \},$$
where $\sigma(x) := \sigma_x  \otimes_{\sO_{X,x}} \kappa(x) \colon \sF(x) \to \sE(x)$ is the map induced by $\sigma$ on the fibers.
\end{enumerate}
\end{definition}
The degeneracy loci $X^{\ge k}(\sG)$ and $D_\ell(\sigma)$ have natural closed subscheme structures given by Fitting ideals, see \cite[\S 7,2]{Laz04}. The two notions are related as follows: let $\sigma \colon \sF \to \sE$ be an $\sO_X$-module map between finite locally free sheaves, and let $\sG: = \Coker (\sigma)$ be the cokernel, then $X^{\ge k}(\sG) = D_{\rank \sE - k}(\sigma)$ as closed subschemes of $X$.

The expected codimension of $D_{\ell}(\sigma) \subset X$ is $(\rank \sE -\ell)(\rank \sF -\ell)$ (if $\ell \le \min\{\rank \sE, \rank \sF\}$). If $\sG$ has homological dimension $\le 1$ and rank $r$, for example if $\sG= \Coker (\sF \xrightarrow{\sigma} \sE)$ is the cokernel of an injective map of $\sO_X$-modules between finite locally free sheaves,  then for any $i \ge 0$, the expected codimension of $X^{\ge r+i}(\sG) \subset X$ is $i(r+i)$.

In the universal local situation where $X=\Hom_\kk(W,V)$ is the total space of maps between two vector spaces $W$ and $V$ over a field $\kk$, there is a tautological map $\tau \colon W \otimes \sO_X \to V \otimes \sO_X$ over $X$, such that $\tau(A) = A$ for $A \in \Hom(W,V)$. 

\begin{lemma}[\cite{FP, GKZ, GG}]  \label{lem:normalbundle:Hom} Let $X = \Hom_\kk(W,V)$, denote $D_\ell  = D_\ell(\tau) \subseteq X$ the degeneracy locus of the tautological map $\tau$ of $\rank \ell$. Then for any $0 \le \ell \le \min\{\rank W, \rank V\}$, the singular locus of $D_\ell$ is $D_{\ell-1}$. Furthermore, for any regular point $A \in D_\ell \backslash D_{\ell-1}$:
	\begin{enumerate}
		\item \label{lem:normalbundle:Hom-1} 
		The tangent space of $D_\ell$ at $A$ is $T_{A} D_\ell = \{T \in \Hom(W,V) \mid T(\ker A) \subseteq \Im A\}$.
		\item \label{lem:normalbundle:Hom-2} 
		The normal space of $D_\ell$ to $X$ at $A$ is $N_{D_\ell} X |_{A} = \Hom (\Ker A, \Coker A)$.
	\end{enumerate}
\end{lemma}
\begin{proof} See \cite[\S 5.1, p. 54-55]{FP}, or \cite[Lemma 4.12]{GKZ}, or \cite[Ex. V (4), p. 145]{GG}.
\end{proof}

In general, let $\sigma: \sF \to \sE$ be a map between vector bundles over a scheme $X$. For a fixed integer $\ell$, regarding the open degeneracy locus $D: =  D_\ell(\sigma) \backslash D_{\ell-1}(\sigma)$, we have the following:

\begin{lemma} \label{lem:normalbundle} Assume $X$ is a Cohen-Macaulay $\kk$-scheme, and $D: =  D_\ell(\sigma) \backslash D_{\ell-1}(\sigma) \subset X$ has the expected codimension $(\rank \sE -\ell)(\rank \sF -\ell)$. Then $\sigma|_D \colon \sE|_D \to \sF|_D$ has constant rank $\ell$ over $D$; $K: = \Ker \sigma|_D$ and $C: = \Coker \sigma|_D$ are locally free sheaves over $D$ of ranks $\rank \sE - \ell$ and $\rank \sF - \ell$ respectively. Moreover, $D \subset X$ is a locally complete intersection subscheme with normal bundle $N_{D/X} \simeq K^\vee \otimes C$.
\end{lemma}

\begin{proof} First, we prove the lemma for the total Hom space $H = |\Hom_X(\sF,\sE)|$. Denote $\pi \colon H = |\Hom_X(\sF,\sE)| \to X$ the projection, and let $\DD_\ell : = D_\ell(\tau_H)\subset H$ be the degeneracy locus for the tautological map $\tau_H \colon \pi^* \sF \to \pi^* \sE$. As the statement is local, we may assume $X = \Spec A$,  $\sF = W \otimes_\kk A$, $\sE = V \otimes_\kk A$, where $A$ is a $\kk$-algebra,$W, V$ are $\kk$-vector spaces. Then $H = \Hom(W, V) \times_\kk X$ is the flat base-change of $\Hom_\kk(W,V)$ along $X \to \Spec \kk$, and $\DD := \DD_\ell \backslash \DD_{\ell-1} = D \times_\kk X$. The desired result holds for $H$ and $\DD$ by Lemma \ref{lem:normalbundle:Hom}.

In general, the map $\sigma \colon \sF \to \sE$ induces a section map $s_{\sigma} \colon X \to H$, such that $\sigma = s_{\sigma}^* \tau_H$ and $D = \DD \times_{X} H$. Since $s$ is the section of a smooth separated morphism, it is a regular closed immersion. Since $H$ and $X$ are Cohen-Macaulay, $\DD \hookrightarrow H$ is a regular immersion, and the intersection $D = \DD \times_X  H \hookrightarrow X$ has the expected codimension, therefore the inclusion $D \hookrightarrow X$ is also a regular immersion, with normal bundle $N_{D/X} = s_{\sigma}^* N_{\DD/H}$. Finally, $s_{\sigma}^* N_{\DD/H} = K^\vee \otimes C$ holds since $K^\vee =  s_{\sigma}^* \Coker(\tau_H^\vee)$ and $C = s_{\sigma}^* \Coker(\tau_H)$.
\end{proof}

\subsection{Chow groups of projective bundles}
Let $X$ be a scheme, and $\sE$ a locally free sheaf of rank $r$ on $X$. Denote $\pi \colon \PP(\sE): = \Proj(\Sym^\bullet \sE) \to X$ the projection. Notice that our convention $\PP(\sE) = \PP_{\rm sub}(\sE^\vee)$ is dual to Fulton's \cite{Ful}. For simplicity, from now on we will denote $\zeta = c_1(\sO_{\PP(\sE)}(1))$, and use the notation $\zeta^i  \cdot \beta : =c_1(\sO_{\PP(\sE)}(1))^i \cap \beta$, where $\beta \in \CH(\PP(\sE))$, to denote the cap product. For each $i \in [0,r-1]$, we introduce the following notations:
	$$\pi_i^*(\blank) = \zeta^i \cdot \pi^*(\blank) \colon \CH_{k-(r-1)+i}(X) \to \CH_k(\PP(\sE)), \quad \forall k \in \ZZ.$$	

The following results are summarised and deduced from \cite[Proposition 3.1, Theorem 3.3]{Ful} but presented in a way that fits better into our current work.

\begin{theorem}[Projective bundle formula] \label{thm:proj.bundle}
\begin{enumerate}
	\item \label{thm:proj.bundle-1}
	(Duality) For any $\alpha \in \CH(X)$,
	\begin{align*}
	\pi_* \, \pi_i^* (\alpha) =  \pi_* (c_1(\sO(1))^i \cap \pi^*(\alpha)) = 
	 	\begin{cases} 0, & i< r-1, \\ 
		\alpha, & i=r-1 . 
		\end{cases}
	\end{align*}
	\item \label{thm:proj.bundle-2}
	For any $k \in \NN$, there is an isomorphism of Chow groups:
	\begin{align*} 
	  \bigoplus_{i=0}^{r-1} \pi_i^* \colon   \bigoplus_{i=0}^{r-1} \CH_{k-(r-1)+i}(X)  \xrightarrow{\sim} \CH_k(\PP(\sE)).
		\end{align*}
	\item \label{thm:proj.bundle-3}
	The projection to the $i$-th summand of the above isomorphism is given by
		\begin{align} \label{eqn:lem:proj}
		\pi_{i\,*} (\blank)= \sum_{j=0}^{r-1-i} (-1)^jc_{j}(\sE) \cap \pi_* (\zeta^{r-1-i-j} \cdot (\blank) ), \quad \text{for} \quad i = 0, 1, \ldots , r-1.
		\end{align}
		Therefore  for any $i,j \in [0,r-1]$, the following holds:
	$$\pi_{i\,*} \, \pi_i^* = \Id_{\CH(X)}, \qquad \pi_{i\,*} \pi_j^* = 0, i\ne j, \quad \text{and} \quad \Id_{\CH(\PP(\sE))} = \sum_{i=0}^{r-1}  \pi_i^* \, \pi_{i\,*}.$$
		\end{enumerate}
\end{theorem}

\begin{proof} \eqref{thm:proj.bundle-1} follows from \cite[Proposition 3.1(a)]{Ful}, and \eqref{thm:proj.bundle-2} follows from \cite[Theorem 3.3]{Ful}, which could also be viewed as a special case of \cite[Proposition 14.6.5]{Ful}. For \eqref{thm:proj.bundle-3}, to agree with Fulton's notation, let $E = \sE^\vee$ be the dual vector bundle, so $\PP_{\rm sub}(E) = \PP(\sE)$. From \eqref{thm:proj.bundle-2}, for any $\beta \in \CH_k(\PP(\sE))$, there exists unique $\alpha_i \in \CH_{k-(r-1)+i}(X)$, $i \in [0,r-1]$ such that 
	$$\beta =  \sum_{i=0}^{r-1} \zeta^i \cdot \pi^* \alpha_i.$$
It follows from the definition of Segre classes $s_i(E) \cap \alpha : = \pi_*(\zeta^{i+r-1} \cdot \pi^* \alpha)$ that:
	$$
	 \pi_* (\zeta^{j} \cdot \beta)  =  \sum_{i=0}^{j}  s_{i}(E) \cap \alpha_{r-1-j+i}, \quad \text{for} \quad  j=0,1,\ldots, r-1.
	$$
Then the desired results follow from solving $\alpha_{i}$'s from the above equations by  using $1 = c(E)s(E) = (1+c_1(E)+c_2(E) + \ldots)(1+ s_1(E)+s_2(E) + \ldots)$.
\end{proof}

Notice that our maps $\pi_{i\,*}$ (resp. projectors $\pi_i^* \, \pi_{i\,*}$) are nothing but the explicit expressions of the correspondences $g_i$ (resp. orthogonal projectors $p_{r-i}$) that are inductively defined in \cite[\S 7, p 457, Definition]{Manin} (resp. \cite[\S 7, p 456, Proposition]{Manin}). By using these maps, Manin \cite[\S 7, p 457]{Manin} establishes an isomorphism of Chow motives:
	\begin{align*}  
	& \bigoplus_{j=0}^{r-1} h^{r-1-j} \circ \pi^*  \colon  \bigoplus_{j=0}^{r-1} \foh(X) (j)  \xrightarrow{\sim}  \foh(\PP(\sE)).
	\end{align*}

\begin{remark} The projector $\pi_{i\,*}$ can be expressed via the universal quotient bundle as: 
	$$\pi_{i*} = \pi_* (c_{r-1-i}(\shT_{\PP(\sE)/X}(-1)) \cap (\blank)) \colon \CH(\PP(\sE)) \to \CH(X).$$
This duality is explained for more general Grassmannian bundles in \cite{J20}.
\end{remark}

\begin{remark}\label{rmk:basis} (Change of basis) For any identification $\PP(\sE) \simeq \PP(\sE \otimes \sL)$, where $\sL \in \Pic X$, if we denote $\zeta' = c_1(\sO_{\PP(\sE \otimes \sL)}(1)) = \zeta + \pi^*c_1(\sL)$, and $\pi'_{i\,*}$ the projectors with respect to $\zeta'^i \cdot \pi^*(\blank)$. Then the two basis $\{\zeta^i\}_{0 \le i \le r-1}$ and $\{\zeta'^i\}_{0 \le i \le r-1}$ differers by an invertible upper triangular change of basis. In particular, for any $0 \le k \le r-1$, the following holds:
	$$\Span \{ \zeta^i \mid 0 \le i \le k \} = \Span \{ \zeta'^i \mid 0 \le i \le k \},$$
where for any subset $\shS \subset \CH^*(\PP(\sE))$, its span is defined by 
	$$\Span \shS : = \big \{ \sum_i  \alpha_i \cap \pi^* \beta_i  \mid \alpha_i \in S, \beta_i \in \CH(X) \big\}.$$
Similarly, for any $0 \le k \le r-1$, $\pi_{k\,*}'$ can be expressed as $\CH(X)$-linear combination of $\pi_{k\,*}, \pi_{k+1\,*}, \ldots, \pi_{r-1\,*}$, and vice versa. 
\end{remark}

\begin{lemma}[See {\cite[Lemma 5.3]{Rie}}] \label{lem:chern} The following equality holds:
	$$c_k(\Omega_{\PP(\sE)/X}(1)) = \sum_{i=0}^k (-1)^i  \zeta^i \cdot \pi^*c_{k-i}(\sE) = (-1)^k \sum_{i=0}^k   \zeta^i \cdot \pi^*c_{k-i}(\sE^\vee).$$
\end{lemma}

\subsection{Blowups}
Let $Z \subset X$ be a codimension $r \ge 2$ locally complete intersection subscheme. Denote $\pi: \widetilde{X} \to X$ the blowup of $X$ along $Z$, with exceptional divisor $E \subset \widetilde{X} $. Then $E = \PP(\sN_{Z/X}^\vee)$ is a projective bundle over $Z$. We have a Cartesian diagram:
	\begin{equation*}
	\begin{tikzcd}[row sep= 2.6 em, column sep = 2.6 em]
	E \ar{d}[swap]{p} \ar[hook]{r}{j} & \widetilde{X} \ar{d}{\pi} \\
	Z \ar[hook]{r}{i}         & X 
	\end{tikzcd}	
	\end{equation*}
The {\em excess bundle} $\sV$ for the diagram is defined by the short exact sequence:		\begin{align*}
	0 \to \sN_{E/\widetilde{X}} \to p^*\sN_{Z/X} \to \sV \to 0.
	\end{align*}
From the excess bundle formula \cite[Theorem 6.3]{Ful}, one obtains the {\em key formula} for blowup:
	\begin{align} \label{eqn:k.f.blowup}
	\pi^* \,  i_* (\blank) = j_* (c_{r-1}(\sV) \cap p^*(\blank)) \colon \qquad \CH_{k}(Z) \to \CH_k(\widetilde{X}).
	\end{align}
The following is summarised from \cite[Proposition 6.7]{Ful}:

\begin{theorem}[Blowups] \label{thm:blowup}
\begin{enumerate}
	\item  \label{thm:blowup-1}
	The following holds: 
		$$\pi_* \, \pi^* = \Id_{\CH(X)} \quad \text{and} \quad p_*(c_{r-1}(\sV) \cap p^*(\blank)) = \Id_{\CH(Z)}.$$
	\item  \label{thm:blowup-2}
	For any $k \ge 0$, there exists a split short exact sequence:
	\begin{align*}
		0 \to \CH_{k}(Z) \xrightarrow{(c_{r-1}(\sV) \cap p^*(\blank), \, -i_*)} \CH_{k}(E) \oplus \CH_{k}(X) \xrightarrow{(\varepsilon, \alpha) \mapsto j_*\varepsilon + \pi^* \alpha} \CH_k(\widetilde{X}) \to 0,
	\end{align*}
where a left inverse of the first map is given by $(\varepsilon, \alpha) \mapsto p_* \, \varepsilon$.
	\item  \label{thm:blowup-3}
	The above exact sequence induces an isomorphism of Chow groups
	\begin{align*} 
		 \CH_{k}(X) \oplus \bigoplus_{i=0}^{r-2} \CH_{k-(r-1)+i}(Z)  \xrightarrow{\sim} \CH_{k}(\widetilde{X}), 
	\end{align*}
given by $(\alpha, \oplus_{i=0}^{r-2}  \beta_i) \mapsto \pi^* \, \alpha + j_*(\sum_{i=0}^{r-2} \zeta^i \cdot p^*\beta_i)$, where $\zeta = c_1(\sO_{\PP(\sN_{Z/X}^\vee)}(1))$.
\end{enumerate}
\end{theorem}
Note that the well-known formula of \eqref{thm:blowup-3} follows from \eqref{thm:blowup-2} by the identification
	\begin{align*}
	\CH_{k}(\widetilde{X}) & = \pi^* \CH_{k}(X) \oplus j_* (\CH_k(E)_{p_*=0})  \\
	 &=  \pi^* \CH_{k}(X) \oplus  \bigoplus_{i=0}^{r-2} j_* (\zeta^ i \cdot p^*\CH_{k-(r-1)+i}(Z)),
	\end{align*}
where $\CH_k(E)_{p_*=0}$ denote the subgroup $\{\gamma \in \CH_k(E) \mid p_* \, \gamma = 0\}$ of $\CH_k(E)$. A similar and more detailed argument is given later in the case of standard flips, see Theorem \ref{thm:flip}. There are similar results on Chow motives by Manin \cite{Manin}; see also Corollary \ref{cor:flip} below.

\section{Cayley's trick and standard flips}
The projectivization can be viewed as a combination of the situation of Cayley's trick and flips. In this chapter we study the Chow theory of the latter two cases.

\subsection{Cayley's trick and Chow group} \label{sec:Cayley}
Cayley's trick is a method to relate the geometry of the zero scheme of a regular section of a vector bundle to the geometry of a hypersurface; see the discussions of \cite[\S 2.3]{JL18}. The relationships for their derived categories were established by Orlov \cite[Proposition 2.10]{Orlov06}, we now focus on their Chow groups.
 
Let $\sE$ be a locally free sheaf of rank $r \ge 2$ on a scheme
$X$, and $s \in H^0(X,\sE)$ a regular section, and denote $Z:=Z(s)$ the zero locus of the section $s$. Denote the projectivization by $q\colon \PP(\sE) = \Proj \Sym^\bullet \sE \to X$. Then under the canonical identification 
	$$H^0(X,\sE) = H^0(\PP(\sE), \sO_{\PP(\sE)}(1))$$
 the section $s$ corresponds canonically to a section $f_s$ of $\sO_{\PP(\sE)}(1)$ on $\PP(\sE)$. Denote the divisor defined by $f_s$ by:
 	$$\shH_s : = Z(f_s) \subset \PP(\sE).$$
 Then $\shH_s = \PP(\sG) =  \Proj \Sym^\bullet \sG$, where $\sG = {\rm coker}(\sO_X \xrightarrow{~s~} \sE)$. Thus $\shH_s$ is a $\PP^{r-2}$-bundle over $X \backslash Z$, and a $\PP^{r-1}$-bundle over $Z$. It follows that $\shH_s|_{Z}$ coincides with $\PP_Z(\sN_i)$, the projectivization of the normal bundle of inclusion $i \colon Z \hookrightarrow X$. The situation is illustrated in the following commutative diagram, with maps as labeled:
\begin{equation}\label{diagram:Cayley}
	\begin{tikzcd}[row sep= 3 em, column sep = 3.6 em]
	\PP(\sN_i) \ar{d}[swap]{p} \ar[hook]{r}{j} & \shH_s \ar{d}{\pi} \ar[hook]{r}{\iota} & \PP(\sE) \ar{ld}[near start]{q} 
	\\
	Z \ar[hook]{r}{i}         & X  
	\end{tikzcd}	
\end{equation}
Since $\sN_i = \sE|_{Z}$ and $\sO_{\PP(\sE)}(1)|_{\PP(\sN_i)} = \sO_{\PP(\sN_i)}(1)$, by abuse of notations, we use $\zeta \cdot (\blank)$ to denote both $c_1(\sO_{\PP(\sE)}(1))\cap(\blank)$ and $c_1(\sO_{\PP(\sN_i)}(1))\cap(\blank)$. The main result of this section is:

\begin{theorem}[Cayley's trick for Chow groups] \label{thm:Cayley}
There exists a split short exact sequence:
	\begin{align}\label{eqn:thm:Cayley:ses}
		0 \to \bigoplus_{i=0}^{r-2} \CH_{k-(r-2)+i}(Z) \xrightarrow{f}  \bigoplus_{i=0}^{r-2} \CH_{k-(r-2)+i}(X) \oplus \CH_{k}(\PP(\sN_i)) \xrightarrow{g} \CH_k(\shH_s) \to 0,
	\end{align}
where the maps $f$ and $g$ are given by 
	\begin{align*}
	&f  \colon \oplus_{i=0}^{r-2}  \gamma_i \mapsto (- \oplus_{i=0}^{r-2} i_*  \gamma_i , ~\sum_{i=0}^{r-2} \zeta^{i+1} \cdot p^* \gamma_i), \\
	&g \colon (\oplus_{i=0}^{r-2}  \alpha_i, ~\varepsilon) \mapsto \sum_{i=0}^{r-2} \zeta^i \cdot \pi^*\alpha_i + j_* \varepsilon
	\end{align*}
(where $p_{i\,*}$ is defined below similar to (\ref{eqn:lem:proj})). A left inverse of $f$ is given by 
	$(\oplus_{i=0}^{r-2}  \alpha_i, ~\varepsilon) \mapsto \oplus_{i=0}^{r-2} p_{i+1\,*} \varepsilon$.
 Furthermore, the above sequence induces an isomorphism
	\begin{align} \label{eqn:cayley.chow}
		\bigoplus_{i=0}^{r-2} \CH_{k-(r-2)+i}(X) \oplus \CH_{k-(r-1)}(Z) \xrightarrow{\sim} \CH_{k}(\shH_s), 
	\end{align}
given by $(\oplus_{i=0}^{r-2}  \alpha_i, ~\gamma) \mapsto \sum_{i=0}^{r-2} \zeta^i \cdot \pi^*\alpha_i + j_* p^* \gamma$, and in this decomposition the projection map to the first $(r-1)$-summands $\CH_{k}(\shH_s) \to \CH_{k-(r-2)+i}(X)$, $i=0,1,\ldots, r-2$, is given by $\beta \mapsto \pi_{i\,*} \,\beta$, where $\pi_{i\,*}$ is defined below by (\ref{eqn:pi_i}), and the projection to the last summand $\CH_{k}(\shH_s) \to \CH_{k-(r-1)}(Z)$ is given by $\beta \mapsto (-1)^{r-1} p_* \, j^*\beta$.
\end{theorem}

For simplicity, we introduce the following notations. For the projective bundles $q: \PP(\sE) \to X$ and $p: \PP(\sN_i) \to Z$, similar to \eqref{eqn:lem:proj}, we denote the projections to the $i$-th factors by
	$$q_{i\,*} \colon \CH_k(\PP(\sE)) \to \CH_{k-(r-1)+i}(X), \quad p_{i\,*} \colon \CH_k(\PP(\sN_i)) \to \CH_{k-(r-1)+i}(Z),$$
which are explicitly given as follows: for any $i=0,1,\ldots, r-1$,
		\begin{align*}
		q_{i\,*} (\blank)& = \sum_{j=0}^{r-1-i} (-1)^jc_{j}(\sE) \cap q_* (\zeta^{r-1-i-j} \cdot (\blank) ), \\
		p_{i\,*} (\blank)& = \sum_{j=0}^{r-1-i} (-1)^jc_{j}(\sN_i) \cap p_* (\zeta^{r-1-i-j} \cdot (\blank)).
		\end{align*}
Furthermore, for any $i \in [0,r-1]$, $\alpha \in \CH(X)$, $\gamma \in \CH(Z)$, we denote
	$$q_i^* \alpha: = \zeta^i \cdot q^*\alpha \quad \text{and} \quad p_i^* \gamma := \zeta^i \cdot p^* \gamma.$$
Then the projective bundle formula Theorem \ref{thm:proj.bundle} states: (1) for all $i,j \in [0,r-1]$,
	$$q_{i\, *} \, q_j^* = \delta_{i,j} \Id_{\CH(X)}, \quad \text{and} \quad p_{i\, *} \, p_j^* = \delta_{i,j} \Id_{\CH(Z)},$$
and (2) for all $\beta \in \CH(\PP(\sE))$ and $\varepsilon \in \CH(\PP(\sN_i))$, the following relations hold:
	$$\beta = \sum_{i=0}^{r-1} q_i^*\, q_{i\,*} \,\beta \quad \text{and} \quad \varepsilon = \sum_{i=0}^{r-1} p^*_i \,p_{i\,*} \, \varepsilon.$$
	
Now for all $\alpha \in \CH_{\ell}(X)$ and $\beta \in \CH_k(\shH_s)$, and all $i \in [0,r-2]$, we define
	$$\pi_i^*\alpha : = \iota^* q_i^* \alpha \in \CH_{\ell+(r-2)-i}(\shH_s), \qquad \pi_{i\,*} \beta : = q_{i+1\,*}  \, \iota_* \, \beta \in \CH_{k-(r-2)+i}(X).$$
Then it follows from the projection formula that $\pi_i^* \alpha = \zeta^i \cdot \pi^* \alpha$, and $\pi_{r-2\,*} = \pi_*$ and 
	\begin{align} \label{eqn:pi_i}
		\pi_{i\,*} (\blank)= \sum_{j=0}^{r-2-i} (-1)^jc_{j}(\sE) \cap \pi_* (\zeta^{r-2-i-j} \cdot (\blank) ), \quad i=0,\ldots, r-2.
		\end{align}	
Notice that $c_{i}(\sE) =c_{i}(\sG)$ for $i \in [0,r-2]$,  where $\sG = {\rm coker}(\sO_X \xrightarrow{\,s\,} \sE)$; the relationships between $\pi_{i\,*}$'s and $\pi_i^*$'s are similar to the case of a $\PP^{r-2}$-bundle.

We prove the theorem by the same steps as the blowup case in \cite[\S 6.7]{Ful}:
\begin{proposition}[cf. {\cite[Proposition 6.7]{Ful}}] \label{prop:caley:key-lemma}
\begin{enumerate}
	\item [(a)] (Key formula). For all $\alpha \in \CH_k(Z)$,
	$$\pi^* \, i_* \, \alpha = j_* (\zeta \cdot p^* \alpha) \in \CH_{k+r-2}(\shH_s).$$
Then by the projection formula $\pi_i^* \, i_* \, \alpha = j_* (\zeta \cdot p_i^* \alpha)$ for all $i\in [0,r-2]$.	
	\item [(b)]  For any $\alpha \in \CH_k(X)$, $i,j \in [0,r-2]$, $\pi_{i\,*} \, \pi_i^* \, \alpha = \alpha$, $\pi_{i\,*} \, \pi_j^* \, \alpha =0$ if $i \ne j$.
	\item [(c)] For $\varepsilon \in \CH(\PP(\sN_i))$, if $j^*\,j_*\,\varepsilon  = 0$ and $p_{1\,*} \varepsilon  = \ldots = p_{r-1\,*} \varepsilon =0$, then $\varepsilon = 0$
	\item [(d)]  
		\begin{enumerate}
		\item[(i)] For any $\beta \in \CH_k(\shH_s)$ there is an $\varepsilon \in \CH_k(\PP(\sN_i))$ such that 
			$$\beta = \sum_{i=0}^{r-2} \pi_i^*\, \pi_{i\,*} \,\beta + j_* \,\varepsilon.$$
		 \item[(ii)] For any $\beta \in \CH_k(\shH_s)$, if $\pi_{i \, *} \beta = 0$, $i \in [0,r-2]$ and $j^* \beta = 0$, then $\beta=0$.
		 \end{enumerate}
\end{enumerate}
\end{proposition}

\begin{proof} 
(a). In fact, from \cite[Remark. 2.5]{JL18} the Euler sequence for $\PP(\sN_i)$ is equivalent to:
	\begin{align}\label{eqn:Euler}
	0 \to \sN_j \to p^* \sN_i \to \sO_{\PP(\sN_i)}(1) \to 0,
	\end{align}
where $\sN_i = \sE|_Z$ and $\sN_j \simeq \Omega_{\PP(\sE)/X}(1)|_{\PP(\sN_i)}$. 
Therefore the excess bundle for the diagram \ref{diagram:Cayley} is given by $ \sO_{\PP(\sN_i)}(1)$. Now from 
\cite[Theorem 6.3]{Ful} and  \cite[Proposition 6.2(1), Proposition 6.6]{Ful}, one has:
	$$\pi^* \,  i_* (\blank) = j_* \, \pi^!_{\PP(\sN_i)} (\blank) = j_* (c_1(\sO_{\PP(\sN_i)}(1)) \cap p^*(\blank)).$$

(b). Since $\iota\colon \shH_s \hookrightarrow \PP(\sE)$ is a divisor of $\sO_{\PP(\sE)}(1)$, $\iota^* \, \iota_* (\blank) = \zeta \cdot (\blank)$, and
	$$\pi_{i\,*} \, \pi_j^* \, \alpha  =  q_{i+1\,*}  \,\iota_* ( \iota^* (q_j^* \alpha) ) 
	=  q_{i+1\,*} (\zeta \cdot \zeta^j \cdot q^* \alpha) = q_{i+1\,*} q_{j+1}^* \,\alpha =\delta_{i,j} \alpha.$$

(c). Since $j^*\,j_*\,\varepsilon = c_{r-1}(\sN_j) \cap \varepsilon$. From the Euler sequence (\ref{eqn:Euler}) and $\sN_i = \sE|_Z$:
	$$ c_{r-1}(\sN_j) = \sum_{i=0}^{r-1} (-1)^i \zeta^i p^* c_{r-1-i}(\sE) = (-1)^{r-1} \zeta^{r-1} + \text{(lower order terms of $\zeta^{i}$)}.$$
Therefore $j^*\,j_*\,\varepsilon  = c_{r-1}(\sN_j) \cap \varepsilon  = 0$ and $p_{1\,*} \varepsilon  = \ldots = p_{r-1\,*} \varepsilon =0$ implies $p_{0 \,*}  \, \varepsilon= p_*(\zeta^{r-1} \cdot \varepsilon) + p_* (\text{(lower order terms of $\zeta^{i}$)} \cap \varepsilon) = \pm p_*( c_{r-1}(\sN_j) \cap \varepsilon ) +  p_* (\text{(lower order terms of $\zeta^{i}$)} \cap \varepsilon) = 0$. Hence $\varepsilon=  \sum_{i=0}^{r-1} p_i^* \, p_{i\,*} \varepsilon = 0$. 

(d)(i). Over the open subscheme $U = X \backslash Z$, $\shH_s|_{\pi^{-1}(U)} = \PP(\sG|_U)$ is a projective bundle with fiber $\PP^{r-2}$. In fact, over $U$ there is an exact sequence of vector bundles $0 \to \sO_U \to \sE|_U \to \sG|_U \to 0$. Then linear subbundle $\PP( \sG|_U) \subset  \PP(\sE|_U)$ is a divisor representing the class $\zeta = c_1 (\sO(1))$. For any $\beta \in \CH(\PP( \sG|_U))$, by Theorem \ref{thm:proj.bundle} applied to $\PP(\sG_U)$, there exists a unique $\alpha_i \in \CH(U)$, such that $\beta = \sum_{i=0}^{r-2} (\iota^*\zeta)^i \cdot \pi^* \alpha_i$. Therefore the following holds:
	$$\iota_* \, \beta =  \sum_{i=0}^{r-2} \iota_* ((\iota^*(\zeta)^i \cdot \iota^*(q^* \, \alpha)) =  \sum_{i=0}^{r-2} \zeta^{i+1} \cdot q^*\, \alpha_i,$$
where the last equality follows from the projection formula and $\iota_* \iota^* (\blank) = \zeta \cdot (\blank)$. From the uniqueness statement of Theorem \ref{thm:proj.bundle} applied to $\PP(\sE)$, we know that
	$$\alpha_i = q_{i+1\,*} \, \iota_* \, \beta  = \pi_{i \, *} \beta.$$ 
Therefore over $U$, the following holds:
	$\beta =  \sum_{i=0}^{r-2} \pi_i^*\, \pi_{i\,*} \,\beta.$
Now for any $\beta \in \CH_k(\shH_s)$, $(\beta - \sum_{i=0}^{r-2} \pi_i^*\, \pi_{i\,*} \,\beta)|_U = 0$. From the exact sequence $\CH(\PP(\sN_i)) \to \CH(\shH_s) \to \CH(\shH_{s}|_U) \to 0$, there exits an $\varepsilon \in \CH(\PP(\sN_i))$ such that $\beta - \sum_{i=0}^{r-2} \pi_i^*\, \pi_{i\,*} \,\beta = j_* \varepsilon$.

(d)(ii). From Step $(d)(i)$ we know that $\beta = j_* \, \varepsilon$, for $\varepsilon \in \CH_k(\PP(\sN_i))$. Since the ambient square of (\ref{diagram:Cayley}) is flat, by the flat base-change formula we have:
	$$i_* \, p_{i+1 \,*} = q_{i+1 \, *} \, (\iota \circ j)_* = \pi_{i \, *} \, j_*, \quad \text{for} \quad i = 0,1,\ldots, r-2.$$
Therefore $i_* \, (p_{i+1 \,*} \varepsilon) =  \pi_{i \,} * \beta = 0$ for $i \in [0,r-2]$. Notice since $\varepsilon = \sum_{i=0}^{r-1} p_i^* \, p_{i \,*} \varepsilon$, one has
	\begin{align*} j_* (p_0^*\, p_{0\,*} \varepsilon) = j_*( \varepsilon) - j_* (\zeta \cdot   \sum_{i=0}^{r-2} p_{i}^* \, p_{i+1 \,*} \varepsilon ) 
	 = j_*( \varepsilon) - \sum_{i=0}^{r-2} \pi_{i}^* i_* (p_{i+1 \,*} \varepsilon) = j_*( \varepsilon).
	\end{align*}
Here the second equality follows from the {\em key formula} of Step $(a)$. Now $j^* j_* (p_0^* \,  p_{0\,*} \varepsilon) = j^* j_*  \varepsilon =0$. By Step $(c)$, $p_0^* \,  p_{0\,*} \varepsilon =0$, hence $\beta = j_* \varepsilon =0$.
\end{proof}

Theorem \ref{thm:Cayley} follows from the above Proposition \ref{prop:caley:key-lemma} as follows:
\begin{proof}[Proof of Theorem \ref{thm:Cayley}]
The fact $gf =0$ follows from Step $(a)$. The surjectivity of $g$ is Step $(d)(i)$. By Step $(b)$, a left inverse of $f$ is given by 
	$h \colon (\oplus_{i=0}^{r-2}  \alpha_i, ~\varepsilon) \mapsto \oplus_{i=0}^{r-2} p_{i+1\,*} \varepsilon.$
In fact, $hf$ is
	$$\oplus_{i=0}^{r-2}  \gamma_i \mapsto \oplus_{i=0}^{r-2} p_{i+1\,*} (\sum_{j=0}^{r-2} \zeta^{j+1} \cdot p^* \gamma_j) =  \oplus_{i=0}^{r-2} (p_{i+1\,*} \sum_{j=0}^{r-2} p_{j+1}^* \, \gamma_i)  = \oplus_{i=0}^{r-2} \gamma_i.$$
To show the exactness of \eqref{eqn:thm:Cayley:ses}, suppose for $\alpha_i \in \CH(X)$ and $\varepsilon \in \CH(\PP(\sN_i))$,
	$\sum_{i=0}^{r-2} \pi_i^*\, \alpha_i  + j_* \varepsilon = 0.$
Then similar to Step $(d)(ii)$, from Step $(b)$, for all $i \in [0,r-2]$,
	$$\alpha_i = - \pi_{i\,*} ( j_* \varepsilon) = - i_{*} \, p_{i+1\,*} \varepsilon \in \CH(X).$$
Now consider 
	$
	\varepsilon' = \varepsilon - \sum_{i=0}^{r-2} p_{i+1}^*\, p_{i+1\,*} \varepsilon,
	$
then similar to the proof of Step $(d)(ii)$, we have
	\begin{align*} j_* \varepsilon' = j_*( \varepsilon) - j_* (\zeta \cdot   \sum_{i=0}^{r-2} p_{i}^* \, p_{i+1 \,*} \varepsilon ) 
	 = j_*( \varepsilon) - \sum_{i=0}^{r-2} \pi_{i}^* i_* (p_{i+1 \,*} \varepsilon) = j_*( \varepsilon) + \sum_{i=0}^{r-2} \pi_{i}^* \alpha_i =0,
	\end{align*}
and $p_{1\,*} \varepsilon'  = \ldots = p_{r-1\,*} \varepsilon' =0$ (since $\varepsilon'  =p_0^* \, p_{0\,*} \varepsilon$). Therefore by Step $(c)$, $ \varepsilon' =0$. Hence 
$(\oplus_i  \alpha_i, \varepsilon) = ( -\oplus_i  i_{*} \, \gamma_i, ~  \sum_{i=0}^{r-2} p_{i+1}^*\,\gamma_i)$ for $\gamma_i= p_{i+1\,*} \varepsilon$. Hence the sequence  \eqref{eqn:thm:Cayley:ses} is exact.

To prove the last statement, we show that for any $\beta \in \CH(\shH_s)$, there exists a unique $\varepsilon \in \CH(\PP(\sN_i))$, such that  $p_{1\,*} \varepsilon  = \ldots = p_{r-1\,*} \varepsilon =0$, and 
	$$\beta = \sum_{i=0}^{r-2} \pi_i^*\, \pi_{i\,*} \,\beta + j_* \,\varepsilon.$$
In fact, for any expression
	$\beta = \sum_{i=0}^{r-2} \pi_i^*\, \alpha_i + j_* \,\varepsilon,$
by replacing $\varepsilon$ by $\varepsilon - \sum_{i=0}^{r-2} p_{i+1}^*\, p_{i+1\,*} \varepsilon$ and $\alpha_i$ by $\alpha_i+ i_* (p_{i+1 \,*} \varepsilon)$, we may  assume $p_{1\,*} \varepsilon  = \ldots = p_{r-1\,*} \varepsilon =0$. Hence by the projective bundle formula, $\varepsilon = p^* \gamma$ for a unique $\gamma \in \CH(Z)$. Now by the flat base-change formula,
	$$\pi_{i\,*} (j_* \, p^* \,\gamma) = q_{i+1\,*} (\iota \circ j)_* \,p^* \,\gamma = q_{i+1\,*}  q^* (i_* \, \gamma) = 0, \qquad i \in [0,r-2].$$
Therefore $\pi_{i\,*} \beta = \pi_{i\,*} ( \sum_{i=0}^{r-2} \pi_i^*\, \alpha_i + j_* \, \, p^* \,\gamma) = \alpha_i$ for all $i=0,1,\ldots, r-2$. Hence we have established the identification:
	\begin{align*}
	\CH_{k}(\shH_s) & =\bigoplus_{i=0}^{r-2} \pi_i^* \CH_{k-(r-2)+i}(X) \oplus j_* (\CH_{k}(\PP(\sN_i))_{p_{1\,*} = \ldots = p_{r-1\,*} =0})
	 \\
	 &= \bigoplus_{i=0}^{r-2} \pi_i^* \CH_{k-(r-2)+i}(X)  \oplus  j_*\, p^*\,\CH_{k-(r-1)}(Z)
	\end{align*}
(where $\CH_{k}(\PP(\sN_i))_{p_{1*} = \ldots =p_{r-1\,*} =0}$ denotes the subgroup $\{\gamma \in \CH_{k}(\PP(\sN_i)) \mid p_{1\,*} \gamma = \ldots =p_{r-1\,*} \gamma =0 \}$ of $\CH_{k}(\PP(\sN_i) )$.) Moreover, the projection maps to first $(r-1)$-summands are respectively given by $\beta \mapsto  \alpha_i =\pi_{i\,*} \beta$, for $i = 0,1\ldots, r-2$. For the formula of the projection to the last summand, it suffices to notice that $p_* \, j^*\, \pi_i^* (\blank)= p_* (p_i^*(i^*(\blank))) = 0$ for $i \in [0, r-2]$ and that $p_* \, j^*\,j_*\,p^*(\blank) = p_* ( c_{r-1} (\Omega_{\PP(\sE)}(1)) \cap p^*(\blank)) =  (-1)^{r-1} \Id$.
\end{proof}

\begin{remark} \label{rmk:Cayley_trick_id}
If we denote $\Gamma = \PP(\sN_i) = \shH_s \times_X Z$, and $\Gamma_* \colon \CH(\shH_s) \to \CH(Z)$ (resp. $\Gamma^* \colon  \CH(Z) \to \CH(\shH_s) $) the map induced by the correspondence $[
\Gamma] \in \CH(\shH_s \times Z)$ (resp. by the transpose $[\Gamma]^{t} \in \CH( Z \times \shH_s)$ of $[\Gamma]$), then 
	$$\Gamma_* = p_* \circ j^* \quad \text{and} \quad\Gamma^* = j_* \circ p^*.$$
In the above proof, we have actually shown that the relations
	$$\Gamma_* \, \Gamma^* = (-1)^r \Id_{\CH(Z)}, \quad \pi_{i\,*} \pi_j^{*} = \delta_{i,j} \Id_{\CH(X)}, \quad \Gamma_* \pi_i^{*} =   \pi_{i\,*} \Gamma^* = 0,$$
hold for any $i,j \in [0,r-2]$, and that the isomorphism (\ref{eqn:cayley.chow}) is given by:
	$$\Id_{\CH(\shH_s)} = \sum_{i=0}^{r-2} \pi_i^{*} \, \pi_{i\,*} + \Gamma^* \, \Gamma_*.$$
\end{remark}

\begin{corollary} \label{cor:Cayley} If $X$, $\shH_s$ and $Z$ are smooth and projective varieties over some ground field $\kk$, then there is an isomorphism of  Chow motives:
	\begin{align*} \label{eqn:cayley.chow.motive}
		\left(\bigoplus_{i=0}^{r-2} h^{r-2-i} \circ \pi^* \right) \oplus [\Gamma]^t \colon \left(\bigoplus_{i=0}^{r-2} \foh(X)(i) \right) \oplus \foh(Z)(r-1) \xrightarrow{\sim} \foh(\shH_s). 
	\end{align*}
\end{corollary}
\begin{proof} By Manin's identity principle, it suffices to notice that for any smooth $T$, the schemes $Z \times T \subset X \times T$ and $\shH_s \times T$ are also in the same situation of Cayley's trick Theorem \ref{cor:Cayley}. Hence the identities of Remark \ref{rmk:Cayley_trick_id} hold for the Chow motives.
\end{proof}

\begin{example} \label{example:fano}. Let $Y \subset \PP^n$ be any complete intersection subvariety over a field $\kk$ of codimension $c \ge 1$, say cut out by a regular section of the vector bundle $\bigoplus_{i=1}^c \sO_{\PP^n}(d_i)$. Following \cite{KKLL}, if we fix a positive integer $r \ge \max\{\sum d_i - n -c, 1-c\}$, then $Y \subset \PP^n \subset \PP^{n+r}=X$ is the zero subscheme of a regular section $s$ of the ample vector bundle
	$$\sE :=  \sO_{ \PP^{n+r}}(1)^{\oplus r} \oplus \bigoplus_{i=1}^c \sO_{ \PP^{n+r}}(d_i).$$
It is shown in \cite{KKLL} that $F_Y := \shH_s \subset \PP(\sE)$ is a Fano variety. Theorem \ref{thm:Cayley} implies:
	$$\CH_*(F_Y) = \CH_{*-r-c+1}(Y) \oplus \bigoplus_{i=0}^{r+c-2} \CH_{*-r-c+2+i}(\PP^{n+r}),$$
and similarly for Chow motives if we assume $Y$ is smooth. Hence the Chow group (resp. motive, rational Hodge structure if $\kk \subset \CC$ and $Y$ smooth) of every complete intersection $Y$ can be split embedded into that of a Fano variety $F_Y$, with complement given by copies of the Chow group (resp. motive, rational Hodge structure) of a projective space $\PP^{n+r}$.
\end{example}

\subsection{Standard flips} \label{sec:flips}
Let $(\Psi, \psi):(X,P) \to (\overline{X}, S)$ be a log-extremal contraction such that 
	\begin{enumerate}
		\item[(i)] $P = \PP_{S, \,\rm sub}(F)$ for a vector bundle $F$ of rank $n+1$ on $S$;
		\item[(ii)] Over every $s \in S$, $(\sN_{P/X})|_{P_s} \simeq \sO_{\PP^n}(-1)^{\oplus(m+1)}$ for some fixed integer $m$.
	\end{enumerate}
By (the same argument of) \cite[\S 1]{LLW}, there exists a vector bundle $F'$ of rank $m+1$ such that $\sN_{P/X} = \sO_{\PP_{\rm sub}(F)}(-1) \otimes \psi^* F'$. If we blow up $X$ along $P$, we get $\pi \colon \widetilde{X} \to X$ with exceptional divisor $E = \PP_{\rm sub}(\sN_{P/X} ) = \PP_{S, \, \rm sub}(F) \times_S \PP_{S, \, \rm sub}(F')$. Furthermore, one can blow down $E$ along fibres of $\PP_{S, \,\rm sub}(F)$ and get $\pi' \colon \widetilde{X} \to X'$ and $\pi'(E) = :P' \simeq \PP_{S, \, \rm sub}(F')$, with $\sN_{P'/X'} \simeq  \sO_{\PP_{\rm sub} (F')}(-1) \otimes \psi'^* F$, where $\psi' \colon P' \to S$ the natural projection. Hence we obtain another log-extremal contraction $(\Psi', \psi'):(X',P') \to (\overline{X}, S)$ which is birational to $(X,P)$. 

The above birational map $f: X \dashrightarrow X'$ is called a {\em standard (or ordinary) flip} of type $(n,m)$. Note that $X >_{K} X'$ (resp. $X \simeq_K X'$) if and only if $n>m$ (resp. $n=m$). 

The geometry is illustrated in the following diagram, with maps as labeled:
\begin{equation*}
\begin{tikzcd}[row sep= 1.2 em, column sep = 2 em]
	 & & E= P\times_S P' \ar[hook]{d}[swap]{j} \ar{lldd}[swap]{p} \ar{rrdd}{p'}& & \\
	 & & \widetilde{X} \ar{ld}[swap]{\pi} \ar{rd}{\pi'} & & \\
	P=\PP_{S, \,\rm sub}(F) \ar[hook]{r}{i} \ar{rrdd}[swap]{\psi}&X  \ar{rd}[swap]{\Psi}& &X'  \ar{ld}{\Psi'} &P' =\PP_{S, \,\rm sub}(F')\ar[hook']{l}[swap]{i'}   \ar{lldd}{\psi'} \\
	& & \overline{X} &&\\
	& & S \ar[hook]{u} & &
\end{tikzcd}
\end{equation*}

If $X >_{K} X'$ (resp. $X \simeq_K X'$), the expected relations of derived categories for the flip (resp. flop) $f\colon X \dashrightarrow X'$ are established by Bondal--Orlov \cite{BO}. In this section we establish the corresponding relations on Chow groups, which complements the results of \cite[\S 3]{LLW}.

From now on we assume $n \ge m$, i.e. $X \ge_K X'$. Denote $\Gamma$ the graph closure of $f$ in $X' \times X$, which is nothing but $\widetilde{X} = X \times_{\overline{X}} X'$. Denote by $\Gamma_*\colon  \CH_k(X) \to \CH_k(X')$  and $\Gamma^* \colon \CH_k(X') \to \CH_k(X)$ the maps induced by $[\Gamma] \in \CH_{\dim X}(X \times X')$. It is easy to see 
	\begin{align*}
	\Gamma_*(\blank) =   \pi'_* \, \pi^*(\blank)  \quad \text{and} \quad \Gamma^*(\blank)  =  \pi_* \, \pi'^*(\blank).	
	\end{align*}
 Denote by $\sV$ and $\sV'$ the respective excess bundles for the blow up $\pi \colon \widetilde{X} \to X$ and $\pi' \colon \widetilde{X} \to X'$, i.e. they are defined by the short exact sequences:
	\begin{align*}
	0 \to \sN_{E/\widetilde{X}} \to p^*\sN_{P/X} \to \sV \to 0, \qquad 0 \to \sN_{E/\widetilde{X}} \to p'^*\sN_{P'/X'} \to \sV' \to 0.
	\end{align*}
Denote $\Phi_* \colon \CH_k(P) \to \CH_k(P')$ respectively $\Phi^* \colon \CH_k(P') \to \CH_k(P)$ the maps given by correspondence $c_{m}(\sV) \in \CH^{m}(P \times_S P')$ respectively $c_{n}(\sV') \in \CH^{n}(P' \times_S P)$, i.e.
	\begin{align*}
	\Phi_*(\blank) =p'_*(c_{m}(\sV)\cap p^*(\blank)) \quad \text{and} \quad \Phi^*(\blank) = p_*(c_{n}(\sV')\cap p'^*(\blank)).
	\end{align*}
It follows from Euler sequence that $\sV = \sO_{P}(-1) \boxtimes \shT_{P'/S}(-1)$ and $\sV' = \shT_{P/S}(-1) \boxtimes \sO_{P'}(-1)$. 

\begin{theorem}[Standard flips]\label{thm:flip} Let $f: X \dashrightarrow X'$ be a standard flip as above, and assume $X'$ is non-singular and quasi-projective. Then 
\begin{enumerate}
	\item \label{thm:flip-1}
	The following holds: 
		$$\Gamma_* \, \Gamma^* = \Id_{\CH(X')} \quad \text{and} \quad \Phi_* \, \Phi^* = \Id_{\CH(P')}.$$
	\item \label{thm:flip-2}
	There exists a split short exact sequence:
	\begin{align*}
		0 \to \CH_{k}(P') \xrightarrow{(\Phi^*, -i'_*)} \CH_{k}(P) \oplus \CH_{k}(X') \xrightarrow{(\gamma, \alpha')\mapsto i_*\gamma + \Gamma^* \alpha'} \CH_k(X) \to 0,
	\end{align*}
where a left inverse of the first map is given by $(\gamma, \alpha')\mapsto \Phi_* \,\gamma$.
	\item \label{thm:flip-3}
	The above exact sequence induces an isomorphism of Chow groups
	\begin{align} \label{eqn:flips.chow}
		 \CH_{k}(X') \oplus \bigoplus_{i=0}^{n-m-1} \CH_{k-n+i}(S)  \xrightarrow{\sim} \CH_{k}(X), 
	\end{align}
given by $(\alpha', \oplus_{i=0}^{n-m-1}  \beta_i) \mapsto \Gamma^* \, \alpha' + i_*(\sum_{i=0}^{n-m-1} \zeta^i \cdot \psi^*\beta_i)$. Furthermore, in the above decomposition, the projection to the first summand is given by $\alpha \mapsto \alpha' = \Gamma_* \alpha$.
\end{enumerate}
\end{theorem}

Notice that in the flop case $m=n$, this result recovers the invariance of Chow groups under flops in \cite{LLW}; and in the flip case $m<n$, this theorem completes the discussion of \cite[\S 2.3]{LLW} by providing the complementary summands of the image of $\Gamma^*$ in the Chow group $\CH(X)$. Finally, as a blowup can be viewed as a standard flip of type $(n,0)$, above theorem recovers the blowup formula Theorem \ref{thm:blowup}.

\begin{proof}[Proof of first part of \eqref{thm:flip-1}] The equality $\Gamma_* \, \Gamma^* = \Id$ follows exactly the same line of proof of \cite[Theorem 2.1]{LLW}, as already mentioned in \cite[\S 2.3]{LLW}. We sketch the proof here for completeness. For any class $[W'] \in \CH_k(X')$, by Chow's moving lemma, (if allowing negative coefficients) we may assume it is represented by a cycle $W'$ which intersects $P'$ transversely. Therefore $\pi'^*[W'] = [\widetilde{W}]$ by \cite[Corollary 6.7.2]{Ful}, where $\widetilde{W}$ is the blowup of $W'$ along $W' \cap P'$. Hence $\Gamma^*[W'] = \pi_* [\widetilde{W}] = [W]$, where $W$ is the image of $\widetilde{W}$, and is also the proper transform of $W'$ along the birational rational map $f^{-1}$. Now we have:
	$$\pi^* [W] = [\widetilde{W}] + j_* \sum_{B} [E_{B}],$$ 
where let $B' \subset W' \cap P'$ be a component, then $E_{B} \subset E$ are $k$-cycles supported over components $\overline{B}  = \psi'(B') \subset \psi'(W' \cap P') \subset S$. A direct computation of dimensions shows that, for a general point $s$, the fiber $E_{B,s}$ over $s$ has dimension:
	$$\dim E_{B,s} \ge \dim E_{B}  - \dim (\overline{B}) \ge \dim E_{B}  - \dim ({B'}) = k - (k - (n+1)) = n+1.$$
Now $E_{B,s}$ must contain positive fibers of of $p'_s \colon \PP_s^n \times \PP_s^m \to \PP_s^m$, as $n+1 > n \ge m$. Hence $\pi_* j_*[E_B] = p_*[E_B] = 0$, and $\Gamma_*\, \Gamma^* [W'] = \pi'_* \, \pi^* \, [W]  =\pi'_* [\widetilde{W}] = [W']$. 
\end{proof}

\begin{remark} Notice that the above argument does {\em not} work in the other direction for $\Gamma^* \Gamma_* [V]$, where $[V] \in \CH_k(X)$. The reason is as follows: the fibre $E_{B',s}$ of the $k$-cycle $E_{B'}$ in $\pi'^*\Gamma^* [V] = \pi'^*[V'] = \widetilde{V} + j_* \sum E_{B'}$ has dimension $\ge m+1$, but $m \le n$, thus the fiber $E_{B',s}$ is not necessarily contracted by $p_*$. However, if $k \le m$ (in which case we may assume $V \cap P = \emptyset$) or $k \ge n+1 + \dim \psi(V \cap S)$ (for example, if $k \ge \dim S + n +1$), then the above argument still works, i.e., the following holds:
	$$\Gamma^* \, \Gamma_* [V] = [V] \quad \text{if} \quad k \le m+1 \text{~or~} k \ge n + \dim \psi(V \cap S).$$
For the intermediate cases $m+1\le k \le n + \dim \psi(V \cap S) \le n + \dim S$, the same argument only implies $\Gamma^* \, \Gamma_* [V] = [V] + i_* \sum_{Z \subset P} [Z]$ for certain cycles $Z \subset P$ supported on $P$; These cycles will be precisely explained by the statements \eqref{thm:flip-2} and \eqref{thm:flip-3} of the theorem.
\end{remark}

\begin{proof}[Proof of second part of \eqref{thm:flip-1}] 
It follows from Lemma \ref{lem:chern} that:
	\begin{align*}
	c_{m}(\sV) &= (-1)^m c_m(\sO_{P}(1) \boxtimes \Omega_{P'/S}(1)) = (-1)^m \sum_{t=0}^{m} \zeta^t \cdot c_{m-t}(\Omega_{P'/S}(1))\\
	&=  (-1)^m \sum_{t=0}^{m} (-1)^{m-t} \sum_{s=0}^{m-t} c_{s}(F') \cdot  (\zeta')^{m-s-t} \cdot \zeta^t,
	\end{align*}
and
	\begin{align*}
	c_{n}(\sV') &= (-1)^n c_n(\Omega_{P/S}(1)\boxtimes \sO_{P'}(1)  ) = (-1)^n \sum_{j=0}^{n}  c_{n-j}(\Omega_{P/S}(1))\cdot (\zeta')^j \\
	&=  (-1)^n \sum_{j=0}^{n} (-1)^{n-j} \sum_{i=0}^{n-j} c_{i}(F) \cdot  \zeta^{n-i-j} \cdot (\zeta')^{j}.
	\end{align*}
The map $\Phi_* \circ \Phi^*$ is given by the convolution of correspondences 
	$$c_{m}(\sV) *c_{n}(\sV') := p_{13*}(p_{12}^*(c_{n}(\sV')) \cdot p_{23}^* (c_{m}(\sV)) ) \in \CH^{m}(P' \times_S P'),$$
 where $p_{ij}$ are the obvious projections from $P' \times_S P \times_S P'$ to the corresponding factors; The cohomological degree $m$ is computed via $m+n -\dim (P/S) = m$. To avoid confusion, we denote the product $P' \times_S P \times_S P'$ by $P_1' \times_S P \times_S P_2'$, and denote the relative $\sO(1)$-classes of $P_1'$ and $P_2'$ by $\zeta'_1$ and $\zeta'_2$ respectively. Therefore
	\begin{align*}
	c_{m}(\sV) *c_{n}(\sV') = p_{13*}\left(\sum_{j=0}^{n} \sum_{t=0}^{m} (-1)^{j+t} \sum_{s=0}^{m-t} \sum_{i=0}^{n-j} c_{s}(F_2') \cdot c_{i}(F) \cdot \zeta^{n+t-i-j}  \cdot(\zeta_1')^{j} \cdot (\zeta_2')^{m-s-t} \right).
	\end{align*}
Since $p_{13*}(\zeta^k) = 0$ for all $0 \le k \le n-1$, in the above expression, the only terms inside the parentheses that could survive $p_{13*}$ are the ones whose indices satisfy $t-i-j \ge 0$. Thus we may assume the indices of the summation satisfy $j \le t \le m$ and $0 \le i \le t-j$. From the definition of the Segre class of $F$, we have $p_{13*}(\zeta^{n+k}) = s_k(F)$, hence:
	\begin{align*}
	c_{m}(\sV) *c_{n}(\sV') = \sum_{j=0}^{m} \sum_{t=j}^{m} (-1)^{j+t} \sum_{s=0}^{m-t}  \sum_{i=0}^{t-j} c_{i}(F) \cdot s_{t-i-j}(F) \cdot  c_{s}(F_2') \cdot(\zeta_1')^{j} \cdot (\zeta_2')^{m-s-t}.
	\end{align*}
From $c(F) s(F) = 1$, we know that $\sum_{i=0}^{t-j} c_{i}(F) \cdot s_{t-i-j}(F) = 0$ unless $t=j$, in which case $c_0(F) s_0(F) =1$. Hence above expression reduces to
	\begin{align*}
	c_{m}(\sV) *c_{n}(\sV') &= \sum_{j=0}^{m}  \sum_{s=0}^{m-j}  c_{s}(F_2') \cdot(\zeta_1')^{j} \cdot (\zeta_2')^{m-j-s} = \sum_{j=0}^{m} c_{m-j}(\shT_{P_2'/S}(-1)) \cdot(\zeta_1')^{j} \\
	& = c_m (\sO_{P_1'}(1) \boxtimes \shT_{P_2'/S}(-1)).
	\end{align*}
(For the second equality, we used Lemma \ref{lem:chern}.) On the other hand, the diagonal $\Delta_{P'} \subset P' \times_S P'$ is the zero locus of a regular section $s$ of the rank $m$ vector bundle $\sO_{P'}(1) \boxtimes \shT_{P'/S}(-1)$; The section $s$ under the canonical identification
	$$\Gamma(P' \times_S P', \sO_{P'}(1) \boxtimes \shT_{P'/S}(-1)) = \Gamma(S, F'^\vee \otimes F') = \Hom_S(F',F')$$
corresponding to $1_{F'} \colon F' \to F'$; Hence $ [\Delta_{P'}] = c_m(\sO_{P'}(1) \boxtimes \shT_{P'/S}(-1))$. Therefore
	$$c_{m}(\sV) *c_{n}(\sV')  = [\Delta_{P'}], \quad \text{hence} \quad \Phi_* \, \Phi^* = \Id_{\CH(P')}.$$
\end{proof}

Before proceeding the rest of the proof of Theorem \ref{thm:flip}, we study more about the maps $\Phi_*$ and $\Phi^*$. First, notice that the projective bundle formula Theorem \ref{thm:proj.bundle} can be regarded as equipping $\CH(P)$ and $\CH(P')$ with natural ``free module structures {\em over} $\CH(S)$".


\begin{lemma} \label{lem:flip:Phi}
	\begin{enumerate}
	\item	\label{lem:flip:Phi-1}
	 The maps  $\Phi_* \colon \CH(P) \to \CH(P')$ and $\Phi^* \colon \CH(P') \to \CH(P)$ are ``$\CH(S)$-linear", i.e. for all $\alpha \in \CH^*(P)$, $\alpha' \in \CH^*(P')$, $\theta \in \CH_*(S)$,
		$$\Phi_*(\alpha \cap \psi^* \theta) = \Phi_*(\alpha) \cap \psi'^* \theta, \qquad \Phi^*(\alpha' \cap \psi'^* \theta) = \Phi^*(\alpha') \cap \psi^* \theta.$$
	\item	 \label{lem:flip:Phi-2}
	Consider the following ``sub-$\CH(S)$-modules" of $ \CH(P)$:
		\begin{align*}
			& \CH(P)_{m} : = \Span\{\zeta^{n-m}, \zeta^{n-m+1}, \ldots, \zeta^{m} \} =  \zeta^{n-m} \cdot \CH(S) \oplus \cdots \oplus  \zeta^{m} \cdot \CH(S) \subset \CH(P); \\
			& \CH(P)_{\Phi_*=0} : = \Span \{ 1, \zeta, \ldots, \zeta^{n-m-1} \}  = 1 \cdot \CH(S) \oplus \cdots \oplus  \zeta^{n-m-1} \cdot \CH(S) \subset \CH(P).
		\end{align*}
	Then $\Phi_*$ is injective on $\CH(P)_{m}$, and with image $\Phi_*(\CH(P)_{m}) = \CH(P')$; Furthermore, $\Phi_*$ is zero on $\CH(P)_{\Phi_*=0}$.
	\item \label{lem:flip:Phi-3}
	$\Phi^*$ is injective, and its image $\Im(\Phi^*)$ satisfies $\CH(P)_{\Phi_*=0}  \simeq \CH(P)/\Im(\Phi^*)$.
	\end{enumerate}
\end{lemma}

\begin{proof} The statement \eqref{lem:flip:Phi-1} follows directly from the projection formula \cite[Theorem 3.2(c)]{Ful} and Theorem \ref{thm:proj.bundle} \eqref{thm:proj.bundle-1}. For \eqref{lem:flip:Phi-2}, notice that for any $0 \le i \le m$,
	\begin{align*}
		\Phi_* (\zeta^{n-m+i})& = p'_* (c_{m}(\sV) \cup p^*\zeta^{n-m+i})  \\ 
		&=  p'_*\left((-1)^m \sum_{t=0}^{m} (-1)^{m-t} \sum_{s=0}^{m-t} c_{s}(F') \cdot  (\zeta')^{m-s-t} \cdot \zeta^{n-m+i+t}\right)  \\
		& = \sum_{t=m-i}^{m} (-1)^{t} \sum_{s=0}^{m-t} c_{s}(F') \cdot s_{t-(m-i)}(F)  \cdot (\zeta')^{m-s-t} \\
		& =  (-1)^{m-i}  (\zeta')^{i} + (\text{lower order terms}). 
		\end{align*}
(For example, $\Phi_* (\zeta^m) = (\zeta')^{n} + (\text{lower order terms})$ and $\Phi_* (\zeta^n) = (-1)^n (\zeta')^{0}$.) This computation together with  \eqref{lem:flip:Phi-1} shows that $\Phi_* (\zeta^{n-m+i} \cap \psi^* \theta) =  (\pm \zeta'^{i} + (\text{lower order terms})) \cap \psi'^* \theta$ for all $\theta \in \CH(S)$, which implies the injectivity of $\Phi_*$ on $\CH(P)_{m}$; The same computation in the case $i<0$ shows that $\Phi_*$ is zero on $\CH(P)_{\Phi_*=0}$. A similar computation shows:
	\begin{align*}
		\Phi^*(\zeta'^i) = \pm \zeta^{n-m+i} + (\text{lower order terms})
	\end{align*}
for $0 \le i \le m$, which implies \eqref{lem:flip:Phi-3}. \end{proof}


\begin{lemma}\label{lem:flip:van} For any $\gamma \in \CH(P)$, if $\Phi_*(\gamma) = 0$ and $i^* \, i_* \gamma =0$. Then $\gamma =0$.
\end{lemma}

\begin{proof}  Let $\gamma = \sum_{i=0}^{n} \zeta^i \cdot \psi^* \theta_i$ for $\theta_i \in \CH(S)$, then from above lemma, $\Phi_* \gamma = 0$ implies $\theta_{n-m} = \theta_{n-m+1} = \ldots = \theta_{n} = 0$. On the other hand, $i^* \, i_* \gamma = c_{m+1}(\sN_{P/X}) \cap \gamma = c_{m+1}(F'\otimes \sO_P(-1)) \cap \gamma = 0$, and 
	\begin{align*}
	c_{m+1}(F'\otimes \sO_P(-1)) & = \sum_{i=0}^{m+1} c_{m+1-i}(F') (-\zeta)^i \\
	&= (-1)^{m+1} \zeta^{m+1} + (\text{lower order terms}).  
	\end{align*}
Hence by the uniqueness of an expression of the form $\sum_{i=0}^{n} \zeta^i \cdot \psi^*(\blank)$, one can inductively show $\theta_{n-m-1} = 0$, $\theta_{n-m-2} = 0$, $\ldots$, $\theta_0 = 0$. Therefore $\gamma=0$.
\end{proof}

\begin{proof}[Proof of \eqref{thm:flip-2} of Theorem \ref{thm:flip}] To show the sequence of \eqref{thm:flip-2} is a complex, simply observe that for any $\gamma' \in \CH(P)$,
	\begin{align*}
	 i_*\,\Phi^*\,\gamma' + \Gamma^*(-i'_*\, \gamma') &= i_*\,p_*\,(c_n(\sV')\cap p'^*\gamma') - \Gamma^*\,i'_*\, \gamma' \\
	 &= \pi_*\,j_*\,(c_n(\sV')\cap p'^*\gamma')  - \Gamma^*\,i'_*\, \gamma'  \\
	& \overset{\text{(k.f.)}}{=\mathrel{\mkern-3mu}=}  \pi_*\,\pi'^*\,i'_*\, \gamma' - \pi_*\,\pi'^*\,i'_*\, \gamma' = 0.
	\end{align*}
(In the above equality and later, (k.f.) means by the key formula (\ref{eqn:k.f.blowup}) of blowup.)

For any $\alpha \in \CH_k(X)$, $\alpha - \Gamma^* \, \Gamma_* \, \alpha = 0$ on $\CH_k(X \backslash P)$. Then from the exact sequence $\CH_k(P) \to \CH_k(X) \to \CH_k(X\backslash P) \to 0$, there exists an element $\gamma \in \CH_k(P)$ such that $\alpha = \Gamma^* \, \Gamma_* \, \alpha + i_* \, \gamma$. This establishes the surjectivity of the last map of the sequence of \eqref{thm:flip-2}. The injectivity of the first map and the left inverse statement follow directly from $\Phi_* \, \Phi_* = \Id$. To show the sequence is exactness in the middle, assume 
$(\gamma, \alpha') \in \CH(P)\oplus \CH(X')$ such that $i_*\gamma + \Gamma^* \alpha'=0$, and we want to find $\gamma'$ such that $(\gamma, \alpha') = (\Phi^* \gamma', -i'_* \gamma)$. Since
	\begin{align*}
	\alpha'  & = \Gamma_* \Gamma^* \alpha'  = -  \Gamma_* \, i_* \,\gamma = - \pi'_* \, \pi^* \,  i_* \,\gamma \overset{\text{(k.f.)}}{=\mathrel{\mkern-3mu}=}  -\pi'_* j_* (c_{m}(\sV) \cap p^*\,\gamma) \\
	 & =  -i'_* \, p'_*(c_{m}(\sV) \cap  p^*\, \gamma) = -i'_* \, \Phi_* (\gamma).
 	\end{align*}
Define
	$\gamma_0 = \gamma - \Phi^* \Phi_* \gamma.$
The goal is to show $\gamma_0 =0$. Notice that 
	\begin{align*} \Phi_* \gamma_0 &= \Phi_* \gamma -  \Phi_*  \Phi^* \Phi_* \gamma = 0, \qquad \text{and also} \\
	i_*\,\gamma_0 & = i_* \,\gamma - i_* \, \Phi^* \Phi_* \gamma =  i_* \,\gamma - i_* \, p_* (c_n{\sV'} \cap p'^* (\Phi^* \, \gamma)) \\
	& = i_* \,\gamma - \pi_*\, j_*\,(c_n{\sV'} \cap p'^* ( \Phi_* \, \gamma))  
	 \overset{\text{(k.f.)}}{=\mathrel{\mkern-3mu}=} i_* \,\gamma -  \pi_* \,\pi'^*\, i'_* \, \Phi_* \, \gamma  \\
	 & =  i_* \,\gamma - \Gamma^* (i'_* \, \Phi_* \, \gamma) =  i_* \,\gamma  + \Gamma^* \, \alpha' =  0.
	\end{align*}
From Lemma \ref{lem:flip:van}, $\gamma_0 = 0$, hence $(\gamma, \alpha') = (\Phi^* \gamma', -i'_* \gamma)$ for $\gamma' = \Phi_* \, \gamma$.
\end{proof} 

\begin{proof}[Proof of \eqref{thm:flip-3} of Theorem \ref{thm:flip}] As before, from Lemma \ref{lem:flip:Phi} \eqref{lem:flip:Phi-3}  and the exact sequence of \eqref{thm:flip-2}, we obtain that for any $\alpha \in \CH_{k}(X)$, there exists an $\alpha' \in \CH_{k}(X)$ and a unique $\gamma \in \CH_k(P)$, such that $\Phi_* \gamma = 0$ and $\alpha = \Gamma^* \alpha' + i_* \gamma$. Further notice that 
	\begin{align*}
		\Gamma_*  i_* \gamma = \pi'_*\pi^*i_*\gamma \overset{\text{(k.f.)}}{=\mathrel{\mkern-3mu}=} \pi'_*j_*(c_n(\sV) \cap p^*\gamma) 
		= i'_*p'_* j_*(c_n(\sV) \cap p^*\gamma)  = i'_* \Phi_* \gamma = 0.
	\end{align*}
Therefore $\alpha' = \Gamma_* \alpha$. Hence we have established: 
	\begin{align*}
		\CH_{k}(X) &= \Gamma^* \, \CH_{k}(X') \oplus i_* (\CH(P)_{\Phi_* = 0}) \\
		& = \Gamma^*\, \CH_{k}(X') \oplus i_*(\bigoplus_{i=0}^{n-m-1} \zeta^i \cdot \psi^* \CH_{k-n+i}(S)),
	\end{align*}
and the projection to the first summand is given by $\alpha \mapsto \alpha' = \Gamma_* \alpha$. 
\end{proof}

We could also write down explicitly the projectors to the last $(n-m)$ summands; We omit the details here as we will not need them. As before, by Manin's identity principle:
\begin{corollary}\label{cor:flip} If $X$ and $X'$ are smooth and projective over some  ground field $\kk$, then there is an isomorphism of Chow motives over $\kk$:
		\begin{align*} \label{eqn:flips.motive}
		[\Gamma]^t \oplus \left( \bigoplus_{i=m+1}^n i_* \circ h^{n-i} \circ \psi^* \right) \colon  \foh(X') \oplus \left( \bigoplus_{i=m+1}^{n} \foh(S)(i) \right) \xrightarrow{\sim} \foh(X).
	\end{align*}
\end{corollary}
As before, the blowup formula for Chow motives of \cite{Manin} could be viewed as the case $m=0$ of the above corollary, as a blowup can be viewed as a standard flip of type $(n,0)$.

\newpage
\section{Main results}
Let $\sG$ be a coherent sheaf of homological dimension $\le 1$ on $X$ (i.e., $X$ is covered by open subschemes $U \subset X$ over which there is a resolution $\sF \xrightarrow{~\sigma~} \sE \twoheadrightarrow \sG$ such that $\sF$ and $\sE$ are locally free of rank $m$ and $n$ respectively, and $\sG = {\rm Coker}(\sigma)$ is of rank $r= n-m \ge 0$).
Denote the projection by $\pi \colon \PP(\sG) \to X$. Similar to the projective bundle case, for any $i \in [0,r-1]$, denote by $\pi_{i}^* \colon \CH_{k-(r-1)+i}(X) \to \CH_{k}(\PP(\sG))$ be map $\pi_i^*(\blank) = \zeta^i \cdot \pi^*(\blank)$, where $\zeta = c_1(\sO_{\PP(\sG)}(1))$. Consider the fiber product
	$$\Gamma : = \PP(\sG) \times_X \PP(\sExt^1(\sG,\sO_X)).$$
Denote the projections by $r_+ \colon \Gamma \to \PP(\sG)$ and $r_- \colon \Gamma \to \PP(\sExt^1(\sG,\sO_X))$. As before, we denote $\Gamma_*\colon  \CH_{k-r}( \PP(\sG)) \to \CH_k(\PP(\sExt^1(\sG,\sO_X)))$ and $\Gamma^* \colon \CH_k(\PP(\sExt^1(\sG,\sO_X))) \to \CH_{k-r}(\PP(\sG))$ the maps induced by the correspondence $[\Gamma] \in \CH(\PP(\sG) \times \PP(\sExt^1(\sG,\sO_X)))$, i.e. 
	\begin{align*}
	\Gamma_*(\blank) =   r_{-\,*} \, r_+^*(\blank)  \quad \text{and} \quad \Gamma^*(\blank)  =   r_{+\,*} \, r_{-}^* (\blank).	
	\end{align*}

The main result of this paper is the following:
\begin{theorem} \label{thm:main} Let $X$ be a Cohen--Macaulay scheme of pure dimension, and let $\sG$ be a coherent sheaf of rank $r \ge 0$ on $X$ of homological dimension $\le 1$. Assume either
	\begin{enumerate}
		\item[(A)] $\PP(\sG)$ and $\PP(\sExt^1(\sG,\sO_X))$ are non-singular and quasi-projective, and 
		\begin{equation} \label{eqn:weak.dim}
	{\rm codim} (X^{\ge r+1}(\sG) \subset X) = r+1, \quad  {\rm codim} (X^{\ge r+i}(\sG) \subset X) \ge  r+2i ~\text{if}~ i \ge 1; ~\text{or}
		\end{equation}
		\item[(B)]  ${\rm codim} (X^{\ge r+i}(\sG) \subset X) =i(r+i)$ (the expected codimension), for all $i \ge 1$.
	\end{enumerate}
Then for any $k \ge 0$, there is an isomorphism of Chow groups:
	\begin{align}\label{eqn:main.thm}
	 \bigoplus_{i=0}^{r-1} \CH_{k-(r-1)+i}(X)  \oplus \CH_{k-r}(\PP(\sExt^1(\sG,\sO_X))) \xrightarrow{\sim} \CH_{k}(\PP(\sG))
	\end{align}
given by $(\oplus_{i=0}^{r-1} \alpha_i,\gamma) \mapsto   \beta = \sum_{i=0}^{r-1} c_1(\sO_{\PP(\sG)}(1))^i \cap \pi^*\alpha_i + \Gamma^* \gamma$. The projection $\beta \mapsto \alpha_i$ is given by the map $\pi_{i\,*}$ of  Lemma \ref{lem:pi_proj}, $0 \le i \le r-1$, and the projection $\beta \mapsto \gamma$ is given by $(-1)^r \Gamma_*$.
\end{theorem}

\begin{remark} \label{remark:exp:dim}
\begin{enumerate}[ leftmargin=*, label=(\roman*)]
	\item \label{remark:exp:dim-1} If $X$ is irreducible, then the dimension condition  (\ref{eqn:weak.dim}) of (A) is equivalent to the requirement that $\PP(\sExt^1(\sG,\sO_X))$ maps birationally to $X^{\ge r+1}(\sG)$, and $\PP(\sG)$, $\PP(\sExt^1(\sG,\sO_X))$ and $\Gamma$
are irreducible and have expected dimensions: 
	$$ \dim  \PP(\sG) = \dim X -1 + r, \qquad \dim \PP(\sExt^1(\sG,\sO_X)) = \dim X -1 -r; $$
	$$ \dim \Gamma = \dim X - 1.$$
 
 	\item \label{remark:exp:dim-2} The only place that we need $\PP(\sG)$ and $\PP(\sExt^1(\sG,\sO_X))$ to be nonsingular and quasi-projective in (A) is to use Chow's moving lemma. Hence the result holds as long as Chow's moving lemma holds for $\PP(\sG)$ and $\PP(\sExt^1(\sG,\sO_X))$.
 
 	 \item  \label{remark:exp:dim-3} It follows from \cite[Theorem 3.4]{JL18} that if $X$ is nonsingular, and $\PP(\sG)$, $\PP(\sExt^1(\sG,\sO_X))$ have expected dimension. Then $\PP(\sG)$ is nonsingular if and only if $\PP(\sExt^1(\sG,\sO_X))$ is. 
 
	\item  \label{remark:exp:dim-4}  The requirement $r+2i$ for codimension in (A) is much weaker than the expected one $i(r+1)$ if $i >>1$, which is required by (B). On the other hand, (B) only requires very weak regularity conditions on the schemes -- $X$ being Cohen-Macaulay. (In fact, the Cohen-Macaulay condition can be dropped, as long as each stratum $X^{\ge i}(\sG) \backslash X^{\ge i+1}(\sG) \subset X$ is a regular immersion of expected dimension.) 
 \end{enumerate}
 \end{remark}
 
\begin{corollary} \label{cor:main.motive} If  $\PP(\sG)$, $\PP(\sExt^1(\sG,\sO_X))$ and $X$ are smooth and projective over some ground field $\kk$, then there is an isomorphism of Chow motives:
	\begin{align*} \label{eqn:cayley.chow.motive}
		\bigoplus_{i=0}^{r-1} h^{r-1-i} \circ \pi^* \oplus [\Gamma]^t \colon \bigoplus_{i=0}^{r-1} \foh(X)(i) \oplus \foh(\PP(\sExt^1(\sG,\sO_X)))(r) \xrightarrow{\sim} \foh(\PP(\sG)). 
	\end{align*}
\end{corollary}
\begin{proof} Similarly as Corollary \ref{cor:Cayley}, for any smooth $T$,  the same constructions and the theorem applies to $X \times T$ and $\sG \boxtimes \sO_T$,   hence in particular the identities $\Id = \Gamma^*\,\Gamma_* + \sum \pi_i^* \pi_{i\,*}$, $\Gamma_* \Gamma^* = \Id$, $\pi_{i\,*} \,\pi_i^* = \Id$, etc (see Lemma \ref{lem:pi_proj} and Lemma \ref{lem:surjectivity} below) hold for all $X \times T$ and $\sG \boxtimes \sO_T$. Then the result follows from Manin's identity principle.
\end{proof}

Before proceeding with the proofs of the theorem, we first explore some general facts.  

\begin{lemma} \label{lem:pi_proj}  
\begin{enumerate} 
	\item \label{lem:pi_proj-1}
	  Define $\pi_{i\,*} \colon \CH_{k}(\PP(\sG)) \to \CH_{k-(r-1)+i}(X)$ the same way as (\ref{eqn:lem:proj}):
\begin{align*}
		\pi_{i\,*} (\blank):= \sum_{j=0}^{r-1-i} (-1)^jc_{j}(\sG) \cap \pi_* (\zeta^{r-1-i-j} \cdot (\blank) ), \quad \text{for} \quad i = 0, 1, \ldots , r-1.
		\end{align*}
	 Then the maps $\pi_{i\,*}$ and $\pi_j^*$ satiafy
		$$\pi_{i\,*} \,\pi_i^* = \Id_{\CH(X)}, \qquad \pi_{i\,*} \,\pi_j^* = 0, \text{~if~} i \ne j, ~ i,j \in [0,r-1].$$
	\item \label{lem:pi_proj-2}
	In the local situation $\sG = {\rm Coker}(\sF \xrightarrow{\sigma} \sE)$, denote $q_{i\,*}$ the corresponding projection functor for the projective bundle $q\colon \PP(\sE) \to X$ defined by (\ref{eqn:lem:proj}), denote $\iota \colon \PP(\sG) \hookrightarrow \PP(\sE)$ the natural inclusion, then the following holds:
		$$\pi_{i\,*}(\blank)  = \sum_{j=0}^{r-1-i} (-1)^js_j(\sF) \cdot q_{m+i+j \, *}(\iota_* (\blank)).$$
\end{enumerate}
\end{lemma}
If we consider the following subgroup
	$$\CH_{k}(\PP(\sG))_{\rm tor.} : =\{ \beta \mid \pi_{i\,*} \beta =0, i \in [0,r-1] \} \subset \CH_{k}(\PP(\sG)).$$ 
Then the lemma implies that there is a decomposition
	$$\CH_{k}(\PP(\sG)) =\left ( \bigoplus_{i=0}^{r-1} \pi_i^* \CH_{k-{r-1}+i}(X) \right)\oplus  \CH_{k}(\PP(\sG))_{\rm tor.}.$$

\begin{proof} For simplicity, we may assume $\sG = {\rm Coker}(\sF \xrightarrow{\sigma} \sE)$, then $c(\sG) = c(\sE) / c(\sF)$ and $\iota \colon \PP(\sG) \hookrightarrow \PP(\sE)$ is given by a regular section of the vector bundle $\sF^\vee \otimes \sO_{\PP(\sE)}(1)$. The for any $a \in [0,r-1]$,
	\begin{align*}
	& \pi_{a\,*} (\sum_{i=0}^{r-1} \pi_i^* \,\alpha_i) = \pi_{a\,*} (\sum_{i=0}^{r-1} \zeta^i \cdot \iota^*q^* \alpha_i) \\
	& =\sum_{j=0}^{r-1-a} (-1)^j c_{j}(\sG) \cap q_* \iota_* (\zeta^{r-1-a-j} \cdot \sum_{i=0}^{r-1} \zeta^i \cdot \iota^*q^* \alpha_i ) \\
	& = \sum_{j=0}^{r-1-a} (-1)^j c_{j}(\sG) \cap q_* (\sum_{i=0}^{r-1}  c_m(\sF^\vee(1)) \zeta^{r-1-a-j+i}c\cdot q^* \alpha_i ) \\
	& = \sum_{j=0}^{r-1-a}\sum_{i=0}^{r-1} \sum_{\nu=0}^m (-1)^{j+\nu} c_{j}(\sG)c_{\nu}(\sF) \cdot q_* ( \zeta^{n-1-a-j+i-\nu} \cdot q^* \alpha_i ).
	\end{align*}
Now set $\mu: = \nu+ j$, notice that above terms which survives under $q_*$ have the indices range $0 \le \mu \le i-1$, $i=\mu+a \ge a$ and $0 \le \nu \le \mu$, therefore:
	\begin{align*}
	&  \pi_{a\,*} (\sum_{i=0}^{r-1} \pi_i^* \,\alpha_i)  = \sum_{i=a}^{r-1} \sum_{\mu=0}^{i-a} \sum_{\nu=0}^{\mu} (-1)^{\mu} c_{j}(\sG)c_{\nu}(\sF) \cdot q_* ( \zeta^{n-1+(i-a) -\mu} \cdot q^* \alpha_i )  \\
	& =\sum_{i=a}^{r-1} \sum_{\mu=0}^{i-a} (-1)^{\mu}c_{\mu}(\sE) \cdot q_* ( \zeta^{n-1+(i-a) -\mu} \cdot q^* \alpha_i ) \\
	& = \sum_{i=a}^{r-1} \sum_{\mu=0}^{i-a} (-1)^{\mu}c_{\mu}(\sE) (-1)^{i-a-\mu}s_{i-a-\mu}(\sE) \cap \alpha_i = \alpha_a.
	\end{align*}
Hence \eqref{lem:pi_proj-1} follows. In general, it suffices to notice that the maps $\pi_{i\,*} \,\pi_i^*$ and $ \pi_{i\,*} \,\pi_j^*$ are globally defined and their values do not depend on local presentations.

The statement \eqref{lem:pi_proj-2} follows directly from expressing $\iota_* (\sum_{i=0}^{r-1} \zeta^i \cdot \iota^*q^* \alpha_i ) = c_{m}(\sF^\vee(1)) \cap (\sum_{i=0}^{r-1} \zeta^i \cdot q^* \alpha_i )$ in terms of the basis $\{\zeta^i\}$ of $\CH(\PP(\sE))$. Notice that one can also show \eqref{lem:pi_proj-2} first, and then \eqref{lem:pi_proj-1} follows easily from \eqref{lem:pi_proj-2}.
\end{proof}

For simplicity of notations, from now on we denote:
	$$\sK : = \sExt^1(\sG,\sO_X), \qquad \pi': \PP(\sK) = \PP( \sExt^1(\sG,\sO_X)) \to X.$$
Therefore we have a fibered diagram:
	\begin{equation}\label{diagram:Gamma}
		\begin{tikzcd}[row sep= 2.6 em, column sep = 2.6 em]
	\Gamma: =\PP(\sG) \times_X \PP(\sK)  \ar{d}[swap]{r_-} \ar{r}{r_+} & \PP(\sG) \ar{d}{\pi} 
	\\
	\PP(\sK) \ar{r}{\pi'}         &X
		\end{tikzcd}	
		\end{equation}	 		
\begin{lemma} \label{lem:proj.general} Assume $\PP(\sG)$, $\PP(\sK)$ and $\Gamma$ have expected dimensions (see Remark \ref{remark:exp:dim}  \ref{remark:exp:dim-1}).
	\begin{enumerate}
		\item  \label{lem:proj.general-1}
		The sheaf $\pi'^*\sG$ has homological dimension $\le 1$ and  rank $r+1$ over $\PP(\sK)$, and 
			$$\Gamma = \PP_{\PP(\sK)}(\pi'^*\sG) = \PP_{\PP(\sG)}(\pi^* \sK).$$
		\item  \label{lem:proj.general-2}
		 The excess bundle for the diagram (\ref{diagram:Gamma}) is $\sO(1,1) := \sO_{\PP(\sG)}(1)\otimes \sO_{\PP(\sK)}(1)$. Hence 
			$$\pi^* \pi'_* (\blank) = r_{+\,*} (c_1(\sO(1,1)) \cap r_-^* (\blank)).$$
		\item  \label{lem:proj.general-3}
		The following holds:
			$$\pi_{i\,*} \pi_j^{*} = \delta_{i,j} \Id_{\CH(X)}, \quad \Gamma_* \pi_i^{*} =   \pi_{i\,*} \Gamma^* = 0, \quad \text{for all $i,j \in [0,r-1]$}.$$
	\end{enumerate}
\end{lemma}
	\begin{proof} It suffices to prove in a local situation, i.e. $0 \to \sF \xrightarrow{\sigma} \sE \to \sG \to 0$ for vector bundles $\sF$ and $\sE$ of rank $m$ and $n$. Dually we have 
		$\sE^\vee \xrightarrow{\sigma^\vee} \sF^\vee \to \sK \to 0.$

For \eqref{lem:proj.general-1}, notice that over $\PP(\sK) \subset \PP(\sF^\vee)$, the composition of the map $\sO_{\PP(\sF^\vee)}(-1) \to \pi'^* \sF \xrightarrow{\sigma} \pi'^* \sE$ is zero, hence $\sigma$ factorises through a map of vector bundles $\overline{\sigma} \colon \shT_{\PP(\sF^\vee)/X}(-1)  \to  \pi'^* \sE$. For the reason of ranks, it easy to see the following sequence is exact:
		$$0 \to \shT_{\PP(\sF^\vee)/X}(-1) \xrightarrow{\overline{\sigma}}  \pi'^* \sE \to \pi'^*\sG \to 0.$$
Therefore $\pi'^*\sG$ has homological dimension $\le 1$, and $\PP(\pi'^*\sG) \subset \PP(\pi'^* \sE) = \PP(\sK) \times_X \PP(\sE)$ is the zero scheme of a regular section of the vector bundle $\Omega_{\PP(\sF^\vee)}(1)\boxtimes \sO_{\PP(\sE)}(1)$. The last equality follows directly from commutativity of projectivization and fiber products.

	For \eqref{lem:proj.general-2}, consider the following factorisation of the (transpose of) the diagram (\ref{diagram:Gamma}):
			\begin{equation*}
		\begin{tikzcd}[row sep= 2.6 em, column sep = 2.6 em]
	\Gamma   \ar[hook]{r}[swap]{\iota'} \ar{d}{r_+}  \ar[bend left = 15]{rr}{r_-} & \PP(\sK) \times_X \PP(\sE) = \PP_{\PP(\sK)}(\sE)  \ar{d}{\pi' \times \Id} \ar{r}[swap]{q'} & \PP(\sK) \ar{d}{\pi'} \\ 
	 \PP(\sG) \ar[hook]{r}{\iota}   \ar[bend right = 15]{rr}[swap]{\pi}  & \PP(\sE) \equiv \PP_X(\sE) \ar{r}{q}    &X.
		\end{tikzcd}	
		\end{equation*}	 
(Here for simplicity we use $q$ to denote both projections of projectivization of $\sE$.) The normal bundles are $\sN_{\iota} = \sF^\vee \boxtimes \sO_{\PP(\sE)}(1)$, and $\sN_{\iota'} =\Omega_{\PP(\sF^\vee)}(1) \boxtimes \sO_{\PP(\sE)}(1)$. Since the right square of the diagram is a smooth and flat, the excess bundle is given by $r_+^* \sN_{\iota} / \sN_{\iota} = \sO(1,1)$.

For \eqref{lem:proj.general-3}, the first equality is Lemma \ref{lem:pi_proj}. For any $\gamma \in \CH(\PP(\sK))$, for $i \in [0,r-1]$, 
	\begin{align*}
		\pi_{i\,*} \Gamma^*\, \gamma &= \sum_{j=0}^{r-1-i} s_j(\sF^\vee) \cdot q_{m+i+j \, *}(\iota_* r_{+\,*} r_{-}^* \gamma ) \\
		&=\sum_{j=0}^{r-1-i} s_j(\sF^\vee) \cdot  q_{m+i+j \, *}((\pi'\times\Id)_* \iota'_* \iota'^*(\gamma \boxtimes 1))\\
		&= \sum_{j=0}^{r-1-i} s_j(\sF^\vee) \cdot \pi'_* \, q_{m+i+j \, *}(\iota'_* \iota'^*\,q^* \gamma )   \\
		&= \sum_{j=0}^{r-1-i} s_j(\sF^\vee) \cdot  \pi'_* \, q_{m+i+j \, *}(c_{m-1}(\Omega_{\PP(\sF^\vee)}(1)\boxtimes \sO_{\PP(\sE)}(1)) \cap \,q^* \gamma ) \\
		&= \sum_{j=0}^{r-1-i} s_j(\sF^\vee) \cdot  \pi'_* \, q_{m+i+j \, *}((\zeta^{m-1} + \text{lower order terms}) \cdot \,q^* \gamma ) \\
		&=0.
	\end{align*}
(The last equality holds since $q_{m+i+j \, *}$ has index range $m+i+j \ge m$.)	Similarly for any $\alpha \in \CH(X)$ and $i \in [0,r-1]$,
	\begin{align*}
	&\Gamma_* \pi_i^* \alpha = r_{-*} r_+^* (\zeta^i \cdot \pi^* \alpha) = r_{-*} \iota'^* (\zeta^ i \cdot (\pi' \times \Id)^* q^* \alpha ) = q_*\,\iota'_* \, \iota'^*(\zeta^i \cdot (\pi'^* \alpha \boxtimes 1)) \\
	&= q_* ((\zeta^{m-1+i} + \text{lower order terms}) \cdot q^* \, \pi'^* \alpha ) = 0,
	\end{align*}
since $q$ is the projection of a $\PP^{n-1}$-bundle, and $m-1+i \le m+r-2 \le n-2$.
\end{proof}

\subsection{First approach} In this approach, we use Chow's moving lemma, hence need $\PP(\sG)$ and $\PP(\sK)$ to be nonsingular and quasi-projective. The idea is: over the open part of the first degeneracy locus, the theorem is almost the case of Cayley's trick. Then the ``error" terms over higher degeneracy loci can be estimated by dimension counting. A similar strategy was used by Fu--Wang to show invariance of Chow groups under stratified Mukai flops \cite{FW08}. 

We first need the following variant of Cayley's trick:
	\begin{lemma}[Variant of Cayley's trick]\label{lem:variant:Cayley} Assume $\sG$ is a coherent sheaf of homological dimension $1$ over a variety $X$, and let $i\colon Z \hookrightarrow X$ be a locally complete intersection subscheme of codimension $r+1$, such that $\sG$ has constant rank $r$ over $X \backslash Z$, and constant rank $r+1$ over $Z$. Denote $\Gamma : =  \PP_Z(i^* \sG) = \PP(\sG) \times_X Z$, and denote $\Gamma_* \colon \CH(\PP(\sG)) \to \CH(Z)$ and $\Gamma^* \colon \CH(Z) \to \CH(\PP(\sG))$ the maps induced by $[\Gamma]$ and $[\Gamma]^t$. Then the following holds:
		\begin{align*}\Gamma_* \, \Gamma^* = (-1)^r \Id_{\CH(Z)}, \quad \pi_{i\,*} \pi_j^{*} = \delta_{i,j} \Id_{\CH(X)}, \quad \Gamma_* \pi_i^{*} =   \pi_{i\,*} \Gamma^* = 0,
		\end{align*}
for any $i,j \in [0,r-1]$. Furthermore, the following decomposition of identity holds:
 	$$\Id_{\CH(\PP(\sG))} =  \Gamma^*\,\Gamma_* + \sum_{i=0}^{r-1} \pi_i^* \pi_{i\,*}.$$
	\end{lemma}
\begin{proof} It suffices to notice that the argument of Theorem \ref{thm:Cayley} for these statements only depends on the properties of the normal bundles, hence still works here. More precisely, we may assume $\sG = {\rm Coker}(\sF \xrightarrow{\sigma} \sE)$ for simplicity, then over $Z$ there exits a line bundle $L$ such that there is an exact sequence of vector bundles:
		$$0 \to L \to \sF|_Z \to \sE|_Z \to i^* \sG \to 0.$$
	Also we have a similar picture as  Cayley's trick (\ref{diagram:Cayley}):
	\begin{equation*}
	\begin{tikzcd}[row sep= 0.5 em, column sep = 3 em]
	& \PP(\sE|_Z) \ar[hook]{rr}   \ar{lddd} & & \PP(\sE) \ar{lddd} \\ 
	\Gamma = \PP(i^* \sG)   \ar{dd}[swap]{p} \ar[hook]{ru} \ar[hook, crossing over]{rr}{j} & & \PP(\sG) \ar{dd}{\pi} \ar[hook]{ru}{\iota}	\\ 
	 ~ & ~ & ~ \\
	Z \ar[hook]{rr}{i}         & & X
	\end{tikzcd}
	\end{equation*}
	%
	%
Denote $G_Z: = i^* \sG$, which is a vector bundle on $Z$. Then it is easy to compute the normal bundles are $\sN_i = L^\vee \otimes G_Z$, $\sN_j = L^\vee \otimes \Omega_{\PP(G_Z)/Z}^1(1)$, and excess bundle for the left square is $\sV = p^* N_i/N_j \simeq L^\vee \otimes \sO_{\PP(G_Z)}(1)$ (see Lemma \ref{lem:stratum} for the more general situation). Therefore $\Gamma_* \Gamma^* = p_* (c_{r}(L^\vee \otimes \Omega_{\PP(G_Z)/Z}^1(1)) \cap p^*(\blank)) = p_*((-1)^r\zeta^{r}  + l.o.t.) \cap  p^*(\blank)) = (-1)^r \Id$, and the rest of the orthogonal relations follow from Lemma \ref{lem:proj.general}. Finally, for the last identity, it suffices to show the subjectivity of $\Gamma^* + \sum \pi_i^*$. For any $\beta \in \CH(\PP(\sG)$, $\beta' = \beta - \sum \pi_i^* \pi_{i\,*} \beta$ is supported on $\PP_Z(G_Z)$, hence can always be expressed in the form $\beta' = j_*(p^* \beta_0 + (\zeta-c_1(L)) \sum_{i=0}^{r-1} \zeta^{i} \cap p^* \beta_{i+1})$. Therefore $\beta' = \Gamma^* \beta_0 +  \sum_{i=0}^{r-1} \pi_i^* i_*  \beta_{i+1}'$, and hence we are done.
\end{proof}

\begin{remark} If we modify the map $f$ in Theorem \ref{thm:Cayley} by $f: \oplus_{i=0}^{r-2}  \gamma_i \mapsto (- \oplus_{i=0}^{r-2} i_*  \gamma_i , ~ (\zeta-c_1(L)) \sum_{i=0}^{r-2} \zeta^{i} \cdot p^* \gamma_i)$, then the sequence of Theorem \ref{thm:Cayley} is still exact. In fact, if we denote $p_{i*}$ the projectors with respect to $\sO_{\PP(\sE)}(1)$, then $p_i^*\circ i^* = j^* \circ \pi_i^*$ holds. Although for $i \in [0,r-1]$, $i_* \circ p_{i+1 \,*}$ and $\pi_{i*} \circ j_*$ are no longer the same, due to the additional factor $c_1(L)$, but they differ by an invertible upper triangular change of basis as Remark. \ref{rmk:basis}. Hence in Proposition \ref{prop:caley:key-lemma}, except from the key formula (a) now becomes $\pi_i^* \circ i_* = j_* ((\zeta - c_1(L) \cap p^*(\blank))$, the  rest still holds. Similar for other the statements.
\end{remark}

Denote $X_i:= X^{\ge r+i+1}(\sG)$ for $i \ge -1$, then there is a stratification $\ldots \subset X_{i+1} \subset X_{i} \subset \ldots \subset X_{1} \subset X_{0} \subset X_{-1} = X$. This induces the corresponding stratifications $\PP(\sG)_i:  = \pi^{-1} X_i$, $\PP(\sK)_i:  = \pi'^{-1} X_i$, and $\Gamma_i : = r_{+}^{-1} \pi^{-1} X_i =  r_{-}^{-1} \pi'^{-1} X_i $. Notice that $\PP(\sG)_{-1} = \PP(\sG)$, but 
	$$\ldots \subset \PP(\sK)_{1} \subset \PP(\sK)_{0} = \PP(\sK)_{-1} = \PP(\sK) \quad \text{and} \quad \ldots \subset \Gamma_{1} \subset \Gamma_{0} =  \Gamma_{-1} = \Gamma,$$  
since $\PP(\sK)$ is supported on $X_0$. Over each stratum $X_{i} \backslash X_{i+1}$, $i\ge0$, diagram (\ref{diagram:Gamma}) is:
\begin{equation*}
		\begin{tikzcd}[row sep= 2.6 em, column sep = 2.6 em]
	\Gamma_{i} \backslash \Gamma_{i+1}  \ar{d}[swap]{\PP^{r+i}\text{-bundle}} \ar{r}{\PP^i\text{-bundle}} & \PP(\sG)_{i} \backslash\PP(\sG)_{i+1}   \ar{d}{\PP^{r+i}\text{-bundle}} 
	\\
	\PP(\sK)_{i} \backslash \PP(\sK)_{i+1}   \ar{r}{\PP^i\text{-bundle}}         &X_{i} \backslash X_{i+1} 
		\end{tikzcd}		
\end{equation*}
The codimension condition (\ref{eqn:weak.dim}) translates into $\dim X_0 = \dim X - (r+1)$, and ${\rm codim}(X_i \subset X_0) \ge 2i+1$. From above diagram, this implies that: for any $i \ge 1$,
	$${\rm codim}(\PP(\sK)_i \subset \PP(\sK)) \ge i+1 \quad \text{and} \quad {\rm codim}(\PP(\sG)_i \subset \PP(\sG)) \ge r+i+1.$$

\begin{lemma} \label{lem:injectivity} If $\PP(\sK) = \PP( \sExt^1(\sG,\sO_X))$ is nonsingular and quasi-projective, and the dimension condition (\ref{eqn:weak.dim}) of (A) holds, then the following holds:
	$$\Gamma_* \, \Gamma^* = (-1)^r \Id_{\PP(\sK)}.$$
\end{lemma}

\begin{proof} The following arguments follow the strategy of Fu--Wang \cite{FW08} for stratified Mukai flops, which is itself a generalization of \cite{LLW}'s treatment for standard flops and flips (see also \S \ref{sec:flips}). For any class $[W] \in \CH_k(\PP(\sK))$, by Chow's moving lemma, we may assume $W$ intersects transversely with $\sum_{i \ge 1} \PP(\sK)_i$. 

First, notice that over the open subset $\mathring{X}:=X\backslash X_{1}$, $\mathring{\PP}(\sK) : = \PP(\sK)_{0} \backslash \PP(\sK)_{1} \simeq \mathring{Z}: = X_{0} \backslash X_{1} \xrightarrow{i} \mathring{X}$ is an inclusion of codimension $r+1$, and $\sG$ has constant rank $r$ over $\mathring{X} \backslash \mathring{Z}$, has constant rank $r+1$ over $\mathring{Z}$, and $\mathring{\Gamma} \simeq \PP(i^*\sG) \subset \mathring{\PP}(\sG) : = \PP(\sG)_{0} \backslash \PP(\sG)_{1}$. 
Therefore we are in the situation of variant of Cayley's trick Lemma \ref{lem:variant:Cayley}. Therefore by Lemma \ref{lem:variant:Cayley}, if we set $\mathring{W} = W \cap \mathring{X}$, then the cycle $r_+^* r_{+\,*} r_{-}^* [\mathring{W}]$ is represented by a $k$-cycle $\mathring{\widetilde{W}}$ which maps generically one to one to a $k$-cycle which is rationally equivalent to $\mathring{W}$, and that $r_*[\mathring{\widetilde{W}}] = \Gamma_* \Gamma^* [\mathring{W}] =(-1)^r [\mathring{W}]$.

Now back to the whole space, if we let $\widetilde{W}$ be the closure of $\mathring{\widetilde{W}}$ in $\Gamma$. Then
	$$r_+^* r_{+\,*} r_{-}^* [W] = [\widetilde{W}] + \sum_{C} a_C [F_{C}],$$ 
where $\widetilde{W}$ is the $k$-dimensional cycle as above, mapping generically one to one to a $k$-cycle that is rationally equivalent to $(-1)^rW$, $a_C \in \ZZ$, and $F_{C}$ are $k$-dimensional irreducible schemes supported over $\pi'(W \cap \sum_{i \ge 1} \PP(\sK)_i)$. (More precisely, let $C'$ be irreducible component of $\pi^{-1} {\pi'} (W \cap \sum_{i \ge 1} \PP(\sK)_i)$, then the fiber $F_C$ runs through the components $\{C= \pi(C') \subset \pi'(W \cap \sum_{i \ge 1} \PP(\sK)_i)\}$; Here different $C'$ may have the same image $C$.) For any $F_C$, take the largest $i$ such that there is a component $D \subset \PP(\sK)_i$ with $B_C :=\pi r_+(F_C)= \pi' r_-(F_C) \subset \pi' (W \cap D)$. For a general $s \in B_C$, the fiber $F_{C,s} \subset \Gamma_{s} \simeq \PP^i_{\kappa(s)} \times_{\kappa(s)} \PP^{r+i}_{\kappa(s)}$ over $s$ has dimension:
	\begin{align*}
	\dim F_{C,s} &\ge \dim F_{C}  - \dim (B_C) \ge \dim F_{C}  - \dim r_-(F_C) \\
	&\ge  \dim F_{C}   - \dim (W \cap D) = k -(k - {\rm codim}(\PP(\sK)_i \subset \PP(\sK)) ) \\
	&= {\rm codim}(\PP(\sK)_i \subset \PP(\sK)) \ge  i+1.
	\end{align*}
But since the general fiber of $\pi'$ over $s$ has dimension $i$, hence $F_{C,s}$ contains positive dimension fibers of $r_-$. Therefore $r_{-\,*} [F_C] = 0$. Hence
	$$\Gamma_* \Gamma^* [W] = r_{-\,*} ([\widetilde{W}] + \sum_{C} a_C [F_{C}]) =  r_{-\,*} [\widetilde{W}] = (-1)^r[W].$$ \end{proof}


\begin{lemma}\label{lem:surjectivity} If $\PP(\sG)$ is nonsingular and quasi-projective, and the dimension condition (\ref{eqn:weak.dim}) of (A) holds, then for every $[V] \in \CH_k(\PP(\sG))_{\rm tor.}$ the following holds: 
	$$\Gamma^* \, \Gamma_* [V]   = (-1)^r [V].$$
\end{lemma}

\begin{proof}
Let $[V] \in \CH_k(\PP(\sG))_{\rm tor.}$,  i.e. $[V] \in \CH_k(\PP(\sG))$ such that $\pi_{i\,*} [V] = 0$ for all $i \in [0,r-1]$. By moving lemma we may assume $V$ intersects transversely with $\sum_{i \ge 1} \PP(\sG)_i$. Similar to the proof of Lemma \ref{lem:injectivity}, by variant of Cayley's trick Lemma \ref{lem:variant:Cayley}, over $\mathring{X}:=X\backslash X_{1}$, $\Gamma^* \Gamma_* [\mathring{V}] = (-1)^r [\mathring{V}]$, where $\mathring{V} = V \cap \mathring{X}$ and $[V_0] \in \PP(\sG)_{\rm tor.}$. Therefore there exists $\mathring{W}$ representing $\Gamma_* [\mathring{V}] \in \CH_{k-r}(\PP(\sK))$ such that $r_{-}^{-1}(\mathring{W})$ is a $k$-dimensional cycle and $r_{+\,*} (r_{-}^{-1}(\mathring{W}))$, though supported on $\PP(\sG)_1$, is rationally equivalent to $(-1)^r\mathring{V}$ in $\PP(\sG)$.

Therefore over the whole space, we have:
	$$r_-^* r_{-\,*} r_{+}^* [V] = [\widetilde{V}] + \sum_{C} a_C [F_{C}],$$ 
where $\widetilde{V}$ is the closure of $r_{-}^{-1}(\mathring{W})$ in $\Gamma$, and hence $r_{+*} \widetilde{V}$ is rationally equivalent to $(-1)^rV$, $a_C \in \ZZ$, and $F_{C}$ are irreducible $k$-dimensional cycles supported over $\pi (V \cap \sum_{i \ge 1} \PP(\sG)_i)$. Similar as before, for any $F_C$, take the largest $i \ge 1$ such that there is a component $D \subset \PP(\sG)_i$ with $B_C :=\pi r_+(F_C)= \pi' r_-(F_C) \subset \pi (V \cap D)$. For a general $s \in B_C$, the fiber $F_{C,s}$ has dimension:
	\begin{align*}
	\dim F_{C,s} &\ge \dim F_{C}  - \dim (B_C) \ge \dim F_{C}  - \dim r_-(F_C) \\
	&\ge  \dim F_{C}   - \dim (V \cap D)  = {\rm codim}(\PP(\sG)_i \subset \PP(\sG)) \ge  r+i+1.
	\end{align*}
Now since the general fiber of $\pi$ over $s$ has dimension $r+i$, hence $F_{C,s}$ contains positive dimension fibers of $r_+$. Therefore $r_{+\,*} [F_C] = 0$, and
	$$\Gamma^* \Gamma_* [V] = r_{+\,*} ([\widetilde{V}] + \sum_{C} a_C [F_{C}]) =  r_{+\,*} [\widetilde{V}] =(-1)^r [V].$$ 
\end{proof}

\begin{proof}[Proof of theorem \ref{thm:main} under condition (A)]
The injectivity of the map (\ref{eqn:main.thm}) follows directly from Lemma\ref{lem:pi_proj} and Lemma \ref{lem:injectivity}; the surjectivity of the map (\ref{eqn:main.thm}) follows from Lemma\ref{lem:pi_proj} and Lemma \ref{lem:surjectivity}. This completes the proof theorem \ref{thm:main} under condition (A).
\end{proof}

\subsection{Second approach} The idea of this second approach is that: if we stratify the space $X$ as before, then over each stratum the theorem reduces to a situation  very similar to standard flips case \S \ref{sec:flips}. Since we will argue over each stratum, we will need all strata to achieve the expected dimensions, but do not require regularity on the total space.

\begin{lemma}  \label{lem:stratum} Let $\sG$ be a coherent sheaf on a Cohen Macaulay scheme $X$ of homological dimension $\le 1$ and rank $r$. For a fixed integer $i \ge 0$, denote $Z= X^{\ge r+i+1}(\sG)$, and assume $X^{\ge r+i+2}(\sG) = \emptyset$. (That is, $Z$ is the bottom degeneracy locus of $\sG$; $\sG$ has constant rank $r+i+1$ over $Z$, and has rank $\le r+i$ over $X \backslash Z$.) Assume furthermore that $Z \subset X$ has the expected codimension $(i+1)(r+i+1)$. Denote $\sK = \sExt^1(\sG, \sO)$, $i \colon Z \hookrightarrow X$ the inclusion. Then $G_Z: =   i^* \sG$, $K_Z : = i^* \sK$ are vector bundles over $Z$ of rank $r+i+1$ and $i+1$ respectively. Consider the following base-change diagram for the fibered product $\Gamma= \PP(\sG) \times_X \PP(\sK)$ along the base-change $Z \hookrightarrow X$, with names of maps as indicated:
	\begin{equation} \label{diagram:strata}
	\begin{tikzcd}[back line/.style={}, row sep=1.2 em, column sep=2.6 em]
& \Gamma_Z = \PP(G_Z) \times_Z \PP(K_Z) \ar[back line]{dd}[near start]{r_{Z-}} \ar[hook]{rr}{\ell} \ar{ld}[swap]{r_{Z+}}
  & & \Gamma \ar{dd}{r_-} \ar{ld}{r_+} \\
 \PP(G_Z) \ar{dd}[swap]{\pi_Z}  
  & &  \PP(\sG)  \ar[crossing over, hookleftarrow, swap]{ll}[near end]{j}\\
& \PP(K_Z) \ar{ld}[swap]{\pi_Z'} \ar[back line, hook]{rr}[near start]{k} 
  & & \PP(\sK) \ar{ld}[near start]{\pi'}  \\
Z  \ar[hook]{rr}{i} & & X  \ar[crossing over, leftarrow]{uu}[near end, swap]{\pi}
	\end{tikzcd}
	\end{equation}
where $\Gamma_Z : = Z \times_X \Gamma = \PP(G_Z) \times_Z \PP(K_Z)$. 

Then the normal bundles of the closed immersions $i, j, k, \ell$ are respectively given by 
	\begin{align*}
	& \sN_i = G_Z \otimes K_Z , \qquad \qquad \qquad \quad \sN_j =  \Omega_{\PP(G_Z)/Z}(1) \boxtimes K_Z, \\
	&\sN_k =G_Z \boxtimes \Omega_{\PP(K_Z)/Z}(1) , \quad \text{and} \quad  \sN_\ell = \Omega_{\PP(G_Z)/Z}(1) \boxtimes \Omega_{\PP(K_Z)/Z}(1). 
	\end{align*}
The excess bundle for the front square is given by $\sV = \sO_{\PP(G_Z)}(1) \boxtimes K_Z$, and the excess bundle for the back square is $\sV' =  \sO_{\PP(G_Z)}(1) \boxtimes \Omega_{\PP(K_Z)/Z}(1)$. Therefore 
	\begin{equation*}
		\pi^* \, i_* (\blank) = j_*(c_{\rm top}(\sV) \cap \pi_Z^*(\blank)), \qquad r_{-}^* k_* (\blank) = \ell_* (c_{\rm top}(\sV') \cap r_{Z-}^*(\blank)).
	\end{equation*}
Similarly the excess bundle for the bottom square is given by $\sW = G_Z \boxtimes \sO_{\PP(K_Z)}(1)$, and for the top square is $\sW' =  \Omega_{\PP(G_Z)/Z}(1)  \boxtimes\sO_{\PP(K_Z)}(1) $. Therefore 
	\begin{equation*}
		\pi'^* \, i_* (\blank) = k_*(c_{\rm top}(\sW) \cap \pi_Z'^*(\blank)), \qquad r_{+}^* j_* (\blank) = \ell_* (c_{\rm top}(\sW') \cap r_{Z+}^*(\blank)).
	\end{equation*}
\end{lemma}
\begin{proof} 
As the statements are local, we may assume $\sG = \Coker (F \xrightarrow{\sigma} E)$, where $E,F$ are vector bundles of rank $n$ and $m$. Then by our assumption on $Z$ and Lemma \ref{lem:normalbundle}, $Z \subset X$ is a closed locally compete intersection subscheme, and $\sN_i = G_Z \otimes K_Z$; Moreover, the image $\im (\sigma|_Z) \subset E|_Z$ is a vector sub-bundle; let us denote it by $B_Z$. Therefore the map $\sigma|_Z$ induces two short exact sequences of vector bundles over $Z$:
	$$0 \to K_Z^\vee \to F|_Z \to B_Z \to 0 \quad \text{and} \quad 0 \to B_Z \to E|_Z \to  G_Z \to 0.$$
	
Next, over $\PP(\sG) \subset \PP(E)$, the composition $\pi^*F \xrightarrow{\pi^*\sigma} \pi^*E \to \sO_{\PP(E)}(1)$ is zero, hence $\pi^*\sigma$ factors through a map between vector bundles $\widetilde{\sigma} \colon \pi^* F \to \Omega_{\PP(E)/X}^1(1)$. As the rank of $\widetilde{\sigma}$ at a point $p \in \PP(\sG)$ agrees with the rank of $\sigma$ at $\pi(p)$, therefore $\widetilde{Z}: =\pi^{-1}(Z) = \PP(G_Z)$ is the bottom degeneracy locus of $\widetilde{\sigma}$. We claim that there is an exact sequence of vector bundles: 
	$$0 \to \pi_Z^* K_Z^\vee \to \pi_Z^* F_Z  \xrightarrow{\widetilde{\sigma} }\Omega_{\PP(E)/X}^1(1)|_{\widetilde{Z}} \to  \Omega_{\PP(G_Z)/Z}^1(1) \to 0.$$
Then by Lemma \ref{lem:normalbundle}, $\sN_j =\Omega^1_{\PP(G_Z)/Z}(1) \boxtimes K_Z$ and
	$\sV = \pi_Z^* \sN_i/ \sN_j \simeq \sO_{\PP(G_Z)}(1) \boxtimes K_Z.$

To prove the claim, it suffices to notice that over $\widetilde{Z}$ there is a commutative diagram:
	\begin{equation*}
	\begin{tikzcd}
		 \pi_Z^* B_Z  \ar[dashed]{r} \ar[equal]{d} & \Omega_{\PP(E)/X}^1(1)|_{\widetilde{Z}} \ar{d} \ar[dashed]{r} & \Omega_{\PP(G_Z)/Z}^1(1) \ar{d}  \\
		 \pi_Z^* B_Z  \ar{r} & \pi_Z^* E_Z \ar{r} \ar{d} & \pi_Z^* G_Z \ar{d} \\
			& \sO_{\PP(E)}(1)|_{\widetilde{Z}} \ar[equal]{r}  & \sO_{\PP(G_Z)}(1) .
	\end{tikzcd}
	\end{equation*}
In the diagram, the three columns and the last two rows are exact, hence the first row is a short exact sequence. Combined with the short exact sequence of vector bundles $0 \to \pi_Z^*(K_Z^\vee) \to \pi_Z^*(F|_Z) \to \pi_Z^*(B_Z) \to 0$, the claim follows. Notice that in above argument we do not use the condition $n \ge m$, hence the same argument works for all the other cases.
\end{proof}

\begin{lemma}[``Virtual" flips]\label{lem:virtual.flips} In the situation of Lemma \ref{lem:stratum}, denote
	\begin{align*}
	& \Psi_*(\blank): = r_{Z-*}(c_{\rm top}(\sW')\cap r_{Z+}^*(\blank))  \colon \CH(\PP(G_Z)) \to \CH(\PP(K_Z)),\\
	& \Psi^*(\blank): = r_{Z+*}(c_{\rm top}(\sV') \cap  r_{Z-}^*(\blank))   \colon \CH(\PP(K_Z)) \to \CH(\PP(G_Z)),
	\end{align*}
and furthermore for any $a \in [0,r-1]$, denote 
	$\pi_{Z,a}^*(\blank) := c_{\rm top}(\sV) \cdot \zeta^a \cap \pi_Z^*(\blank).$
Then
	\begin{enumerate}
	\item \label{lem:virtual.flips-1}
	$\Psi_* \Psi^* = (-1)^r \Id$;
	\item \label{lem:virtual.flips-2}
	Then for any $k \ge 0$, there is an isomorphism of Chow groups:
	\begin{align*}
	 \bigoplus_{a=0}^{r-1} \CH_{k-(r-1)+a}(Z)  \oplus \CH_{k-r}(\PP(K_Z)) \xrightarrow{\sim} \CH_{k}(\PP(G_Z))
	\end{align*}
given by $(\oplus_{a=0}^{r-1} \alpha_a, \gamma) \mapsto  \sum_{a=0}^{r-1} \pi_{Z,a}^* \alpha_a + \Psi^* \gamma$.
	\item \label{lem:virtual.flips-3}
	The following identities hold: for any $a \in [0,r-1]$,
		$$\Gamma^* k_* (\blank) = j_* \Psi^*(\blank), \quad \Gamma_* j_* (\blank) = k_* \Psi_*(\blank), \quad \pi_a^* \, i_* (\blank) = j_* \, \pi_{Z,a}^*(\blank).$$
	\end{enumerate}	
\end{lemma}

\begin{proof} For the first two statements, notice that if we write $F = G_Z^\vee$, $F'=K_Z^\vee$, with rank $n=r+i$ and $m=i$, and $S= Z$, then $P=\PP(G_Z)$, $P'=\PP(K_Z)$, $E = \Gamma_Z$, and we are in a very similar situation as the standard flip case \S \ref{sec:flips}. 
In fact, for \eqref{lem:virtual.flips-1}, using the notation of the proof of Theorem \ref{thm:flip}, then $\Psi_*$ and $\Psi^*$ correspond to the correspondence given by $(-1)^{n}c_{n}(\sV')$ and $(-1)^m c_{m}(\sV)$ respectively (instead of $c_{m}(\sV)$ for $\Phi_*$ and $c_{n}(\sV')$ for $\Phi^*$). However the composition $c_{n}(\sV') * c_{m}(\sV)$ is still computed by the same formula as $c_{m}(\sV) * c_{n}(\sV')$ (with the role of first and third factor of the product $P'\times_S P \times_S P'$ switched), by commutativity of intersection product. Hence $c_{n}(\sV') * c_{m}(\sV) = [\Delta_{P'}]$, and $\Psi_* \Psi^* = (-1)^{m+n} \Id = (-1)^r \Id$.

For \eqref{lem:virtual.flips-2}, the same argument of Lemma \ref{lem:flip:Phi} works. In fact, the image of $\Psi^*$ is the ``sub-$\CH(Z)$-module" generated by $1,\zeta,\ldots,\zeta^i$. Hence up to elements of $\Im \Psi^* = \Span \{ 1, \zeta^1, \ldots, \zeta^i \}$, the 
map 
	$$\pi_{Z,a}^*(\blank) =  \zeta^a \cdot c_{\rm top}(\sV) \cdot \cap \pi_Z^*(\blank) =  \zeta^a \cdot (\zeta^{i+1} + \text{lower order terms}) \cdot \pi_Z^*(\blank)$$
hits each element of the basis $\{\zeta^{i+a+1} \mod \Im \Psi^*\}_{a \in [0,r-1]}$ of the  quotient 
	$\CH(\PP(G_Z))/\Im \Psi^*$. 
Therefore the result follows.
	
For \eqref{lem:virtual.flips-3}, it follows directly from Lemma \ref{lem:stratum} that for any $\gamma \in \CH(\PP(K_Z))$, 
	\begin{align*}
	\Gamma^*k_* \gamma = r_{+*} r_-^*k_* \gamma = r_{+*} \ell_{*} (c_{\rm top}(\sV') \cap r_{Z-}^* \gamma) = j_* r_{Z+\,*}  (c_{\rm top}(\sV') \cap r_{Z-}^* \gamma) = j_* \Psi^* \gamma,
	\end{align*}
and similarly $\Gamma_* j_* = k_* \Psi_*$. Also for any $a \in [0,r-1]$ and $\alpha \in \CH(Z)$, 
	\begin{align*}
	\pi_a^*(i_* \alpha) = \pi^*(\zeta^a \cdot i_* \alpha) =   j_*  (\zeta^a \cdot c_{\rm top}(\sV) \cap \pi_Z^*(\alpha)) = j_* \pi_{Z,a}^*(\alpha). 
	\end{align*}
\end{proof}

\begin{proof}[Proof of theorem \ref{thm:main} under condition (B)]
Stratify the space $X$ by the same way as in the first approach, namely $X_i:= X^{\ge r+i+1}(\sG)$ for $i \ge -1$, and similarly for $\PP(\sG)_i$, $\PP(\sK)_i$ and $\Gamma_i$. For each $i \ge -1$, we will denote the natural inclusions by: $i_{i} \colon X_i \hookrightarrow X$, $j_{i} \colon \PP(\sG)_i \hookrightarrow \PP(\sG)$,  $k_{i} \colon \PP(\sK)_{i} \hookrightarrow \PP(\sK)$ and $\ell_{i} \colon \Gamma_i \hookrightarrow \Gamma$. For $i \ge 0$, we also denote $i_{i,i-1} \colon X_{i} \hookrightarrow X_{i-1}$ the natural inclusion, and $j_{i,i-1}$, $k_{i,i-1}$ and $\ell_{i,i-1}$ are defined similarly. Finally for each pair $(i,j)$ with $j > i \ge -1$, denote by $X_{i\backslash j} : = X_{i} \backslash X_{j}$; $\PP(\sG)_{i\backslash j}$, $\PP(\sK)_{i\backslash j}$ and $\Gamma_{i\backslash j}$ are defined by the same manner. By abuse of notations, the inclusion $i_i \colon X_{i\backslash j} \hookrightarrow X \backslash X_j = X_{-1 \backslash j}$ is also denoted by $i_i$, and similarly for other inclusions.

For any fixed integer $i \ge 0$, if we assume condition (B) of Theorem \ref{thm:main} is satisfied, then $Z: = X_{i\backslash i+1} \subset X \backslash X_{i+1} = X_{-1 \backslash i+1}$ is a locally complete intersection subscheme of codimension $(i+1)(r+i+1)$, and $\sG$ has constant rank $r+i+1$ over $Z$. Therefore the conditions of Lemma \ref{lem:stratum} are satisfied by $Z \subset X \backslash X_{i+1}$ and $\sG$, with $\PP(G_Z) = \PP(\sG)_{i \backslash i+1}$, $\PP(K_Z) = \PP(\sK)_{i \backslash i+1}$ and $\Gamma_Z = \Gamma_{i \backslash i+1}$, also $i = i_i$, $j=j_i$, $k=k_i$ and $\ell = \ell_i$. Hence results of Lemma \ref{lem:virtual.flips} can be applied.

Now our goal is to show the isomorphism of Lemma \ref{lem:virtual.flips} \eqref{lem:virtual.flips-2} over each stratum can indeed be integrated into an isomorphism of the map (\ref{eqn:main.thm}) of Theorem \ref{thm:main}.

\medskip \textit{Surjectivity of the map (\ref{eqn:main.thm}).} For each $i \ge -1$, there is an exact sequence:
	$$
	\begin{tikzcd}
	 \CH(\PP(\sG)_{i \backslash i+1})  \ar{r}{j_{i\,*} }& \CH(\PP(\sG) \backslash \PP(\sG)_{ i+1})    \ar{r} &  \CH(\PP(\sG) \backslash \PP(\sG)_{i})  \ar{r} & 0,
	\end{tikzcd}
	$$
for which if $i = i_{\max}+1$, then the middle term is the whole space, where $i_{\max}$ is the largest number such that $X_{i_{\max}} \ne \emptyset$.  (Since $X$ is locally Noetherian of pure dimension, there exists only finitely many strata and such an $i_{\max}$ always exists.)  Therefore inductively we see $\CH(\PP(\sG))$ is generated by the images of $j_{i \,*} \colon \CH(\PP(\sG)_{i \backslash i+1}) \to \CH(\PP(\sG))$ for all strata $\PP(\sG)_{i \backslash i+1}$, $i \ge -1$, where $i=-1$ corresponds to the open stratum.  

Hence we need only show that the image of the map (\ref{eqn:main.thm}) contains the image of the strata $\CH(\PP(\sG)_{i \backslash i+1})$ in $\CH(\PP(\sG))$ for each $i \ge -1$. The open stratum case $i=-1$ follows from projective bundle formula. For other cases, i.e $i \ge 0$, set $Z: = X_{i\backslash i+1} \subset X \backslash X_{i+1}$ as above, and for simplicity denote $j_{*} :=j_{i \,*} \colon \CH(\PP(\sG)_{i \backslash i+1}) \to \CH(\PP(\sG))$, which agrees with notations of Lemma \ref{lem:stratum} and Lemma \ref{lem:virtual.flips}; Similarly for the maps $i,k,\ell$. Then by Lemma \ref{lem:virtual.flips} \eqref{lem:virtual.flips-2}, any $\alpha \in \CH(\PP(\sG)_{i \backslash i+1}) = \CH(\PP(G_Z))$ can be written as $\alpha = \sum_{a=0}^{r-1} \pi_{Z,a}^* \alpha_a + \Psi^* \gamma$, for certain $\alpha_a \in \CH(Z)$ and $\gamma \in \PP(K_Z) = \PP(\sK)_{i\backslash i+1}$. Therefore by Lemma \ref{lem:virtual.flips}\eqref{lem:virtual.flips-3},
	\begin{align*} j_* (\alpha) = j_*(\sum_{a=0}^{r-1} \pi_{Z,a}^* \alpha_a + \Psi^* \gamma) =  \sum_{a=0}^{r-1} \pi_{a}^* (i_* \alpha_a) + \Gamma^* (k_* \gamma),
	\end{align*}
i.e. the image of $j_*$ is contained in the image of the map (\ref{eqn:main.thm}). Hence we are done.
	
\medskip \textit{Injectivity of the map (\ref{eqn:main.thm}).} This part is a little tricky; The key observation is that the above excision exact sequence becomes a short exact sequence if we take the image of first map. The injectivity of $\pi_a^*$ follows from Lemma \ref{lem:pi_proj}; It remains to show the injectivity of $\Gamma^*$. For each $i \ge -1$, there is a commutative diagram of {\em short exact sequences}:
	\begin{equation*}
	\begin{tikzcd}
	0  \ar{r} & \Im k_{i\,*}  \ar{r} \ar{d}{\Gamma^*|_{\Im k_{i\,*}}}  & \CH(\PP(\sK)_{-1 \backslash i+1})  \ar{r}  \ar{d}{\Gamma^*|_{-1 \backslash i+1}} &  \CH(\PP(\sK)_{-1 \backslash i})  \ar{r} \ar{d}{\Gamma^*|_{-1 \backslash i}} & 0. \\
	0   \ar{r} & \Im j_{i\,*}   \ar{r} & \CH(\PP(\sG)_{-1 \backslash i+1})  \ar{r} &  \CH(\PP(\sG)_{-1 \backslash i})  \ar{r} & 0,
	\end{tikzcd}
	\end{equation*}
where recall the maps $k_{i\,*} $ and $ j_{i\,*}$ are the inclusions to (an open subset of) the whole space:
	$$k_{i\,*} \colon \CH(\PP(\sK)_{i \backslash i+1}) \to \CH(\PP(\sK)_{-1 \backslash i}), \qquad j_{i\,*} \colon \CH(\PP(\sG)_{i \backslash i+1}) \to \CH(\PP(\sG)_{-1 \backslash i+1}).$$
We want to show that for each $i \ge 0$, the map $\Gamma^*|_{\Im k_{i\,*}}$ is injective. Set $Z: = X_{i\backslash i+1} \subset X \backslash X_{i+1}$ as above, then the question reduces to show: in the following commutative diagram 
	$$
	\begin{tikzcd}[row sep= 2.6 em, column sep = 5 em] \CH(\PP(\sK)_{i\backslash i+1}) \ar[hook]{r}{\Psi^*} \ar[two heads]{d}{k_{i\,*}} & \CH(\PP(\sG)_{i\backslash i+1}) \ar[two heads]{d}{j_{i\,*}} \\
	\Im k_{i*}  \ar{r}{\Gamma^*|_{\Im k_{i\,*}}}& \Im j_{i*},
	\end{tikzcd}
	$$
(which is commutative by Lemma \ref{lem:virtual.flips}\eqref{lem:virtual.flips-3}) the injection $\Psi^*$ induces an injection $\Gamma^*$ on the image. In fact, for any $\gamma \in \CH(\PP(\sK)_{i\backslash i+1})$, if $\Gamma^* k_{i\, *} \gamma = j_{i\,*} \Psi^* \gamma = 0$, then by Lemma \ref{lem:virtual.flips} \eqref{lem:virtual.flips-1}, we have $\gamma = (-1)^r \Psi_*  \Psi^* \gamma$. Therefore by Lemma \ref{lem:virtual.flips}\eqref{lem:virtual.flips-3},
	$$k_{i\,*} \, \gamma = (-1)^r k_{i\,*}\, \Psi_* \, \Psi^*\, \gamma = (-1)^r \Gamma_* \,j_{i\,*} \, \Psi^*\, \gamma = 0.$$
Hence $\Gamma^*|_{\Im k_{i\,*}}$ is injective. Now by induction, starting with the case $i=0$, when the injectivity of $\Gamma^*|_{-1 \backslash 1}$ follows from the following commutative diagram
	\begin{equation*}
	\begin{tikzcd}
	0  \ar{r} & \Im k_{0\,*} \ar[equal]{r} \ar[hook]{d}{\Gamma^*|_{\Im k_{i\,*}}}  & \CH(\PP(\sK)_{-1 \backslash 1})  \ar{r}  \ar{d}{\Gamma^*|_{-1 \backslash 1}} & 0  \ar{r} \ar{d}{0} & 0. \\
	0   \ar{r} & \Im j_{0\,*}   \ar[hook]{r} & \CH(\PP(\sG)_{-1 \backslash 1})  \ar{r} &  \CH(\PP(\sG)_{-1 \backslash 0})  \ar{r} & 0,
	\end{tikzcd}
	\end{equation*}
we can inductively show that $\Gamma^*|_{-1 \backslash i}$ is injective for all $i=0,1, 2, \ldots, i_{\max}, i_{\max} +1$, where $i_{\max}$ is the largest number $i_{\max}$ such that $X_{i_{\max}} \ne \emptyset$. Therefore $\Gamma^* = \Gamma^*|_{-1 \backslash i_{\max}+1}$ is injective on the whole space. Notice that, from above argument, we also obtain that $\Gamma_*\, \Gamma^* = (-1)^r \Id$ holds, since it is true on the image of each stratum. Together with Lemma \ref{lem:pi_proj} and Lemma \ref{lem:proj.general}, this completes the proof of Theorem \ref{thm:main}. 
 \end{proof}

\subsection{First examples}
\subsubsection{Universal $\Hom$ spaces} \label{sec:Hom} Let $S$ be a Cohen--Macaulay scheme, and let $V$ and $W$ be two vector bundles over $S$. Without loss of generality, we may assume $\rank W \le \rank V$. Consider the total space of maps between $V$ and $W$:
	$$X = |\sHom_S(W, V)| =  |\sHom_S(V^\vee, W^\vee)|.$$
 Then there are tautological maps over $X$:
 	$$\phi \colon W \otimes \sO_X \to V \otimes \sO_X \quad \text{and} \quad \phi^\vee \colon V^\vee \otimes \sO_X \to W^\vee \otimes \sO_X.$$ 
Let $\sG = {\rm Coker} (\phi)$ and $\sK = \sExt^1(\sG,\sO_X) = {\rm Coker}(\phi^\vee)$. Then it is easy to see that the condition (B) of Theorem \ref{thm:main} is satisfied, and Theorem \ref{thm:main} holds for 
	$$\PP(\sG) =\Tot_{\PP(V)}(W^\vee \otimes_{S} \Omega_{\PP(V)/S}(1)) \quad \text{and} \quad \PP(\sK) =\Tot_{\PP(W^\vee)}(\Omega_{\PP(W^\vee)/S}(1) \otimes_{S} V).$$ 
Notice that any map $\sigma \colon W \to V$ over $S$ determines a section $s_{\sigma} \colon S \to X$, such that $s_\sigma^* \phi = \sigma$, $s_\sigma^* \phi^\vee = \sigma^\vee $. Then $\Coker (\sigma)$ and $\Coker (\sigma^\vee)$ (and their projectivizations) are just the pull-backs of $\sG$ and $\sK$ (and the projectivizations $\PP(\sG)$ and $\PP(\sK)$) along the section map $s_{\sigma}$. 

Similarly, we can consider the projectivization version: 
 	$$Y = \PP_{S,\, \rm sub}(\sHom_S(W, V)) =  \PP_{S,\, \rm sub}(\sHom_S(V^\vee, W^\vee)).$$
Over $Y$ there are tautological maps: 
	$$\psi \colon W \otimes \sO_{Y}(-1) \to V \otimes \sO_Y,  \quad \text{and} \quad  \psi^\vee \colon V^\vee \otimes \sO_Y \to W^\vee \otimes \sO_{Y}(1).$$
Then the condition (B) of Theorem \ref{thm:main} is satisfied for $\sM = {\rm Coker} (\psi)$ and $\sN = \sExt^1(\sM,\sO_X) = {\rm Coker}(\psi^\vee)$, and Theorem \ref{thm:main} holds for 
	$$\PP(\sM) = \PP_{\PP(V),\,\rm sub}(W^\vee \otimes \Omega_{\PP(V)/S}(1)) \quad \text{and} \quad \PP(\sN) =  \PP_{\PP(W^\vee),\,\rm sub}(\Omega_{\PP(W^\vee)/S}(1) \otimes V).$$ 
One may also consider the linear sections of the space $Y$ as in HPD theory \cite{KuzHPD, BBF}.

\subsubsection{Flops and Springer resolutions} \label{sec:springer}
 In the situation of Theorem \ref{thm:main}, if we take $r=0$, then $\PP(\sG)$ and $\PP(\sK) = \PP(\sExt^1(\sG,\sO_X))$ are both Springer type partial desingularizations of the first degeneracy locus $X_{\rm sg}(\sG) = X^{\ge 1}(\sG) \subset X$. They are related by a flop, and $\Gamma = \PP(\sG) \times_{X} \PP(\sK)$ is the graph closure for the rational map $ \PP(\sG) \dasharrow \PP(\sK)$. For simplicity, we assume $X $ is irreducible. Then Theorem \ref{thm:main} states that if either 
 	\begin{enumerate}
	\item[(A)] $\PP(\sG)$, $\PP(\sK)$ are smooth and quasi-projective (hence they are both resolutions of $X_{\rm sg}(\sG)$), $\Gamma = \PP(\sG) \times_{X} \PP(\sK)$ is irreducible and $\dim \Gamma =\dim X -1$; or 
	\item[(B)] $X$ is Cohen--Macaulay and $\codim X^{\ge i}(\sG) =i^2$ for $i \ge 1$.
	\end{enumerate} 
Then the graph closure $\Gamma$ of the flop $\PP(\sG) \dasharrow \PP(\sK)$ induces isomorphisms:
 	$$\Gamma^* \colon \CH(\PP(\sK)) \simeq \CH(\PP(\sG)), \qquad \Gamma_* \colon \CH(\PP(\sG)) \simeq \CH(\PP(\sK)).$$
 
 

\subsubsection{Cohen--Macaulay subschemes of codimension $2$} \label{sec:CM2}
Let $X$ be an irreducible scheme, and $Z \subset X$ a codimension $2$ subscheme whose ideal $\sI_Z$ has homological dimension $\le 1$. This holds in particular for any codimension two Cohen-Macaulay subscheme $Z \subset X$ inside a regular scheme $X$, by the Auslander--Buchsbaum theorem. (In fact, in this case $X$ clearly has the resolution property, there always exist locally free sheaves $\sF$ and $\sE$, and a short exact sequence $0 \to \sF \to \sE \to \sI_Z \to 0$, with $\rank \sF = \rank \sE -1$; and by the Hilbert--Burch theorem, any Cohen--Macaulay codimension $2$ subscheme of $X$ arises in this way.)

Consider the degeneracy $X^{\ge 1+i}(\sI_Z)$ for $i \ge 0$ as before (note $\rank 
\sI_Z = 1$), then $X^{\ge 1+i}(\sI_Z)$ is the loci where the ideal $\sI_Z$ needs no less than $i+1$ generators. It is known (e.g. see \cite{ES}) that if $\codim X^{\ge 1+i}(\sI_Z) \ge i+1$ for $i \ge 1$, then 
	$\pi \colon \PP(\sI_Z) = \Bl_Z X \to X$
is the blowup of $X$ along $Z$ and is irreducible, and
	$\widetilde{Z}: = \PP(\sExt^1(\sI_Z, \sO_X))$
is the Springer type desingularization of $Z$. Notice that if $X$ is Goreinstein, then $\widetilde{Z} \simeq \PP(\sExt^1(\sI_Z, \omega_X)) = \PP(\omega_Z)$, where $\omega_X$ and $\omega_Z$ are the dualizing sheaves.  Theorem \ref{thm:main} states that if whether 
 	\begin{enumerate}
	\item[(A)] $\Bl_Z X$ and $\widetilde{Z}$ are smooth and quasi-projective, $\widetilde{Z}$ maps birational to $Z$ (hence $\widetilde{Z}$ is a resolution of $Z$), and $\codim X^{\ge 1 +i}(\sI_Z) \ge 1+2i$ for $i \ge 1$ (or equivalently $\Gamma := \Bl_Z X \times_{X} \widetilde{Z}$ is irreducible and $\dim \Gamma =\dim X -1$); or 
	\item[(B)] $X$ is Cohen--Macaulay and $\codim X^{\ge 1+i}(\sI_Z) = i(1+i)$ for $i \ge 1$.
	\end{enumerate} 
Then for any $k \ge 0$, there is an isomorphism of Chow groups:
 	$$\Gamma^* \oplus \pi^* \colon \CH_{k-1}(\widetilde{Z}) \oplus \CH_k(X) \xrightarrow{\sim} \CH_k(\Bl_Z X).$$

\section{Applications}
\subsection{Symmetric powers of curves} \label{sec:SymC}
Let $C$ be a smooth projective curve of genus $g \ge 1$ over $\CC$, for $d \in \ZZ$ denote by $C^{(d)}$ the {$d$-th symmetric power of $C$}. Then $C^{(d)}$ is smooth projective of dimension $d$, parametrises effective zero cycles of degree $d$ on $C$. By convention $C^{(0)} = \{ 0 \}$ is the trivial zero cycle; $C^{(d)} = \emptyset$ for $d <0$. There is an {\em Abel--Jacobi map}: 
	$$AJ \colon C^{(d)} \to \Pic^d(C), \qquad AJ \colon D \mapsto \sO(D),$$
 where $\Pic^d(C)$ is the Picard variety of line bundles of degree $d$ on $C$. The fibre of $AJ$ over a point $\sL = \sO(D) \in X=\Pic^{d}(C)$ is the linear system $|\sL| =\PP_{\rm sub}(H^0(C,\sL)) = \PP(H^0(C,\sL)^\vee)$. If $d \ge 2g-1$, by Riemann-Roch $AJ$ is a projective $\PP^{d-g}$-bundle over $\Pic^{d}(C)$, which makes the case $0 \le d \le 2g-2$ most interesting. If $g \le d \le 2g-1$, then $AJ$ is surjective, with generic fiber $\PP^{d-g}$, and the fiber dimension jumps over   $W_{d}^{d-g+i}$ for $i\ge1$, where $W_{d}^{k}$ is the {\em Brill--Noether locus}, defined as:
	$$W_{d}^k := W_{d}^k(C) : = \{\sL \mid \dim H^0(C,\sL) \ge k+1\} \subset \Pic^d(C).$$
If $0 \le d \le g-1$, then $AJ$ maps birationally onto the Brill--Noether loci $W_{d}^{0} \subset \Pic^d(C)$, which has codimension $g-d$, and the dimension jumps over $W_{d}^{i}$ for $i \ge 1$.

The cases $g-1 \le d \le 2g-2$ and $0 \le d \le g-1$ are naturally related by the involution $\sO(D) \mapsto \sO(K-D)$, which induces canonical isomorphism 
$W^k_d  \simeq W^{g-d+k-1}_{2g-2-d}$. 
Following Toda \cite{Tod}, from now on we use the following notation: set an integer $n \ge 0$, and set 
	$$d = g-1 +n, \quad \text{and} \quad d' = 2g-2-d = g-1-n.$$
(We do not restrict ourselves to $n \le g-1$, though this is the most interesting case.) Therefore apart from the usual Abel-Jacobi map, 
we also have its involution version:
	$$AJ^\vee \colon C^{(d')}=C^{(g-1-n)} \to \Pic^{d}(C), \qquad AJ^\vee \colon  D \mapsto \sO(K_C - D).$$	
The fiber of $AJ^\vee$ over a point $\sL \in \Pic^{d}(C)$ is the linear system $|\sL^\vee(K_C)| =\PP_{\rm sub}(H^1(C,\sL)^*) = \PP(H^1(C,\sL))$. Therefore we have the following fibered diagram:
    \begin{equation}\label{diag:AJ}
    	\begin{tikzcd}[row sep=0.6 em, column sep=2.3 em]
    	  	&\Gamma: = C^{(g-1+n)} \times_{\Pic^{g-1+n}(C)} C^{(g-1-n)}  \ar{ld}[swap]{r_+} \ar{rd}{r_-}& \\
	C^{(g-1+n)} \ar{rd}[swap]{AJ}& & C^{(g-1-n)} \ar{ld}{AJ^\vee} \\
		& \Pic^{g-1+n}(C).
	\end{tikzcd}
    \end{equation}

\begin{corollary} \label{cor:SymC} 
For a smooth projective curve $C$ of genus $g \ge 1$, and integers $n \ge 0$, $k \ge 0$, there is an isomorphism of integral Chow groups:
	\begin{align*}
	\CH_{k-n}(C^{(g-1-n)}) \oplus \bigoplus_{i=0}^{n-1} \CH_{k-(n-1)+i}(\Pic^{g-1+n} C)    \xrightarrow{\sim} \CH_{k}(C^{(g-1+n)})
	\end{align*}
given by $(\gamma, \oplus_{i=0}^{n-1} \alpha_i) \mapsto   \beta =\Gamma^* \gamma +  \sum_{i=0}^{n-1} c_1(\sO(1))^i \cap (AJ)^*\alpha_i $, where $\Gamma^* = r_{+\,*}\, r_{-}^{*}$ as usual, and $\sO(1)$ is the line bundle $\sO_{\PP(\sG)}(1)$ under the identification $C^{(g-1+n)} = \PP(\sG)$ below. The same map also induces an isomorphism of Chow motives:
	\begin{align*} 
	[\Gamma]^t \oplus \bigoplus_{i=0}^{n-1} h^i \circ (AJ)^*  \colon  \foh(C^{(g-1-n)})(n) \oplus \left( \bigoplus_{i=0}^{n-1} \foh(\Pic^{g-1+n}(C)) (i) \right) \xrightarrow{\sim} \foh(C^{(g-1+n)}).
	\end{align*}
\end{corollary}

Notice that $C^{(g-1-n)} = \emptyset$ if $n>g-1$, hence the result is most interesting if $0 \le n \le g-1$. To prove the corollary, we show that the above situation fits into the picture of Theorem \ref{eqn:main.thm} and satisfies condition (A) (if $C$ is not hyperelliptic). 

Set $X: = \Pic^{g-1+n}(C)$, and let $D$ be an effective divisor of large degree on $C$. $\forall \sL \in \Pic(X)$, the exact sequence
    $0 \to \sL \to \sL(D) \to \sL(D)|_D \to 0$ induces an exact sequence:
    $$0 \to H^0(C, \sL) \to H^0(C, \sL(D)) \xrightarrow{\mu_{D}} H^0(C, \sL(D)|_D) \to H^1(C, \sL) \to 0.$$
Globalizing (the dual of) above sequence yields the desired picture: let $\sL_{\rm univ}$ be the universal line bundle of degree $g-1+n$ on $C \times X$, and $\pr_C, \pr_X$ be obvious projections, then
	$$\sE := (\pr_{X *} (\pr_{C}^* \sO(D) \otimes \sL_{\rm univ}))^\vee \quad \text{and} \quad \sF := (\pr_{X *} (\pr_{C}^* \sO_D(D) \otimes \sL_{\rm univ}))^\vee$$
are vector bundles on $X$ of ranks $\deg(D) + n$ and $\deg(D)$, with
a short exact sequence
	$$0 \to \sF \xrightarrow{\sigma = \mu_D^\vee} \sE \twoheadrightarrow \sG \to 0,$$
where $\sG := {\rm Coker}(\sigma)$ is the sheafification of  $H^{0}(C,\sL)^\vee$, has homological dimension $\le 1$ and rank $n$, and $\sK: = \sExt^1(\sG,\sO_X) = {\rm Coker}(\sigma^\vee)$ is the sheafification of $H^{1}(C,\sL)$. Therefore
	$$ C^{(d)}\equiv C^{(g-1+n)} \simeq \PP(\sG), \quad \text{and} \quad C^{(d')} \equiv C^{(g-1-n)} \simeq  \PP(\sK).$$

Then the stratification $X_i: = X^{\ge n+i+1}(\sG)$ for $i \ge -1$ of the Theorem \ref{eqn:main.thm} corresponds to Brill--Noether loci as follows (recall $d= g-1+n$, $d'=g-1-n$): 
	$$X_i = W_{d}^{n+i} \simeq W_{d'}^{i}$$
Recall the following facts from \cite{ACGH}: 
\begin{enumerate}
	\item \textit{(Brill--Noether inequality)}  The expected dimension of $W_d^k$ is the {\em Brill--Noether number} $\rho(g,k,d) := g - (k+1)(g-d+k)$. The following holds:
	$ \dim W_d^k  \ge  \rho(g,d,k),$
and $W_{d}^k \ne \emptyset$ if $\rho(g,k,d) \ge 0$, $W_{d}^k$ is connected if $\rho(g,k,d) \ge 1$. 	
	\item \textit{(Clifford's inequality)} For an effective divisor $D$ of degree $d$, $1 \le d \le 2g-1$, then $r(D) : = \dim H^{0}(C,\sO(D)) - 1 \le \frac{1}{2} d$.
	\item \textit{(Martens theorem)}  Assume $g \ge 3$, and $(d,k) \in \{2 \le d \le g-1, 1 \le k \le \frac{d}{2}\} \cup \{g-1 \le d \le 2g-4, d-g +2 \le k \le \frac{d}{2}\}$. If $C$ is not hyperelliptic, then 
	$\dim W_d^k \le d - 2k - 1$.
If $C$ is hyperelliptic, then $\dim W_d^k = d - 2k$.
\end{enumerate}

\begin{proof}[Proof of Corollary \ref{cor:SymC}.]  
\medskip \noindent \textit{Uninteresting cases.}
Note that $\Gamma \ne \emptyset$ if and only if  $0 \le n \le g-1$; The corollary for the cases $n \ge g$ follow from the projective bundle formula Theorem \ref{thm:proj.bundle}. If $n = g-1$, then $AJ^\vee \colon C^{(g-1-n)} \simeq \{ [\omega_C] \}\in \Pic^{2g-2}(C)$, and $\Gamma = \PP_{\rm sub}(H^0(C,\omega_C))\simeq \PP^{g-1} \subset C^{(2g-2)}$, and the fibered diagram is a Cayley's trick diagram with $Z = \{[\omega_C]\}$ a point, then the results follows from Theorem \ref{thm:Cayley}. Hence we need only consider the case $0 \le n \le g-2$ and $g \ge 2$. If $g=2$, then $n=0$, $d=d'=1$, $\Gamma \simeq C$, $\Gamma^* \colon \CH(C) \simeq \CH(C)$ is the isomorphism induced the hyperelliptic involution on $C$. Hence we may assume from now on $g \ge 3$, $0 \le n \le g-2$ and $d =g-1+n \in [g-1, 2g-3]$.

\medskip\noindent  \textit{The case when $g \ge 3$ and $C$ is not hyperelliptic.}
We show the condition (A) is satisfied, i.e.
	$${\rm codim} (W_d^n \subset X) = n+1, \quad {\rm codim} (W_{d}^{n+i} \subset X) \ge n+2i+2 \quad \text{for $i \ge 1$}.$$
The first equality always holds, since $C^{(d')}$ maps birationally onto $W_d^n \simeq W_{d'}^0$. For the second inequality, notice that if $d = 2g-3$, $n=g-2$, then $W_d^{n+i} = \emptyset$ if $i \ge 1$ by Clifford's inequality, since $2n+2i = 2g - 4 +2i > d$ if $i \ge 1$. Hence we may assume $d  \in [g-1, 2g-4]$, and Marten's theorem can be applied. Therefore if $C$ is not hyperelliptic, then for any $i \ge 1$:
	$${\rm codim} (W_{d}^{n+i} \subset X) \ge g - (d - 2(n+i) -1)= g- (g-1+n) + 2(n+i) +1= n + 2i +2.$$
	
\medskip \noindent  \textit{The case $g \ge 3$, and $C$ is hyperelliptic.} Take a disc $D$ in the moduli space $\shM_g$ intersecting transversely to the hyperelliptic locus, with zero point $[C]$, and consider the universal curve $\sC$ over $D$. Then the general fiber of $\sC$ is non-hyperelliptic, and by above estimates condition (A) is satisfied by the family $\sC$ (with relative Hilbert schemes $\Hilb_{g-1 \pm n}(\sC/D)$ of zero dimensional subscheme on the fibres of length $g-1 \pm n$) as well as  the generic fibre $\sC_{\eta}$. Therefore the identities of the maps between Chow groups (e.g. $\Gamma_* \, \Gamma^* = \Id$, decomposition of $\Id = \Gamma^* \Gamma_* + \sum_i \pi_i^* \pi_{i\,*}$, etc) of Theorem \ref{thm:main} for $\Hilb_{g-1\pm n}(\sC/D)$ (or $\sC_{\eta}^{(g-1\pm n)}$) specialize to the same identities for the central fiber $\sC_0 = C$ (see \cite[Ch. 10]{Ful}), hence induces the isomorphism of Corollary \ref{cor:SymC} for the hyperelliptic curve $C$.
 \end{proof}

\begin{remark} \label{rmk:SymC} The isomorphisms of Corollary \ref{cor:SymC} are over the ring $\ZZ$. If working with rational coefficients, as pointed out to us by the referee, one could also deduce the $\QQ$-linear version of Corollary \ref{cor:SymC} from del Ba{\~n}o's works \cite{dB1,dB2}: \cite[Proposition 3.7]{dB2} implies 
	$$\foh_\QQ(C^{(n)}) \simeq \bigoplus_{n_0 + n_1 + n_2 = n} 1^{\otimes n_0} \otimes  \lambda^{n_{1}} \foh_\QQ^{1}(C)  \otimes \LL^{\otimes n_2} =  \bigoplus_{n_0 + n_1 + n_2 = n} (\lambda^{n_{1}} \foh_\QQ^{1}(C))(n_2) ,$$
where $n_0, n_1, n_2$ are non-negative integers, $\foh_\QQ(\blank) = \foh(\blank) \otimes \QQ$, and $\lambda^n$ is the $\lambda$-structure on the $\QQ$-linear pseudoabelian category of effective Chow motives. By \cite{dB1}:
	$$\foh_\QQ({\rm Jac}(C)) =\bigoplus_{k=0}^{2g} \lambda^k \foh_\QQ^1(C).$$
Combining these two formulae, one obtains the desired result for $\foh_\QQ$. 
By using Ba{\~n}o's works \cite{dB1,dB2} as above, \cite[Proposition 1.6]{GL} also independently obtains the isomorphism of Chow motives of Corollary \ref{cor:SymC} with rational coefficients.
\end{remark}

\subsection{Nested Hilbert schemes of surfaces} \label{sec:Hilb} 
Let $S$ be a smooth surface over $\CC$, for $n \ge 0$, denote $\Hilb_n=\Hilb_n(S)$ the Hilbert scheme of $n$-points on $S$, i.e. $\Hilb_n$ parametrizes colength $n$ ideals $I_{n} \subset \sO_S$ (or equivalently, length $n$ zero dimension subschemes $\zeta_n=V({I_n}) \subset S$). Furthermore , define the {\em nested} Hilbert scheme by:
	\begin{align*}
	\Hilb_{n,n+1}  = \{(I_{n+1} \subset I_n)  \mid I_{n}/I_{n+1} \simeq \CC(x),  \text{~for some $x \in S$}\} \subset \Hilb_n \times \Hilb_{n+1}.
	\end{align*}
Then $\Hilb_{n,n+1}$ parametrizes zero-dimensional subschemes $\eta_n = V(I_n) \subset \eta_{n+1} = V(I_{n+1}) \subset S$ of length $n$ and $n+1$ respectively, such that $\eta_{n+1} / \eta_n = \CC(x)$ for some $x \in S$. Similarly, one can consider higher {\em nested} Hilbert scheme:
	\begin{align*}
	\Hilb_{n-1,n,n+1}  = \{I_{n+1} \subset I_{n} \subset I_{n-1} \mid  I_{n}/I_{n+1} \simeq \CC(x),  I_{n}/I_{n-1} \simeq \CC(x), \text{~for some $x \in S$} \}.
	\end{align*}
Let $X = \Hilb_n(S) \times S$, and let $Z_n \subset X$ be the universal subscheme. Then $X$ is smooth, and $Z_n \subset X$ is a Cohen--Macaulay subscheme of codimension $2$.  

The following is summarised from Ellingsrud-- Str{\o}mme \cite{ES}, Negu{\c{t}} \cite{Neg, Neg1} and  Maulik--Negu{\c{t}} \cite[Proposition 6.3 \& 6.8]{MN19}.
\begin{lemma} \label{lem:hilb}
\begin{enumerate}
	\item \label{lem:hilb-1}
$\Hilb_{n,n+1}(S) = \PP(\sI_{Z_n}) = \Bl_{Z_n} (X)$ is smooth of dimension $2n+2$;
	\item \label{lem:hilb-2}
	$ \Hilb_{n-1,n}(S) = \PP(\sExt^1(\sI_{Z_n},\sO_X)) = \PP(\omega_{Z_n})$ is smooth of dimension $2n$;
	\item \label{lem:hilb-3}
	 $\Hilb_{n-1,n,n+1}(S) = \Hilb_{n-1,n}(S) \times_{X} \Hilb_{n,n+1}(S)$ is smooth of dimension $2n+1$.
\end{enumerate}
\end{lemma}

Consider the fibered diagram:
\begin{equation*}\label{diag:Hilb}
    	\begin{tikzcd}[row sep=0.5 em, column sep=3 em]
    	  	&\Gamma_n: = \Hilb_{n-1,n,n+1}(S) \ar{ld}[swap]{r_{-}} \ar{rd}{r_{+}}& \\
	 \Hilb_{n-1,n}(S)   \ar{rd}[swap]{\pi_{-}}& & \Hilb_{n,n+1}(S) \ar{ld}{\pi_{+}} \\
		& X= \Hilb_n(S) \times S.
	\end{tikzcd}
    \end{equation*}   
\begin{corollary}\label{cor:Hilb}
\begin{enumerate}
	\item \label{cor:Hilb-1}
	For any $k \ge 0$, there is an isomorphism of Chow groups:
	\begin{align*}
	\CH_{k-1}(\Hilb_{n-1,n}(S)) \oplus \CH_{k}(\Hilb_n(S) \times S)    \xrightarrow{\sim} \CH_{k}(\Hilb_{n,n+1}(S))
	\end{align*}
given by $(\gamma,  \alpha) \mapsto   \beta =\Gamma_n^* \gamma + \pi_+^*\alpha$. The same map also induces:
	$$
	[\Gamma_n]^t \oplus \pi_+^* \colon \foh(\Hilb_{n-1,n}(S))(1) \oplus \foh(\Hilb_n(S) \times S)    \xrightarrow{\sim} \foh(\Hilb_{n,n+1}(S)).
	$$
 	\item \label{cor:Hilb-2}
	If we consider the ``Zig-Zag shape" diagram of length $d \in [0,n]$:
	\begin{equation*}
	\begin{tikzcd} [row sep=1 em, column sep=1 em]
	  & \Gamma_{n-d+1}  \ar{ld}{r_-} \ar{rd}[swap]{r_+} &  & \cdots \ar{ld}{r_-} \ar{rd}[swap]{r_+} &  & \Gamma_n \ar{ld}{r_-} \ar{rd}[swap]{r_+}  & \\
	\Hilb_{n-d,n-d+1} & & \Hilb_{n-d+1,n-d+2} & & \Hilb_{n-1,n} & & \Hilb_{n,n+1}
	\end{tikzcd}
	\end{equation*}
Then it follows from \eqref{cor:Hilb-2} that following maps are split injective:
	\begin{align*}
	&\Gamma_n^* \, \Gamma_{n-1}^* \, \cdots \Gamma_{n-d+1}^* \colon \CH_{k-d}(\Hilb_{n-d,n-d+1}) \hookrightarrow \CH_k(\Hilb_{n,n+1}),& \text{for} ~d=1,2,\ldots,n;\\
	&\Gamma_n^* \, \Gamma_{n-1}^* \, \cdots \Gamma_{n-d+1}^*  \, \pi_{+}^* \colon \CH_{k-d}( \Hilb_{n-d} \times S)  \hookrightarrow \CH_{k}(\Hilb_{n,n+1}), &\text{for}~d=0,1,\ldots, n,
	\end{align*}
and similarly for Chow motives. (Note that $\Gamma_i^* = r_{+\,*}\, r_{-}^{*}$ as usual.)
	\item \label{cor:Hilb-3}
	The maps $\Gamma_n^* \cdots \Gamma_{n-d+1}^* \pi_{+}^* $ for $d\in [0,n]$ from \eqref{cor:Hilb-2} induce decompositions:
	\begin{align*} \CH_{k}(\Hilb_{n,n+1}(S))  = &  \CH_{k}(\Hilb_n(S) \times S) \oplus \CH_{k-1}(\Hilb_{n-1}(S) \times S) \\
	& \oplus \ldots \oplus  \CH_{k-n+1}(S \times S)\oplus \CH_{k-n}(S), \qquad \forall k \ge 0\\
	\foh(\Hilb_{n,n+1}(S))  = &  \foh(\Hilb_n(S) \times S) \oplus \foh(\Hilb_{n-1}(S) \times S)(1) \\
	& \oplus \ldots \oplus  \foh(S \times S)(n-1) \oplus \foh(S)(n).			\end{align*}
\end{enumerate}
\end{corollary}

The above results are especially interesting in the case when $S$ is a $K3$ surfaces, see \cite{Yin14, MN19, Ob19}. Note that the map $\Gamma_{n} \circ  \Gamma_{n-1}$ is also given by the correspondence $[\Gamma_{n}] * [\Gamma_{n-1}] = [\Hilb_{n-2,n-1,n,n+1}]$. This is because the fiber squares for the fiber product
	$$\Hilb_{n-2,n-1,n,n+1} = \Hilb_{n-2,n-1,n} \times_{ \Hilb_{n-1,n}} \Hilb_{n-1,n,n+1}$$
does not have excess bundle, see \cite[Proposition 2.21]{Neg} (where the result was shown for stable sheaves, but the same proof works for nested Hilbert schemes). 

\begin{proof}[Proof of Corollary \ref{cor:Hilb}] Let $\sG =\sI_Z$. It remains to check condition (A) of Theorem \ref{eqn:main.thm} is satisfied. In fact, notice that $X^{\ge r+i}(\sG) = X^{\ge 1+i}(\sI_{Z_n})$ is the loci where $\sI_{Z_n}$ needs $\ge 1+i$ generators at a pint $(I,x)$, or equivalently 
	$$X^{\ge 1+i}(\sI_{Z_n}) = \{ (I,x) \in \Hilb_n \times S \mid \dim I(x) \ge 1+i \}.$$
It follows from \cite[proof of Proposition 3.2]{ES} that ${\rm codim} (X^{\ge 1+i} \subset X) \ge 2i$ for all $i \ge 1$. On the other hand, we already know that $\Gamma = \Hilb_{n-1,n,n+1}(S)$ is irreducible and of expected dimension by Lemma \ref{lem:hilb} \eqref{lem:hilb-3}, therefore ${\rm codim} (X^{\ge 1+i} \subset X) \ge 1+2i$, and the condition (\ref{eqn:weak.dim}) is satisfied (see Remark \ref{remark:exp:dim} \ref{remark:exp:dim-1}). Finally \eqref{cor:Hilb-2} and \eqref{cor:Hilb-3} follow from \eqref{cor:Hilb-1} .
\end{proof}

\begin{remark} De Cataldo and Migliorini have established the decompositions of the rational Chow groups of $\Hilb_n(S)$ in \cite[Corollary 5.1.5]{dCM1} and $\Hilb_{n,n+1}(S)$ in \cite[Theorem 3.3.1]{dCM2}. In the view of Remark \ref{rmk:SymC}, it is reasonable to expect that one could also deduce the $\QQ$-linear version of the above corollary from the decompositions of $\CH(\Hilb_n(S))_{\QQ}$ and $\CH(\Hilb_{n,n+1}(S))_{\QQ}$ in \cite{dCM1,dCM2}.
\end{remark}

\subsection{Voisin maps.} \label{sec:Voisin} Let $Y \subset \PP_{\CC}^5$ be a cubic fourfold not containing any plane, $F(Y)$ be the Fano variety of lines on $Y$ which is a hyperk{\"a}hler fourfold of type $K3^{[2]}$, $Z(Y)$ be the Lehn--Lehn--Sorger--Van Straten eightfold constructed in \cite{LLSVS17} which is a hyperk{\"a}hler manifold of type $K3^{[4]}$. Voisin construct a rational map $v \colon F(Y) \times F(Y) \dashrightarrow Z(Y)$ of degree six in \cite{Voi16} using geometry of $Y$. In \cite{Chen}, Chen shows that the Voisin map $v$ can be resolved by blowing up the incident locus 
	$$Z = \{ (L_1,L_2) \in F(Y) \times F(Y) \mid L_1 \cap L_2 \ne \emptyset\},$$
using the interpretation \cite{LLMS17} and \cite{LPZ18} of above spaces as moduli of stable objects in the Kuznetsov component $Ku(Y) = \langle \sO_Y, \sO_Y(1), \sO_Y(2) \rangle^{\perp}$ \cite{Kuz10}, with respect to a Bridgeland stability condition $\sigma$ on $Ku(Y)$ constructed in \cite{BLMS17}.

More precisely, following \cite{Chen}, the Voisin map can be viewed as family of extensions $v \colon M_{\sigma}(\lambda_1) \times M_{\sigma}(\lambda_1+\lambda_2) \dashrightarrow M_{\sigma}(2 \lambda_1 + \lambda_2)$ as follows, where $\lambda_1, \lambda_2 \in \shK_{\rm num}(Ku(Y))$ are the natural basis of an $A_2$ lattice \cite{AT14}. By the works of \cite{LLMS17} and \cite{LPZ18}, there are identifications of moduli spaces $M_{\sigma}(\lambda_1)=F(Y)$, $M_{\sigma}(\lambda_1+\lambda_2) = F(Y)$ and $M_{\sigma}(2 \lambda_1 + \lambda_2) = Z(Y)$. Let $\shF$, $\shP$ and $\shE$ be the respective pullbacks of the (quasi-)universal objects on $M_{\sigma} \times Y$ to the moduli spaces $M_{\sigma}(\lambda_1)$, $M_{\sigma}(\lambda_1+\lambda_2)$ and $M_{\sigma}(2 \lambda_1 + \lambda_2)$. Then the Voisin map $v$ sends a pair $(F,P) \in  M_{\sigma}(\lambda_1) \times M_{\sigma}(\lambda_1+\lambda_2)$ which satisfies $\dim \Ext^{1}(F,P) = 1$ to the unique class of nontrivial extension of $F$ by $P$.

If we denote $X = F(Y) \times F(Y)$, and let $\sExt_f^i(\sF,\shP)$ be the sheafification of the group $\Ext^i(F,P)$ for the family $f \colon X \times Y \to X$, the following is proved in Chen's work \cite{Chen}: 
	\begin{enumerate}
		\item $\sExt_f^1(\sF,\shP) =\sI_Z$ (where  $\sI_Z$ is the ideal sheaf of $Z \subset X$, and $Z$ is the incident locus $\{L_1 \cap L_2 \ne \emptyset\}$ defined above), has homological dimension $1$, and $Z \subset X=F(Y) \times F(Y)$ is Cohen--Macaulay of codimension $2$.
		\item The degeneracy loci of $\sExt_f^1(\sF,\shP) =\sI_Z$ over $X$ are given by ($X = X^{\ge 1}(\sI_Z)$, and): 
			\begin{align*}
			Z &=X^{\ge 2}(\sI_Z) = \{(F,P) \mid \dim \Ext^1(F,P) \ge 2\}, \\
		 	\Delta_2 &= X^{\ge 3}(\sI_Z) = \{(F,P) \mid \dim \Ext^1(F,P) \ge 3\},
			\end{align*}
		and $X^{\ge 1+i}(\sI_Z) = \emptyset$ for $i \ge 3$. Here $\Delta_2 \subset F(Y) \times F(Y)$ is the type II locus:
			$\{L \in \Delta \simeq F(Y) \mid \sN_{L/Y} \simeq \sO(1)^{\oplus 2} \oplus \sO(-1)\},$ which is an algebraic surface \cite{Voi16}.
		\item $\sExt^1(\sI_Z,\sO_X) = \sExt^2_f(\shP, \shF) = \omega_Z$, where $\omega_Z$ is the dualizing sheaf of $Z$.
		\item The blowing up $\pi \colon\PP(\sI_Z)  =  \Bl_Z (F(Y) \times F(Y)) \to F(Y) \times F(Y)$ resolves the Voisin map $v$, and if $Y$ is very general (i.e. $\shK_{\rm num}(Ku(Y)) = A_2$), then the resolved Voisin map $\widetilde{v} \colon \Bl_Z (F(Y) \times F(Y)) \to Z(Y)$ is (the projection of) a relative $Quot$-scheme 
			$$\Bl_Z (F(Y) \times F(Y)) = Quot_{Ku(Y)/Z(Y)}(\shE, \lambda_1 + \lambda_2)$$
		 of stable quotients of $\shE$ inside $\shA \subset Ku(Y)$ over $Z(Y)$, where $\shA$ is the heart of $\sigma$.
	\end{enumerate}
Therefore the sheaf $\sI_Z$ satisfies condition (B) of Theorem \ref{eqn:main.thm}. If we  consider
	$$\pi' \colon \widetilde{Z}: = \PP_X(\sExt^2_f(\shP, \shF)) = \PP_Z(\omega_Z) \to X$$
which is a small (partial) resolution of the incidence locus $Z$. Then the projection $ \widetilde{Z} \to Z \subset X$ is an isomorphism over $Z \backslash \Delta_2$, and a $\PP^1$-bundle over $\Delta_2$. Therefore we have a diagram:
	\begin{equation*}\label{diag:Hilb}
    	\begin{tikzcd}[row sep=0.6 em, column sep= 1 em]
    	  	&\Gamma: = \widetilde{Z} \times_X  \Bl_Z (F(Y) \times F(Y)) \ar{ld}[swap]{r_{-}} \ar{rd}{r_{+}}& \\
	\widetilde{Z}= \PP(\sExt^2_f(\shP, \shF))  \ar{rd}[swap]{\pi'}& & \Bl_Z (F(Y) \times F(Y))\ar{ld}{\pi} \\
		& X= F(Y) \times F(Y) 
	\end{tikzcd}
    \end{equation*}   
 
\begin{corollary}\label{cor:Voisin}
For any $k \ge 0$, there is an isomorphism of Chow groups:
	$$\Gamma^* \oplus \pi^* \colon \CH_{k-1}(\widetilde{Z}) \oplus \CH_k(F(Y) \times F(Y)) \xrightarrow{\sim} \CH_k(\Bl_Z ( F(Y) \times F(Y))),$$
where $\Gamma^* = r_{+\,*} \, r_-^*$ as usual. If $\widetilde{Z}$ and $\Bl_Z ( F(Y) \times F(Y))$ are smooth, then the same map induces an isomorphism of Chow motives:
	$$
	[\Gamma]^t \oplus \pi^* \colon \foh(\widetilde{Z})(1) \oplus \foh((F(Y) \times F(Y)) \xrightarrow{\sim} \foh(\Bl_Z ( F(Y) \times F(Y))).
	$$
\end{corollary}
Note that it follows from \cite{JL18} there is a semiorthogonal decomposition:
	$$D(\Bl_Z ( F(Y) \times F(Y))) = \langle D(F(Y) \times F(Y)), D(\widetilde{Z})\rangle, $$
therefore $\widetilde{Z}$ is smooth if and only if $\Bl_Z ( F(Y) \times F(Y))$ is. If this is the case \footnote{In fact, one can show $\widetilde{Z}$ is smooth if $\Delta_2$ is a smooth surface (see e.g. \cite[Lemma B3]{JL18}); Amerik \cite{Am09} shows that $\Delta_2$ is smooth for a general $Y$. On the other hand, if $Y$ is very general, Chen's interpretation \cite{Chen} of $\widetilde{Z}$ as a Quot-scheme over $X$ shows that $\Bl_Z X$ is smooth.}, since the resolution $\widetilde{Z} \to Z$ is $\IH$-small, by taking Betti cohomology realisation of the Chow motives, above map induces isomorphisms of Hodge structures:
	\begin{align*}
	H^n(\Bl_Z ( F(Y) \times F(Y)),\QQ) & \simeq H^n(F(Y) \times F(Y),\QQ) \oplus H^{n-2}(\widetilde{Z},\QQ)  \\
	& \simeq H^n(F(Y) \times F(Y),\QQ) \oplus \IH^{n-2}(Z,\QQ),
	\end{align*}
for any $n \ge 0$, where $\IH$ is the intersection cohomology.

\medskip
\begingroup
   \small
 
\subsection*{Further speculations} 
\begin{enumerate}[leftmargin=*]

\item Let $\sigma \colon \sF \to \sE$ be a map between vector bundles over a Cohen--Macaulay scheme $S$, then there is a section map $s_\sigma \colon S \to |\Hom(\sE,\sF)|$. The condition (B) of Theorem \ref{eqn:main.thm} always holds over $ |\Hom(\sE,\sF)|$. Assume that a suitable relative Chow theory $\CH(X \to S)$ has a well-behaved Tor-independent base-change theory, similar to the base-change theory for derived categories \cite{Kuz11SOD}. Then one can pull back along the section map $\sigma$ and obtain a projectivization formula for $S$ under a much weaker condition. The candidate theories that the author has in mind are Fulton's bivariant intersection theory \cite[Chapter 17]{Ful}, the theory of pure Chow motives over a base $S$ \cite{CH}, and the theory of Higher Chow groups over a base $S$ \cite[Chapter II]{Lev}.

\item This work is inspired by its counterpart in derived categories \cite{JL18}, where the projectivization formula was proved using the techniques developed in \cite{JLX17,RT,KuzHPD}. It is interesting whether or not one can ``decategorify" other interesting semiorthogonal decompositions obtained by these techniques. Examples include various cases of homological projective duality and flops, see \cite{JLX17, JL18, RT,KuzHPD}. Note that usually, results of derived categories only imply {\em ungraded} results for {\em rational} Chow groups and motives; but see \cite{BT16} where essential {\em graded} information of Chow groups is recovered from derived categories.

\item The projectivization formula for derived categories is closely related to the wall--crossing and d-critical flips studied by Toda \cite{Tod2,Tod}. It would be interesting to extend the results of this paper to the cases of Donaldson-Thomas type moduli spaces considered there.

\item The projectivization formula considered in this paper fits into a broad framework of the study of Quot schemes of locally free quotients \cite{J20,J19}. 

\item Since the resolution $\PP(\sExt^1(\sG,\sO_X)) \to X_{\rm sg}(\sG)$ is usually $\IH$-small, it is reasonable to expect one may replace $\CH(\PP(\sExt^1(\sG,\sO_X)))$ by the {\em intersection} Chow group \cite{CH} of $X_{\rm sg}(\sG)$.

\item The projectivization formula of Chow groups should hold for Deligne--Mumford stacks, with $\CH$ replaced by $\CH_{\QQ}$;  It would also be interesting to study the ring structure of $\CH(\PP(\sG))$ in the case when $X$ and $\PP(\sG)$ are smooth.

\end{enumerate}
\endgroup


\begin{thebibliography}{99}


\bibitem[AT14]{AT14}
Addington, Nicolas, and Richard Thomas. 
{ \em Hodge theory and derived categories of cubic fourfolds.}
 Duke Mathematical Journal 163, no. 10 (2014): 1885--1927.

\bibitem[Am09]{Am09}
Amerik, Ekaterina. 
{\em A computation of invariants of a rational self-map.} Annales de la Facult{\'e} des sciences de Toulouse: Math{\'e}matiques, vol. 18, no. 3, pp. 481--493. 2009.

\bibitem[ACGH]{ACGH}
Arbarello, Enrico, Maurizio Cornalba, Phillip Griffiths, and Joseph Harris. {\em Geometry of algebraic curves Vol. I}. Grundlehren der mathematischen Wissenschaften {\bf 267}. Springer--Verlag, New York (1985).

\bibitem[dB1]{dB1}
del Ba{\~n}o, Sebastian.
{\em On motives and moduli spaces of stable vector bundles over a curve.} Thesis.

\bibitem[dB2]{dB2}
del Ba{\~n}o, Sebastian.
{\em On the Chow motive of some moduli spaces.} J. Reine Angew. Math. 532 (2001), 105--132.

\bibitem[BLMS17]{BLMS17}
Bayer, Arend, Mart{\'\i} Lahoz, Emanuele Macr{\`\i}, and Paolo Stellari.
{\em Stability conditions on Kuznetsov components.} Annales Scientifiques de l'{\'E}cole Normale Sup{\'e}rieure. Elsevier, 2021.

\bibitem[BK19]{BK19}
Belmans, Pieter, Andreas Krug.
{\em Derived categories of (nested) Hilbert schemes.} arXiv:1909.04321.

\bibitem[BBF16]{BBF}
Bernardara, Marcello, Michele Bolognesi, and Daniele Faenzi. 
{\em Homological projective duality for determinantal varieties.} Advances in Mathematics 296 (2016): 181--209.

\bibitem[BT16]{BT16} 
Bernardara, Marcello, Goncalo Tabuada. 
{\em Chow groups of intersections of quadrics via homological projective duality and (Jacobians of) non-commutative motives.} Izv. Math. 80, no. 3 (2016): 463.

\bibitem[BO]{BO} Bondal, Alexei, and Dmitri Orlov. {\em Semiorthogonal decomposition for algebraic varieties.} Arxiv: alg-geom/9506012.


\bibitem[dCM1]{dCM1}
de Cataldo, Mark Andrea A., and Luca Migliorini.
{\em The Chow groups and the motive of the Hilbert scheme of points on a surface,} J. of Algebra 251 (2002) no.2, 824--848.


\bibitem[dCM2]{dCM2}
de Cataldo, Mark Andrea A., and Luca Migliorini.
{\em The Chow motive of semismall resolutions.}
Mathematical Research Letters (2004), 11, 151--170.


\bibitem[Chen18]{Chen}
Chen, Huachen. 
{\em The Voisin map via families of extensions.} arXiv:1806.05771 (2018).

\bibitem[CH]{CH}
Corti, Alessio, and Masaki Hanamura.
{\em Motivic decomposition and intersection Chow groups. I} Duke Math. J. 103 (2000), 459--522.


\bibitem[ES]{ES}
Ellingsrud, Geir, and Stein Str{\o}mme. 
{\em An intersection number for the punctual Hilbert scheme of a surface.} Transactions of the American Mathematical Society 350.6 (1998): 2547--2552.

\bibitem[FW08]{FW08}
Fu, Baohua, and Chin-Lung Wang. 
{\em Motivic and quantum invariance under stratified Mukai flops.}
 Journal of Differential Geometry. Vol. 80, no. 2 (2008): 261--280.

\bibitem[FHL21]{FHL}
Fu, Lie, Victoria Hoskins, and Simon Pepin Lehalleur. 
{\em Motives of moduli spaces of rank 3 vector bundles and Higgs bundles on a curve.} arXiv:2102.07546 (2021).

\bibitem[Ful]{Ful}
Fulton, William,
{\em Intersection theory.} Springer Science \& Business Media, 2013.

\bibitem[FP]{FP}
Fulton, William, and Piotr Pragacz. 
{\em Schubert varieties and degeneracy loci.}
 Springer, 2006.

\bibitem[GKZ]{GKZ}
Gelfand, Israel M., Mikhail Kapranov, and Andrei Zelevinsky. 
{\em Discriminants, resultants, and multidimensional determinants.}
 Springer Science \& Business Media, 2008.
 
\bibitem[GG]{GG}
Golubitsky, Martin, and Victor Guillemin. 
{\em Stable mappings and their singularities.} Vol. 14. 
Springer Science \& Business Media, 2012.

\bibitem[GL20]{GL}
G{\'o}mez, Tom{\'a}s, and Kyoung-Seog Lee.
{\em Motivic decompositions of moduli spaces of vector bundles on curves.} arXiv:2007.06067 (2020).

\bibitem[J20]{J20}
Jiang, Qingyuan,
{\em On the Chow theory of Quot schemes of locally free quotients}. arXiv:2010.10734 (2020).

\bibitem[J21]{J19}
Jiang, Qingyuan,
{\em Derived categories of Quot schemes of locally free quotients, I}. \href{https://drive.google.com/file/d/1WTmsiat_GYWswQCJ4KwkCwlfqWkqDhzd/view?usp=sharing}{Preprint} (2021). 

\bibitem[JLX17]{JLX17}
Jiang, Qingyuan, Naichung Leung and Ying Xie, 
{\em Categorical Pl\" ucker Formula and Homological Projective Duality.} J. Eur. Math. Soc. (JEMS) 23 (2021), no. 6, 1859--1898

\bibitem[JL18]{JL18}
Jiang, Qingyuan, and Naichung Leung,
 {\em Derived category of projectivization and flops.} arXiv:\href{https://arxiv.org/abs/1811.12525}{1811.12525} (2018).

\bibitem[KMP]{KMP}
Kahn, Bruno, Jacob P. Murre, and Claudio Pedrini.
{\em On the transcendental part of the motive of a surface.} In Algebraic cycles and motives. Vol. 2, volume 344 of London Math. Soc. Lecture Note Ser., pages 143--202. Cambridge Univ. Press, Cambridge, 2007.
 
\bibitem[KKLL17]{KKLL}
Kiem, Young-Hoon, In-Kyun Kim, Hwayoung Lee, and Kyoung-Seog Lee. 
{\em All complete intersection varieties are Fano visitors.}  Advances in Mathematics 311 (2017): 649--661.

\bibitem[KT]{KT}
Koseki, Naoki, and Yukinobu Toda. 
{\em Derived categories of Thaddeus pair moduli spaces via d-critical flips.} arXiv preprint arXiv:1904.04949 (2019).

\bibitem[Kuz07]{KuzHPD}
Kuznetsov, Alexander. 
{\em Homological projective duality.} Publ. Math. IH{\`E}S. 105, no. 1 (2007): 157-220.

\bibitem[Kuz10]{Kuz10}
Kuznetsov, Alexander. 
{\em Derived categories of cubic fourfolds.} In Cohomological and geometric approaches to rationality problems, pp. 219--243. Springer, 2010.

\bibitem[Kuz11]{Kuz11SOD} 
Kuznetsov, Alexander.
{\em  Base change for semiorthogonal decompositions.} 
Comp.\ Math.\ 147 (2011), 852--876.

\bibitem[LLMS18]{LLMS17}
Lahoz, Mart{\'\i}, Manfred Lehn, Emanuele Macr{\`\i}, and Paolo Stellari. 
{\em Generalized twisted cubics on a cubic fourfold as a moduli space of stable objects.} J. Math. Pures Appl. 114 (2018): 85--117.

\bibitem[Laz]{Laz04}
Lazarsfeld, Robert.
{\em Positivity in algebraic geometry II}. 
Vol. 49. Springer Science  \&  Business, 2004.

\bibitem[LLW10]{LLW}
Lee, Yuan-Pin, Hui-Wen Lin, and Chin-Lung Wang. 
{\em Flops, motives, and invariance of quantum rings.} Annals of mathematics (2010): 243--290.

\bibitem[LLSvS17]{LLSVS17}
Lehn, Christian, Manfred Lehn, Christoph Sorger, and Duco Van Straten. 
{\em Twisted cubics on cubic fourfolds.} J. Reine Angew. Math. 2017, no. 731 (2017): 87--128.

 \bibitem[Lev]{Lev}
 Levine, Marc.
 {\em Mixed motives.} No. 57. American Mathematical Soc., 1998.
 
 \bibitem[LPZ18]{LPZ18}
 Li, Chunyi, Laura Pertusi, and Xiaolei Zhao. 
 {\em Twisted cubics on cubic fourfolds and stability conditions.}
  arXiv preprint arXiv:1802.01134 (2018).
 
 \bibitem[Man]{Manin} 
 Manin, Ju I. 
 {\em Correspondences, motives and monoidal transformations.} Mathematics of the USSR-Sbornik 6, no. 4 (1968): 439.
 
\bibitem[MN19]{MN19}
Maulik, Davesh, and Andrei Negu{\c{t}}. 
{\em Lehn's formula in Chow and Conjectures of Beauville and Voisin.} Journal of the Institute of Mathematics of Jussieu: 2020--08--03, p.1-39.


 \bibitem[Mu1]{Mu1}
Murre, Jacob P.
 {\em On the motive of an algebraic surface.} J. Reine Angew. Math., 409:190--204, 1990.
 
\bibitem[Mu2]{Mu2}
Murre, Jacob P.
{\em On a conjectural filtration on the Chow groups of an algebraic variety. I. The general conjectures and some examples.} Indag. Math. (N.S.), 4(2):177--188, 1993.

\bibitem[Ne17]{Neg1}
Negu{\c{t}}, Andrei.
{\em W-algebras associated to surfaces.} arXiv:1710.03217 (2017).

\bibitem[Ne18]{Neg}
Negu{\c{t}}, Andrei 
{\em Hecke correspondences for smooth moduli spaces of sheaves.} arXiv:1804.03645 (2018).

\bibitem[Ob19]{Ob19}
Oberdieck, Georg. 
{\em A Lie algebra action on the Chow ring of the Hilbert scheme of points of a K3 surface.} arXiv preprint arXiv:1908.08830.

\bibitem[O06]{Orlov06}
Orlov, Dmitri . {\em Triangulated categories of singularities, and equivalences between Landau--Ginzburg models}. Mat. Sb., 197(12):117--132, 2006. arXiv:math/0503630. 12.

\bibitem[Rie14]{Rie}
Rie{\ss}, Ulrike. 
{\em On the Chow ring of birational irreducible symplectic varieties.} Manuscripta Mathematica 145.3--4 (2014): 473--501.

\bibitem[Tha]{Tha}
Thaddeus, M. 
{\em Stable pairs, linear systems and the Verlinde formula.} Invent.
Math. 117 (1994), no. 2, 317--353. MR 1273268.

\bibitem[T15]{RT}
Thomas, Richard P.,
{\em Notes on homological projective duality. }
Algebraic geometry: Salt Lake City 2015, Proc. Sympos. Pure Math., vol. 97, Amer. Math. Soc., Providence, RI, 2018, pp. 585--609.

\bibitem[Tod18a]{Tod2}
Toda, Yukinobu. 
{\em Birational geometry for d-critical loci and wall-crossing in Calabi-Yau 3-folds.} arXiv preprint arXiv:1805.00182 (2018).

\bibitem[Tod18b]{Tod}
Toda, Yukinobu. 
{\em Semiorthogonal decompositions of stable pair moduli spaces via d-critical flips.}  arXiv:1805.00183 (2018).

\bibitem[Vi13]{Vial}
Vial, Charles. 
{\em Algebraic cycles and fibrations.} Doc. Math 18 (2013): 1521--1553.

\bibitem[Vi15]{Vial15}
Vial, Charles. 
{\em Chow-Kuenneth decomposition for 3-and 4-folds fibred by varieties with trivial Chow group of zero-cycles.} J. Algebraic Geom 24, no. 1 (2015): 51--80.

\bibitem[Voi16]{Voi16}
Voisin, Claire. 
{\em Remarks and questions on coisotropic subvarieties and 0-cycles of hyper-K{\"a}hler varieties}. K3 surfaces and their moduli. Springer,  2016. 365--399.

\bibitem[Yin15]{Yin14}
Yin, Qizheng, 
{\em Finite-dimensionality and cycles on powers of K3 surfaces,} Comment. Math. Helv. 90 (2015), no. 2, 503--511.

\end{thebibliography}
\end{document}